\documentclass{article}
\usepackage {amsmath, amsfonts, amssymb, mathrsfs, amsthm,stmaryrd, sectsty,hyphenat,url,accents,graphicx}
\usepackage[usenames]{color}
\usepackage{palatino}
\usepackage[all]{xy}
\usepackage[toc,page]{appendix}
\usepackage{enumitem}

\def\mc#1{\mathcal{#1}}

\def\tx#1{\textrm{#1}}

\def\R{\mathbb{R}}

\def\C{\mathbb{C}}
\def\Q{\mathbb{Q}}

\def\Z{\mathbb{Z}}

\def\lmod{\backslash}

\def\hat{\widehat}

\def\sm{\smallsetminus}

\def\<{\langle}
\def\>{\rangle}
\def\Wedge{\bigwedge}

\newenvironment{mytitle}
{\begin{center}\large\sc}
{\end{center}}

\newtheorem{thm}{Theorem}[subsection]
\newtheorem{lem}[thm]{Lemma}
\newtheorem{pro}[thm]{Proposition}
\newtheorem{cor}[thm]{Corollary}
\newtheorem{fct}[thm]{Fact}
\newtheorem{clm}[thm]{Claim}

\theoremstyle{definition}
\newtheorem{rem}[thm]{Remark}
\newtheorem{dfn}[thm]{Definition}

\sectionfont{\center\sc\normalsize}
\subsectionfont{\bf\normalsize}

\numberwithin{equation}{subsection}

\begin{document}

\begin{mytitle}
The Aubert and Bernstein involutions for disconnected groups
\end{mytitle}

\begin{center}
    Yongshen Cheng and Tasho Kaletha
\end{center}

\begin{abstract}
    We extend to arbitrary disconnected reductive $p$-adic groups the duality on the category of smooth finite-length complex representations defined by Aubert, as well as its cohomological analog defined by Bernstein, and prove various properties of these functors, such as uniqueness, preservation of irreducibility, compatibility with parabolic induction and restriction, and a character formula. As a sample application, we obtain a definition of the Steinberg representation for a disconnected reductive $p$-adic group, compute its character, and discuss its twisted endoscopic properties.
\end{abstract}

\tableofcontents

\section{Introduction}

Let $G$ be a connected reductive group over a non-archimedean local field $F$. Let $\mc{R}(G)$ be the category of finite-length smooth representations with complex coefficients of the group $G(F)$. In \cite[\S3]{Aub95} (see also \cite{Aub96}) a covariant functor $D_A^G : \mc{R}(G) \to \mc{R}(G)$ was defined; it is now known as ``Aubert duality'', or sometimes ``Aubert--Zelevinsky duality''. This functor has had many important applications to the representation theory of $p$-adic groups and to automorphic forms, for example \cite{MW89}, \cite[\S7]{Art13}, \cite{JiangLiu2025}, \cite{CMBO2024}.

A related, but contravariant, functor $D_B^G : \mc{R}(G) \to \mc{R}(G)$ was defined in \cite[\S5.1]{BernsteinPadicGroups}. Bernstein begins by considering the functor 
\[ \tx{RHom}_\mc{H}(-,\mc{H}) : \mc{D}^b(G) \to \mc{D}^b(G) \]
on the bounded derived category $\mc{D}^b(G)$ of smooth representations of $G(F)$, reinterpreted as modules over the ``big Hecke algebra''  $\mc{H}=\mc{C}^\infty_c(G(F))$ of $G$. He shows that on each Bernstein component of $\mc{R}(G) \subset \mc{D}^b(G)$ the output of this functor is concentrated in a single cohomological degree and after shifting by this degree this functor descends to an endofunctor $D_B^G$ of this Bernstein block.

The functors $D_A^G$ and $D_B^G$ satisfy the following properties. Let $P=MU \subset G$ be a parabolic subgroup and let $i_P^G$ and $r_P^G$ denote the normalized parabolic induction and restriction (i.e. Jacquet module) functors. Let $P^-$ be the unique parabolic subgroup that is $M$-opposite to $P$.

\begin{enumerate}
    \item $(D_A^G)^2 \cong \tx{id}$ and $(D_B^G)^2 \cong \tx{id}$.
    \item $D_A^G$ and $D_B^G$ commute with the contragredient functor $(-)^\vee$.
    \item $D_A^G \circ (-)^\vee \cong D_B^G$.
    \item $D_A^G \circ i_P^G  \cong i_{P^-}^G \circ D_A^M$ and $D_B^G \circ i_P^G \cong i_{P^-}^G\circ D_B^M$.
    \item $D_A^M \circ r_P^G \cong r_{P^-}^G \circ D_A^G$ and $D_B^M \circ r_P^G\cong r_{P}^G \circ D_B^G$.
    \item $D_A^G$ and $D_B^G$ preserve irreducibility.
    \item $D_A^G \cong \tx{id}$ and $D_B^G \cong (-)^\vee$  on each supercuspidal Bernstein block.
\end{enumerate}
Properties (1,4,5) are well-known for $D_B^G$. Property (3) is the principal result of \cite{BBK18}. Property (2) was not known until very recently, and is the principal result of \cite{YS26}. Using that, properties (1,4,5) for $D_A^G$ follow from those for $D_B^G$. Property (6) is \cite[Corollary 3.9(b)]{Aub95}, while property (7) is elementary. It turns out that properties (4,5,7) together with the (co- resp. contra-)variance of $D_A^G$ resp $D_B^G$ and a certain compatibility with Frobenius reciprocity uniquely characterize these functors, see Propositions \ref{pro:dguniq} and \ref{pro:dhuniq}.

In this paper we extend the functors $D_A^G$ and $D_B^G$ to the setting where $G$ is replaced by a ``disconnected reductive group'', by which we mean an affine algebraic $F$-group $\tilde G$ whose identity component $G=\tilde G^\circ$ is reductive. We make no assumptions about the component group $\pi_0(\tilde G)$, which is thus allowed to be non-cyclic and even non-abelian, although it is always finite due to the affineness of $\tilde G$. The topological group $\tilde G(F)$ is then again a locally compact, totally disconnected group, in fact a locally pro-$p$ group, where $p$ is the residual characteristic of $F$, so the category $\mc{R}(\tilde G)$ of finite-length smooth representations of $\tilde G(F)$ with complex coefficients is defined in the same way as for the connected group $G$. We define functors 
\[ D_A^{\tilde G},D_B^{\tilde G} : \mc{R}(\tilde G) \to \mc{R}(\tilde G), \]
the first covariant while the second contravariant, and show that they satisfy the analogs of many of the properties of $D_A^G$ and $D_B^G$.

The definition of the endofunctor $D_B^{\tilde G}$ is simpler than that of $D_A^{\tilde G}$ and we begin with it. Let $\mc{\tilde H}=\mc{C}^\infty_c(\tilde G(F))$ be the analog of the big Hecke algebra for the disconnected group $\tilde G$. We consider the endofunctor
\[ \tx{RHom}_\mc{\tilde H}(-,\mc{\tilde H}) : \mc{D}^b(\tilde G) \to \mc{D}^b(\tilde G) \]
on the bounded derived category $\mc{D}^b(\tilde G)$ of smooth representations of $\tilde G(F)$. We show that again on each Bernstein component of $\mc{R}(\tilde G) \subset \mc{D}^b(\tilde G)$ the output of this functor is concentrated in a single cohomological degree (in fact, the same degree as for $G$); hence after degree shift this functor descends to an endofunctor of this Bernstein block, which we denote by $D_B^{\tilde G}$.

The definition of the endofunctor $D_A^{\tilde G}$ is more subtle. It can be given in two equivalent ways, as well as a third way that produces a variant. The endofunctor $D_A^G$ of $\mc{R}(G)$ is defined in \cite[\S3]{Aub95} on each Bernstein block of $\mc{R}(G)$ by means of a certain exact complex
    \[ 0 \to Y_r^G \to Y_{r-1}^G \to \dots \to Y_t^G \to D_A^G \to 0 \]
of endofunctors $Y_i^G : \mc{R}(G) \to \mc{R}(G)$, where $t$ is a parameter depending on the block (the notation in \cite{Aub95} is $\tilde X_i$ in place of our notation $Y_i$). The first construction of $D_A^{\tilde G}$, given in Definition \ref{dfn:aubdisc1}, utilizes the same complex used in the definition of $D_A^G$, but endows it with an action of $\tilde G(F)$. More precisely, we show that for each $i$ the functor $Y_i^G \circ \tx{Res}^{\tilde G}_{G} : \mc{R}(\tilde G) \to \mc{R}(G)$ can be naturally lifted to an endofunctor $Y_i^{\tilde G} : \mc{R}(\tilde G) \to \mc{ R}(\tilde G)$, where $\tx{Res}^{\tilde G}_{G} : \mc{R}(\tilde G) \to \mc{R}(G)$ denotes the restriction functor, and that we again obtain an exact complex 
\[ 0 \to Y_r^{\tilde G} \to Y_{r-1}^{\tilde G} \to \dots \to Y_t^{\tilde G} \to D_A^{\tilde G} \to 0 \]
on each Bernstein block of $\mc{R}(\tilde G)$. This definition has the advantage of making some properties of $D_A^{\tilde G}$ easy to show (in particular compatibility with restriction and preservation of irreducibility), but has the disadvantage that the $\tilde G(F)$-action is a bit cumbersome to define.

A second definition, or rather a reinterpretation of the above definition, given in Corollary \ref{cor:dgdisc-ri},  uses a natural generalization $'Y_i^{\tilde G}$ of $Y_i^G$ to the disconnected setting, based on the concept of parabolic subgroups and parabolic induction and restriction in the disconnected setting. It seems more natural conceptually and the $\tilde G(F)$-action is more transparent. On the flip side, the compatibility with restriction to $G(F)$ and the preservation of irreducibility are less clear from this point of view.

A third definition, which produces a variant $'D_A^{\tilde G}$ of the functor $D_A^{\tilde G}$, is obtained as follows. In their work \cite{BoSe76} on the cohomology of the spherical building of $G$, Borel and Serre introduce a certain cochain complex that a priori appears different from the complex $Y$ mentioned above. In fact, their complex is only defined for $\Z$-modules rather than $G(F)$-representations, but is easily adapted to the latter situation. It can be shown that their complex is isomorphic to the complex $Y$, but the isomorphism involves a subtle sign. However, if one tries to equip the Borel--Serre complex with an action of $\tilde G(F)$ in a natural way, one is led to a construction that produces a different outcome than that for $Y$. It turns out that the two outcomes differ from each other by a twist by a certain interesting sign character on $\tilde G(F)/G(F)$. This produces a variant $'D_A^{\tilde G}$ of the functor $D_A^{\tilde G}$.

The properties of $D_A^{\tilde G}$ and $D_B^{\tilde G}$ that we prove in this paper are as follows. Let $\tilde P=\tilde M U \subset \tilde G$ be a parabolic subgroup of $\tilde G$, a notion reviewed in \S\ref{sub:pardisc}. Let $\tilde P^-$ be the $\tilde M$-opposite parabolic subgroup, as in Definition \ref{dfn:paropdisc}. We show that

\begin{enumerate}
    \item $(D_A^{\tilde G})^2 \cong \tx{id}$ and $(D_B^{\tilde G})^2 \cong \tx{id}$.
    \item $D_A^{\tilde G}$ and $D_B^{\tilde G}$ commute with the contragredient functor $(-)^\vee$.
    \item $D_A^{\tilde G} \circ (-)^\vee \cong D_B^{\tilde G}$.
    \item $D_A^{\tilde G} \circ i_{\tilde P}^{\tilde G} \cong i_{\tilde P^-}^{\tilde G} \circ D_A^{\tilde M}$ and $D_B^{\tilde G} \circ i_{\tilde P}^{\tilde G} \cong i_{\tilde P^-}^{\tilde G} \circ D_B^{\tilde M}$.
    \item $D_A^{\tilde M} \circ r_{\tilde P}^{\tilde G} \cong r_{\tilde P^-}^{\tilde G} \circ D_A^{\tilde G}$ and $D_B^{\tilde M} \circ r_{\tilde P}^{\tilde G} \cong r_{\tilde P}^{\tilde G} \circ D_B^{\tilde G}$.
    \item $D_A^{\tilde G}$ and $D_B^{\tilde G}$ preserve irreducibility.
    \item $D_A^{\tilde G} \cong \tx{id}$ and $D_B^{\tilde G} \cong (-)^\vee$ on each supercuspidal Bernstein block.
    \item $D_A^{\tilde G}$ extends $D_A^G$ in the sense that $\tx{Res}^{\tilde G}_G \circ D_A^{\tilde G}$ is equivalent to $D_A^G \circ \tx{Res}^{\tilde G}_{G}$. The same holds for $D_B^{\tilde G}$.
\end{enumerate}
Properties (6-8) for $D_A^{\tilde G}$ are proved directly. Many of the other properties are first proved for $D_B^{\tilde G}$ by arguments similar to the case of connected groups and then transferred to $D_A^{\tilde G}$ via properties (2,3), following the strategy of \cite{YS26}.  For the key property (3) the argument is different from that of the connected case, and goes as follows. Using the fact that $D^{\tilde G}_B$, $D^{\tilde G}_A$ and $(-)^\vee$, when restricting to $G(F)$, coincide with $D^G_B$, $D^G_A$ and $(-)^\vee$ respectively, we obtain for any $\tilde \pi\in \mc{R}(\tilde G)$ a functorial isomorphism $\tx{Res}^{\tilde G}_G(D_A^{\tilde G}(\tilde \pi^\vee))\to \tx{Res}^{\tilde G}_G(D_B^{\tilde G}(\tilde \pi))$. Using the equivariance properties for $D_A^{\tilde G}$ and $D_B^{\tilde G}$ discussed in \S\ref{sub:dgequiv} we then show that this isomorphism comes from an isomorphism $D_A^{\tilde G}(\tilde \pi^\vee)\to D_B^{\tilde G}(\tilde \pi)$.

As in the connected case, Properties (4,5,7) together with the (co- resp. contra-) variance of $D_A^{\tilde G}$ resp. $D_B^{\tilde G}$ and a certain compatibility with Frobenius reciprocity characterizes these functors uniquely, see Propositions \ref{pro:dguniqdisc} and \ref{pro:dhuniqdisc}.

We now give a brief overview of the contents of the paper. In \S\ref{sec:aub} we review the construction of the functors $D_A^G$ and $D_B^G$ for a connected reductive group $G$. In particular, in \S\ref{sub:defcomp} we review the complex $Y$, in \S\ref{sub:aubconn} we review $D_A^G$, in \S\ref{sub:bercon} we review $D_B^G$, in \S\ref{sub:dgequiv} we establish the equivariance properties of the complex $Y$ in order to prepare for the definition of $D_A^{\tilde G}$. In \S\ref{sub:dguniq} we discuss the uniqueness of $D_A^G$ and $D_B^G$. In \S\ref{sub:bs} the Borel--Serre complex is introduced and compared with the complex of \S\ref{sub:aubconn}.

In \S\ref{sec:aubdisc} we define and study the functors $D_A^{\tilde G}$ and $D_B^{\tilde G}$. Besides proving the above properties, we also derive in \S\ref{sub:cf} a character formula for $D_A^{\tilde G}(\tilde \pi)$ for any $\tilde \pi \in \mc{R}(\tilde G)$. This formula is equivalent to a formula for the action of $D_A^{\tilde G}$ on the Grothendieck group of $\mc{R}(\tilde G)$. We use it to compare $D_A^{\tilde G}$ with an involution defined on the space of twisted characters of certain twisted classical groups defined by Bin Xu in \cite{Xu17Canad}. In \S\ref{sub:sign} we introduce the sign character $\epsilon_{\tilde G}$ and the variant $'D_A^{\tilde G}$ of $D_A^{\tilde G}$ that is based on the Borel--Serre complex and show that $'D_A^{\tilde G}=D_A^{\tilde G} \otimes \epsilon_{\tilde G}$. It turns out that $\epsilon_{\tilde G}$ is relevant in the applications of $D_A^{\tilde G}$ to the study of endoscopy. As already mentioned above, besides the initial definition of $D_A^{\tilde G}$ given in \S\ref{sub:defdgdis}, which uses the complex that defines $D_A^G$, we give a reinterpretation of $D_A^{\tilde G}$ in \S\ref{sub:rp} using a different complex that is built from parabolic induction and restriction relative to disconnected versions of parabolic subgroups. These subgroups are defined and studied in \S\ref{sub:pardisc}. This material is not new and has already appeared in various forms in many different places, for example \cite[\S2.D]{BDR17} and \cite[\S2.3]{DillerySchwein23}, and the ideas go all the way back to \cite{Iwahori65}. Due to the various forms in which this material appears in the literature we find it useful to present it here in the form that we need. Using that, we define generalizations to the disconnected setting of parabolic induction and restriction in \S\ref{sub:parindresdisc}, which are then used in \S\ref{sub:rp}.

In \S\ref{sec:stein} we apply the general discussion of $D_A^{\tilde G}$ to obtain an extension of the Steinberg representation of $G(F)$ to $\tilde G(F)$. In fact, it turns out that $'D_A^{\tilde G}$ is more appropriate for this. We then discuss the character of this representation as well as its properties with respect to twisted endoscopy. This is the material that originally motivated the present paper. It brought forth the sign $\epsilon_{\tilde G}$ which can be seen as a component of the analog for disconnected groups of the Kottwitz sign that was originally defined in \cite{Kot83} for connected groups.

\textbf{Acknowledgements:} This paper was initially conceived in 2021 as a canonical extension of the Steinberg representation to disconnected groups, motivated by the study of twisted versions of the Kottwitz sign \cite{Kot83} that appear in the Langlands conjectures for disconnected groups, as put forth in \cite{KalLLCD}. We thank Bernhard Mühlherr for answering some questions about an analogous construction over finite fields given in \cite{MS95}. The Steinberg construction was extended to the duality functor $D_A^{\tilde G}$ for all disconnected groups $\tilde G$ in 2023, motivated by applications to automorphic forms, and a brief summary of the results in the special case of $\tilde G = G \rtimes \<\theta\>$ for a connected reductive quasi-split $F$-group $G$ and a pinned involution $\theta$ appears in \cite[Appendix B]{AGIKMS}, because they are used in some arguments of that paper. We thank David Hansen for pointing out that, contrary to what was stated in an earlier version of \cite{AGIKMS}, the property $(D_A^G)^2 \cong \tx{id}$ was not known at the time, even for connected $G$. This prompted the work \cite{YS26}, which proves the property $(D_A^G)^2 \cong \tx{id}$ for connected $G$, and motivated us to finally finish this paper.

\section{Review of the Aubert involution for connected groups} \label{sec:aub}

Let $F$ be a non-archimedean local field and let $G$ be a connected reductive $F$-group. Let $\mc{R}(G)$ denote the category of smooth representations of $G(F)$ of finite length. In this section we recall the definition of the Aubert duality functor $D_A^G : \mc{R}(G) \to \mc{R}(G)$ defined in \cite[\S3]{Aub95} (see also \cite{Aub96}).

For $\pi \in \mc{R}(G)$ we denote by $V_\pi$ the underlying vector space, i.e. the image of $\pi$ under the forgetful functor $\mc{R}(G) \to \tx{Vect}_\C$.

\subsection{Adjoint functors}

Let $\mc{C},\mc{D}$ be categories. Recall that two covariant functors $L : \mc{D} \leftrightarrow \mc{C} : R$ are called \emph{adjoint} if there exist mutually inverse bijections
\[ \Phi_{Y,X}  : \tx{Hom}_\mc{C}(LY,X) \leftrightarrow \tx{Hom}_\mc{D}(Y,RX) : \Psi_{Y,X} \]
that are natural in $X \in \mc{C},Y \in \mc{D}$. Equivalently, there exist natural transformations
\[ \eta : \tx{id}_\mc{D} \to RL, \qquad \epsilon  : LR \to \tx{id}_\mc{C}, \]
such that the compositions $(\epsilon L) \circ (L\eta) : L \to L$ and $(R \epsilon)\circ(\eta R) : R \to R$ are the identity natural transformations of $L$ and $R$, respectively.

The equivalence of the two formulations is realized by the following formulas: $\Phi_{Y,X}(f)=R(f) \circ \eta_Y$, $\Psi_{Y,X}(g) = \epsilon_X \circ L(g)$, $\epsilon_X = \Psi_{RX,X}(\tx{id}_{RX})$, $\eta_Y = \Phi_{Y,LY}(\tx{id}_{LY})$.

\subsection{Parabolic induction and restriction}

Let $P \subset G$ be a parabolic subgroup, $U \subset P$ its unipotent radical, and $M=P/U$ its reductive quotient. The modulus character $\delta_P : P(F) \to \R_{>0}$, defined by $\int_{P(F)}f(px)dp = \delta_P(x)\int_{P(F)}f(p)dp$ with respect to any left-invariant Haar measure $dp$ on $P(F)$, is trivial on $U(F)$ and descends to $P(F)/U(F)=M(F)$.

We have the unnormalized parabolic induction functor
\[ I_P^G : \mc{R}(M) \to \mc{R}(G),\qquad \sigma \mapsto \pi, \]
where
\[ V_\pi = \{f : G(F) \to V_\sigma| f(pg) = \sigma(p)f(g) \} \]
and $(\pi(x)f)(g) = f(gx)$ for $x,g \in G(F)$, as well as its normalized analogue
\[ i_P^G(\sigma) = I_P^G(\sigma \otimes \delta_P^{1/2}).\]
We further have the un-normalized parabolic restriction (also called the Jacquet) functor 
\[ R_P^G : \mc{R}(G) \to \mc{R}(M),\qquad \pi \mapsto \sigma, \]
where $V_\sigma = (V_\pi)_{U(F)} = V_\pi/\<\pi(u)v-v|u \in U(F),v \in V_\pi\>$ and $\sigma(m)v=\pi(p)v$ for $v \in V_\pi$ and $p \in P(F)$ with image $m \in M(F)$, and its normalized analogue
\[ r_P^G(\pi) = R_P^G(\pi) \otimes \delta_P^{-1/2}|_{M(F)}. \]
The pairs of functors $(r_P^G,i_P^G)$ and $(R_P^G,I_P^G)$ are adjoint; more precisely, the map
\[ \Psi_{\pi,\sigma} : \tx{Hom}_{G(F)}(\pi,i_P^G(\sigma)) \to \tx{Hom}_{M(F)}(r_P^G(\pi),\sigma),\qquad \alpha \mapsto \tx{ev}_1 \circ \alpha, \]
is an isomorphism functorial in $\pi$ and $\sigma$, where $\tx{ev}_1 : i_P^G(\sigma) \to \sigma$ is the morphism that evaluates elements of $i_P^G(\sigma)$ at $1 \in G(F)$. The inverse $\Phi_{\pi,\sigma}$ of $\Psi_{\pi,\sigma}$ sends $\beta : r_P^G(\pi) \to \sigma$ to $\alpha : \pi \to i_P^G(\sigma)$ given by $\alpha(v)(g)=\beta(gv)$ for $v \in V_\pi$ and $g \in G(F)$. The same holds for the pair $(R_P^G,I_P^G)$.

We have the obvious identity 
\begin{equation} \label{eq:xnorm}
    i_P^G \circ r_P^G = I_P^G \circ R_P^G.
\end{equation}
The two adjunction morphisms
\begin{equation} \label{eq:irpg_adj}
\epsilon_P^G : r_P^G \circ i_P^G \to \tx{id}_{\mc{R}(M)}\qquad \tx{and} \qquad \eta_P^G : \tx{id}_{\mc{R}(G)} \to i_P^G \circ r_P^G
\end{equation}
are given as follows. For $\sigma \in \mc{R}(M)$ the morphism $\epsilon_P^G(\sigma) : r_P^G(i_P^G(\sigma)) \to \sigma$ sends an element of its domain, which is an equivalence class of functions $f : G(F) \to V_\sigma$, to $f(1) \in V_\sigma$, while for $\pi \in \mc{R}(G)$ the morphism $\eta_P^G(\pi) : \pi \to i_P^G(r_P^G(\pi))$
sends $w \in V_\pi$ to the function $f : G(F) \to V_\pi/\<\pi(u)v-v|u \in U(F),v \in V_\pi\>$ defined as $f(g)=\pi(g)w$.

There are induction and restriction in stages isomorphisms given as follows. For parabolic subgroups $P \subset Q \subset G$ with unipotent radicals $U \subset P$ and $V \subset Q$ and Levi quotients $M=P/U$ and $L=Q/V$, we have $V \subset U$ and $P/V \subset Q/V=L$ is a parabolic subgroup of $L$ with unipotent radical $U/V$ and Levi quotient $(P/V)/(U/V)=P/U=M$. We have the mutually inverse, functorial isomorphisms
\begin{equation} \label{eq:ipg_stages}
i_P^G\sigma \leftrightarrow i_Q^Gi_{P/V}^L \sigma
\end{equation}
relating an element $f : G \to V_\sigma$ of $i_P^G\sigma$ and an element $f' : G \times L \to V_\sigma$ of $i_Q^Gi_{P/V}^L\sigma$ by $f(g)=f'(g,1)$ and $f'(g,l)=\delta_Q^{-1/2}(l)f(qg)$, where $g \in G(F)$ and $q \in Q(F)$ with image $l \in L(F)$. Moreover, we have the mutually inverse isomorphisms
\begin{equation} \label{eq:rpg_stages}
r_P^G\pi \leftrightarrow r_{P/V}^L r_Q^G\pi
\end{equation}
given by the natural identification $(V_\pi)_{U(F)}=((V_\pi)_{V(F)})_{(U/V)(F)}$.

The functors $i_P^G$ and $r_P^G$ respect automorphisms of $G$ as follows. If $a$ is an $F$-automorphism of $G$ and $\pi$ is a representation of $G$, write $a\pi$ for the representation with underlying vector space $V_\pi$ and action of $G(F)$ given by $a\pi = \pi \circ a^{-1}$. Then the vector spaces of $r_P^G\pi$ and $r_{aP}^G(a\pi)$ are naturally identified and this identification induces an isomorphism of representations
\begin{equation} \label{eq:rpg_auto}
a(r_P^G(\pi)) \to r_{aP}^G(a\pi),
\end{equation}
functorial in $\pi$. For a representation $\sigma$ of $M(F)$ write $a\sigma = \sigma \circ a^{-1}$ for the representation of $aM(F)$ on the same underlying vector space $V_\sigma$. Then
\begin{equation} \label{eq:ipg_auto}
a(i_P^G\sigma) \to i_{aP}^G(a\sigma), f \mapsto f \circ a^{-1}
\end{equation}
is an isomorphism of representations of $G(F)$, functorial in $\sigma$.

The following result is \cite[Theorem 2.9, Remark 2.10]{BZ77}, where it is stated for irreducible $\sigma$, but immediately generalizes to finite-length $\sigma$ by the exactness of $i_P^G$.
\begin{thm}[Bernstein-Zelevinsky] \label{thm:bz77.2.10}
    Let $P_1,P_2 \subset G$ be two parabolic subgroups of $G$ with a common Levi factor $M$. For any finite-length $\sigma \in \mc{R}(M)$, the set of irreducible subquotients of $i_{P_1}^G(\sigma)$ is the same as that of $i_{P_2}^G(\sigma)$.
\end{thm}

The above theorem provides an alternative proof of the following claim, which is \cite[Lemma 5.4(iii)]{BDK86}.
\begin{cor}
    Let $P_1,P_2 \subset G$ be parabolic subgroups with Levi factors $M_1,M_2$ and let $g \in G(F)$ be such that $gM_1g^{-1}=M_2$. For any finite-length $\sigma \in \mc{R}(M_1)$ the irreducible subquotients of $i_{P_1}^G(\sigma)$ and $i_{P_2}^G(\sigma \circ \tx{Ad}(g)^{-1})$ agree.
\end{cor}
\begin{proof}
    Apply \eqref{eq:ipg_auto} with $a=\tx{Ad}(g)$, use that $a(i_P^G\sigma) \cong i_P^G\sigma$, and apply Theorem \ref{thm:bz77.2.10}.
\end{proof}

\subsection{Definition of a complex} \label{sub:defcomp}

In this section we will recall the complex considered in \cite[\S3]{Aub96}. It is functorially associated to an object of $\mc{R}(G)$ and its terms are of the form
\[ X_P=i_P^Gr_P^G = I_P^GR_P^G, \]
where $P$ is a parabolic subgroup of $G$, see \eqref{eq:xnorm}.

The differentials of the complex are based on the following construction of a functorial homomorphism 
\begin{equation} \label{eq:phipq}
    \varphi_P^Q : X_Q \to X_P
\end{equation}
for a pair of parabolic subgroups $P \subset Q \subset G$. Using the induction and restriction in stages isomorphisms \eqref{eq:ipg_stages} and \eqref{eq:rpg_stages} we write $X_P=i_Q^Gr_P^G=i_Q^Gi_{P/V}^Lr_{P/V}^Lr_Q^G$, where $V$ is the unipotent radical of $Q$. Composing the adjunction map $\eta_{P/V}^L : \tx{id}_{\mc{R}(L)} \to i_{P/V}^Lr_{P/V}^L$ (cf. \eqref{eq:irpg_adj}) with $i_Q^G$ on the left and $r_Q^G$ on the right we obtain the desired functorial homomorphism $\varphi^Q_P : X_Q \to X_P$. It has the following simple explicit description: elements of $X_P$ are functions $G \to (V_\pi)_U$, while elements of $X_Q$ are functions $G \to (V_\pi)_V$, since $V \subset U$, there is a natural projection map $(V_\pi)_V \to (V_\pi)_U$, and $\varphi^Q_P$ acts by composing a function $G \to (V_\pi)_V$ with the projection map $(V_\pi)_V \to (V_\pi)_U$.

Fix a minimal parabolic subgroup $P_0 \subset G$ and a maximal split torus $A_0 \subset P_0$, so that $M_0=\tx{Cent}(A_0,G)$ is a Levi factor of $P_0$, and let $S \subset X^*(A_0)$ be the corresponding set of (relative) simple roots. Let $r=|S|=\tx{dim}(A_0/A_G)$. Each $G(F)$-conjugacy class of $F$-parabolic subgroups contains a unique element that contains $P_0$. There is a bijection $J \leftrightarrow P_J$ between the set of subsets of $S$ and the set of $F$-parabolic subgroups of $G$ containing $P_0$, where the Lie algebra of $P_J$ is the sum of $A_0$-weight spaces for all weights that are integral linear combinations of $S$ with non-negative contributions of $S \sm J$. For $I \subset J$ we have $P_I \subset P_J$. 

Write $X_J=X_{P_J}$ for short. For $I \subset J$ we have the functorial homomorphism $\varphi^J_I=\varphi^{P_J}_{P_I} : X_J \to X_I$ of \eqref{eq:phipq}. In order to organize these terms into a complex certain signs are needed, that we now recall.

For $J \subset S$ define the 1-dimensional complex vector space 
\[ \Lambda_J = \Wedge^{|S-J|}(\C^{S-J}). \]
For $I \subset J \subset S$ with $|J|=|I|+1$ let $J \sm I=\{i\}$ and define the isomorphism 
\[ \epsilon^J_I :  \Lambda_J \to \Lambda_I,\qquad \omega \mapsto \omega \wedge e_i. \]
Define $Y_J : \mc{R}(G) \to \mc{R}(G)$ as the composition of $X_J$ (applied first) and tensoring over $\C$ with $\Lambda_J$, considered as a trivial $1$-dimensional representation of $G(F)$ (applied second), thus, for $\pi \in \mc{R}(G)$, 
\[ Y_J(\pi) = X_J(\pi) \otimes_\C \Wedge^{|S-J|}(\C^{S-J}). \] 
Define the morphism of functors $\psi^J_I : Y_J \to Y_I$ as $\psi^J_I = \varphi^J_I \otimes \epsilon^J_I$.

\begin{rem} \label{rem:sign}
    If we introduce a total order on $S$, then this specifies an isomorphism $\Lambda_J \to \C$ of $\C$-vector spaces, and hence an isomorphism of functors $Y_J \to X_J$, under which $\psi^J_I$ is identified with $\varphi^J_I$ multiplied by a certain sign. This is the sign modification of $\varphi^J_I$ that is necessary to organize the functors $X_J$ into a complex. We will look into this in more detail in \S\ref{sub:bs}.
\end{rem}

\begin{rem}
    We have slightly changed the notation of \cite{Aub95} and are using $Y$ in place of $\tilde X$ and $\psi$ in place of $\tilde\varphi$, because we will use the tilde symbol systematically for disconnected groups and associated objects.
\end{rem}

For $0 \leq t \leq r$ define the functor $Y_t : \mc{R}(G) \to \mc{R}(G)$ by
\[ Y_t = \bigoplus_{\substack{J \subset S\\ |J|=t}} Y_J \]
and for $1 \leq t \leq r$ the morphism of functors $d_t : Y_t \to Y_{t-1}$ as the sum of the maps $\psi^J_I$ for all $I \subset J$ with $|J|=t$ and $|I|=t-1$. More precisely, if $y_t \in Y_t(\pi) = \bigoplus_{|J|=t} Y_J(\pi)$ and we write $y_t$ as $\sum_J y_t(J)$ with $y_t(J) \in Y_J(\pi)$, then for $I \subset S$ with $|I|=t-1$ we have
\[ d_t(y_t)(I) = \sum_{\substack{I \subset J \subset S\\ |J|=t}} \psi^J_I(y_t(J)). \]
We thus obtain a (not necessarily exact) sequence of functors $\mc{R}(G) \to \mc{R}(G)$
\[ 0 \to Y_r \to Y_{r-1} \to \dots \to Y_0 \to 0, \]
in particular for each $\pi \in \mc{R}(G)$ a (not necessarily exact) sequence of $G(F)$-representations
\[ 0 \to Y_r(\pi) \to Y_{r-1}(\pi) \to \dots \to Y_0(\pi) \to 0. \]
Note that we have the natural identifications $Y_r=X_r=\tx{id}_{\mc{R}(G)}$, i.e. $Y_r(\pi)=X_r(\pi)=\pi$. 

\subsection{The Aubert involution for connected groups} \label{sub:aubconn}

\begin{dfn}
For $0 \leq t \leq r$ consider the full subcategory $\mc{R}_t(G)$ of $\mc{R}(G)$ consisting of the smooth representations $\pi$ with the following property. For every irreducible subquotient $\sigma$ of $\pi$ there exists $I \subset S$ with $|I|=t$ and an irreducible supercuspidal representation $\tau$ of $M_I(F)$ so that $\sigma$ is a subquotient of $i_{P_I}^G(\tau)$. 
\end{dfn}

\begin{rem}
Bernstein's decomposition implies the orthogonal decomposition
\[ \mc{R}(G) = \prod_{t=0}^r \mc{R}_t(G). \]    
The factor $\mc{R}_0(G)$ is the product of principal series Bernstein components, including the Iwahori-spherical representations and other principal series representations. The factor $\mc{R}_r(G)$ consists of those representations all of whose subquotients are supercuspidal, and every supercuspidal representation lies in $\mc{R}_r(G)$.
\end{rem}

\begin{thm}[Theorem 3.6 of \cite{Aub95}, \cite{Aub96}] \label{thm:aubmain}
    The functors $Y_0,\dots,Y_{t-1}$ vanish on $\mc{R}_t(G)$, and the sequence
    \[ 0 \to Y_r \to Y_{r-1} \to \dots \to Y_t \]
    is exact when restricted to $\mc{R}_t(G)$.
\end{thm}

\begin{dfn}
    Define the functor $D_A^t : \mc{R}_t(G) \to \mc{R}_t(G)$ as 
    \[ D_A^t := \tx{cok}(d_{t+1} : Y_{t+1} \to Y_t). \] Define the \emph{Aubert duality} functor $D_A^G : \mc{R}(G) \to \mc{R}(G)$ as 
    \[ D_A^G :=\prod_{t=0}^r D_A^t. \]
\end{dfn}

Since the functors $Y_t$ are covariant and exact, the same holds for $D_A^G$.

\begin{rem} \label{rem:dgwldf}
    If we replace the chosen pair by $(M_0',P_0')=g(M_0,P_0)g^{-1}$ for some $g\in G(F)$, then the Aubert duality produced by this pair is equivalent to $D_A^G$ under $\tx{Ad}(g)$. 
\end{rem}

\begin{pro} \label{pro:dgird}
    If $\pi$ is irreducible, then so is $D_A^G(\pi)$.
\end{pro}
\begin{proof}
    This is \cite[Corollary 3.9(b)]{Aub95}.
\end{proof}

\begin{pro} \label{pro:dgsc}
    There is an isomorphism of  functors $\tx{id}\to D_A^G$ on the factor $\mc{R}_r(G)$; in fact, they are equal.
\end{pro}
\begin{proof}
    By Theorem \ref{thm:aubmain}, the functors $Y_0,\dots,Y_{r-1}$ vanish on the factor $\mc{R}_r(G)$. Since $Y_r= X_r=\tx{id}$, we conclude that $D_A^G=\tx{id}$ on $\mc{R}_r(G)$.
\end{proof}

To study more properties of $D_A^G$, we need to use another duality functor: the cohomological duality defined by Bernstein.

\subsection{Cohomological duality for connected groups} \label{sub:bercon}

In this subsection we review the definition of cohomological duality, which is defined and studied in \cite[\S5.1]{BernsteinPadicGroups}. Let $\mc{D}^b(G)$ be the bounded derived category of the category of smooth representations of $G(F)$. We have the $G(F)$-bimodule $\mc{H}=\mc{C}^\infty_c(G(F))$ and for any $M\in \mc{D}^b(G)$ we can consider the complex $\tx{RHom}_{G(F)}(M, \mc{H})$, where we have used the left $G(F)$-module structure on $\mc{H}$. The remaining right module structure on $\mc{H}$ endows this complex with a right $G(F)$-module structure, which can be converted to a left $G(F)$-module in the following standard way: $g\cdot x: = x \cdot g^{-1}$ for $g \in G(F)$. This provides the contravariant functor: $D_{\tx{coh}}^G: \mc{D}^b(G)\to \mc{D}^b(G)$. Note that $D_{\tx{coh}}^G$ lands in $\mc{D}^b(G)$ because the category of smooth representations of $G(F)$ has finite cohomological dimension, see \cite[Theorem 29]{BernsteinPadicGroups} and \cite[Lemma 2.2]{YS2026M}. This gives the embedding $\mc{R}(G) \subset\mc{D}^b(G)$ and we can consider $D_{\tx{coh}}^G(\pi) \in \mc{D}^b(G)$ for any smooth representation $\pi$ of $G(F)$.

\begin{thm}[{\cite[Theorem 31]{BernsteinPadicGroups}}] \label{thm:bermain}
    For $\pi\in \mc{R}_t(G)$, the complex $D_{\tx{coh}}^G(\pi)$ has non-zero cohomology at a unique degree $d(t)$ that depends only on $t$ and this cohomology lies in $\mc{R}_t(G)$.
\end{thm}

\begin{dfn}
    Define the functor $D_B^t: \mc{R}_t(G) \to \mc{R}_t(G)$ as $\tx{H}^{d(t)}(D_{\tx{coh}}^G(-))$. Define the cohomological duality $D_B^G=\prod_tD_B^t: \mc{R}(G) \to \mc{R}(G)$. 
\end{dfn}

Note that $D_B^G$ is a contravariant functor. It enjoys the following list of nice properties. Let $P=MU$ be an $F$-parabolic subgroup of $G$ with $M$-opposite $P^-$.

\begin{pro} \label{pro:dhprp}
    The cohomological duality $D_B^G$ satisfies the following properties:
    \begin{enumerate}
        \item $D_B^G$ is an involution, that is, $D_B^G\circ D_B^G$ is equivalent to the identity functor.

        \item $D_B^G\circ i^G_P$ is equivalent to $i^G_{P^-}\circ D_B^M$.

        \item $D_B^M\circ r^G_P$ is equivalent to $r^G_P\circ D_B^G$.

        \item $D_B^G$ preserves irreducibility and is isomorphic to the contragredient functor $(-)^\vee$ on each supercuspidal Bernstein block.

        \item $D_B^G$ commutes with the contragredient functor.

        \item $D_B^G$ is equivalent to $D_A^G \circ (-)^\vee$. 
    \end{enumerate}
\end{pro}
\begin{proof}
    The first four properties are proved in \cite[Theorem 31]{BernsteinPadicGroups}. See also \cite{YS2026M} for an exposition. Property (5) is the main result of \cite{YS26}. Property (6) is the main result of \cite{BBK18}.
\end{proof}

With Proposition \ref{pro:dhprp}(6), we may deduce similar properties for the Aubert duality $D_A^G$.

\begin{pro} \label{pro:dgprp}
    The Aubert duality functor $D_A^G$ satisfies the following properties.
    \begin{enumerate}
        \item $D_A^G$ is an involution, that is, $(D_A^G)^2$ is equivalent to the identity functor.
        \item $D_A^G\circ i^G_P$ is equivalent to $i^G_{P^-}\circ D_A^M$.
        \item $D_A^M\circ r^G_P$ is equivalent to $r^G_{P^-}\circ D_A^G$.
        \item $D_A^G$ commutes with the contragredient functor.
    \end{enumerate}
\end{pro}

\subsection{Uniqueness properties of $D_A^G$ and $D_B^G$} \label{sub:dguniq}

Proposition \ref{pro:dhprp} and \ref{pro:dgprp}  give a list of properties of $D_A^G$ and $D_B^G$. In fact, they are uniquely characterized by these properties. To be more precise, we have to introduce more notions. Recall that we have adjunction maps \eqref{eq:irpg_adj} for the Frobenius reciprocity:
\[\epsilon^G_P(\tau): r^G_Pi^G_P\tau \to \tau \qquad \tx{and} \qquad \eta^G_P(\pi): \pi\to i^G_Pr^G_P\pi \]
for any $\tau \in \mc{R}(M)$ and $\pi\in\mc{R}(G)$. Applying $D^M_A$ and $D^G_A$ respectively and combining with Proposition \ref{pro:dgprp}(2)(3), we have maps:
\[r^G_{P^-}i^G_{P^-}D^M_A(\tau)\xrightarrow{\cong} D^M_A(r^G_Pi^G_P\tau) \xrightarrow{D^M_A(\epsilon^G_P(\tau))} D^M_A(\tau)\]
and 
\[D_A^G(\pi)\xrightarrow{D_A^G(\eta^G_P(\pi))}D^G_A(i^G_Pr^G_P\pi) \xrightarrow{\cong} i^G_{P^-}r^G_{P^-}D^G_A(\pi).\]
It turns out that these two maps equal $\epsilon^G_{P^-}(D^M_A(\tau))$ and $\eta^G_{P^-}(D^G_A(\pi))$ respectively, see \cite{YS26}, and we refer to this property as \emph{Aubert duality preserves Frobenius reciprocity}.

\begin{pro}[\cite{YS26}] \label{pro:dguniq}
    The family of functors $D^M_A: \mc{R}(M)\to \mc{R}(M)$, with $M$ varying over all (standard) Levi subgroups of $G$, is uniquely determined by the following properties:
    \begin{enumerate}
        \item $D^M_A$ is an exact covariant functor for all $M$.
        \item $D_A^G\circ i^G_P$ is equivalent to $i^G_{P^-}\circ D_A^M$.
        \item $D_A^M\circ r^G_P$ is equivalent to $r^G_{P^-}\circ D_A^G$.
        \item $D_A^M$ preserves Frobenius reciprocity in the above sense.
        \item Each $D_A^M$ respects Bernstein decomposition, and it is isomorphic (in fact, equal to) the identity functor on supercuspidal blocks.
    \end{enumerate}
\end{pro}

A similar conclusion holds for $D^G_B$. To state it, we first recall the second adjointness theorem (\cite[VI 9.6 Theorem]{Ren10}), which says that $i^G_P$ is left adjoint to $r^G_{P^-}$; in other words, we have adjunction maps
\[\kappa^G_P(\pi): i^G_Pr^G_{P^-}\pi \to \pi  \qquad \tx{and} \qquad \zeta^G_P(\tau): \tau \to r^G_{P^-}i^G_P\tau.\]
Starting from Frobenius reciprocity, or first adjointness as in \cite{YS26}, applying $D^G_B$ and combining with Proposition \ref{pro:dhprp}(2)(3), we have maps:
\[D^M_B(\tau)\xrightarrow{D^M_B(\epsilon^G_P(\tau))} D^M_B(r^G_Pi^G_P\tau) \xrightarrow{\cong} r^G_Pi^G_{P^-}D^M_B(\tau)\]
and 
\[i^G_{P^-}r^G_PD^G_B(\pi)\xrightarrow{\cong} D^G_B(i^G_Pr^G_P\pi) \xrightarrow{D^G_B(\eta^G_P(\pi))} D^G_B(\pi).\]
It turns out these two maps equal $\zeta^G_{P^-}(D^M_B(\tau))$ and $\kappa^G_{P^-}(D^G_B(\pi))$ respectively. Conversely, starting from second adjointness, we get first adjointness back. This is proved in \cite[Proposition 2.3]{YS26}, and we refer to this property as \emph{cohomological duality permutes first and second adjointness}.

\begin{pro}[\cite{YS26}] \label{pro:dhuniq}
    The family of functors $D^M_B: \mc{R}(M)\to\mc{R}(M)$, with $M$ varying over all (standard) Levi subgroups of $G$, is uniquely determined by the following properties:
    \begin{enumerate}
        \item $D^M_B$ is an exact contravariant functor for all $M$.
        \item $D_B^G\circ i^G_P$ is equivalent to $i^G_{P^-}\circ D_B^M$.
        \item $D_B^M\circ r^G_P$ is equivalent to $r^G_P\circ D_B^G$.
        \item $D^G_B$ permutes first and second adjointness in the above sense.
        \item Each $D^M_B$ respects Bernstein decomposition, and it is isomorphic to the contragredient functor on supercuspidal blocks.
    \end{enumerate}
\end{pro}

\subsection{Equivariance properties of $D_A^G$ and $D_B^G$} \label{sub:dgequiv}

In this subsection we explore the equivariance properties of the functors $D_A^G$ and $D_B^G$ with respect to automorphisms of $G$. The former will be used in the next section for the construction of the functor $D_A^{\tilde G}$ for a disconnected group $\tilde G$.

Let $a$ be an $F$-automorphism of the connected reductive group $G$ that preserves the chosen minimal parabolic pair $(P_0,M_0)$. Then $a$ induces a bijection on the set $S$.

For a representation $\pi \in \mc{R}(G)$ define $a\pi \in \mc{R}(G)$ to be the representation of $G(F)$ twisted by $a^{-1}$, that is, it has underlying vector space $V_\pi$ and $(a\pi)(g)=\pi(a^{-1}(g))$. Assume given two representations $\pi_1,\pi_2 \in \mc{R}(G)$ and an isomorphism $\theta_a : a\pi_1 \to \pi_2$. We will define an isomorphism $aY_t(\pi_1) \to Y_t(\pi_2)$ for any $0 \leq t \leq r$. 

Consider a parabolic subgroup $P \subset G$. We have $X_P(\pi)=i_P^G(r_P^G(\pi))$. Composing \eqref{eq:ipg_auto} and \eqref{eq:rpg_auto} we obtain the isomorphism
\[ c_P(\pi_1): a(i_{P}^G(r_{P}^G(\pi_1))) \to i_{aP}^G(r_{aP}^G(a\pi_1)), \]
which we further compose with the isomorphism
\[ i_{aP}^G(r_{aP}^G(a\pi_1)) \to i_{aP}^G(r_{aP}^G(\pi_2))\]
obtained functorially from $\theta_a$ to arrive at an isomorphism 
\[ a(i_{P}^G(r_{P}^G(\pi_1))) \to i_{aP}^G(r_{aP}^G(\pi_2)), \]
that is, an isomorphism 
\begin{equation} \label{eq:xpat}
X_P(a,\theta_a) : aX_{P}(\pi_1)\xrightarrow{c_P(\pi_1)}X_{aP}(a\pi_1) \xrightarrow{X_{aP}(\theta_a)} X_{aP}(\pi_2). 
\end{equation}
We can make this explicit as follows. 
\begin{eqnarray*}
   X_P(\pi)&=&\{f : G(F) \to (V_\pi)_{U(F)}|f(pg)=\pi(p)f(g)\},\\
    X_P(a,\theta_a)(f)(g)&=&\theta_a(f(a^{-1}(g))).
\end{eqnarray*}
In addition, for $J \subset S$ we define the isomorphism
\[ \lambda_J(a) : \Lambda_J \to \Lambda_{aJ} \]
as the $|S-J|$-th exterior power of the isomorphism $\C^{S-J} \to \C^{S-aJ}$ induced by the bijection $a : (S-J) \to (S-aJ)$.
The tensor product $X_{P_J}(a,\theta_a) \otimes \lambda_J(a)$ provides an isomorphism
\[ Y_J(a,\theta_a) : aY_{J}(\pi_1) \to Y_{aJ}(\pi_2) \]
which is the composition of $c_J^a(\pi_1):=c_{P_J}(\pi_1)\otimes \lambda_J(a): aY_J(\pi_1)\to Y_{aJ}(a\pi_1)$ with $Y_{aJ}(a\pi_1)\to Y_{aJ}(\pi_2)$. Note that these two maps are $G(F)$-equivariant isomorphisms. Putting these together for all $J \subset S$ with $|J|=t$ we obtain the desired isomorphism
\begin{equation} \label{eq:xtaut}
Y_t(a,\theta_a) : aY_t(\pi_1) \xrightarrow{c_t^a(\pi_1)} Y_t(a\pi_1) \to Y_t(\pi_2).
\end{equation}

More precisely, for $y_t \in Y_t(\pi_1)$ with $y_t = \sum_J y_t(J)$, $y_t(J) \in Y_J(\pi_1)$, we have
\begin{equation} \label{eq:xtaut-prec}
Y_t(a,\theta_a)(y_t)({aJ}) = Y_J(a,\theta_a)(y_t(J)).
\end{equation}
Note that, for any $\pi \in \mc{R}(G)$, the vector spaces $aY_t(\pi)$ and $Y_t(\pi)$ are naturally identified, just the action of $G(F)$ on them differs by the twist by $a$. In particular, both $aY_\bullet(\pi)$ and $Y_\bullet(\pi)$ are complexes (with the same differential $d_t$) of representations of $G(F)$.

The following lemma will be used in the proof of Proposition \ref{pro:dgequiv}.

\begin{lem}\ \\[-15pt] \label{lem:equiv}
    \begin{enumerate}
        \item For two parabolic subgroups $P \subset Q \subset G$ one has $X_P(a,\theta_a)\circ\varphi^Q_P=\varphi^{aQ}_{aP}\circ X_Q(a,\theta_a)$.
        \item For two subsets $I \subset J \subset S$ with $|J|=|I|+1$ one has $\lambda_I(a)\circ\epsilon^J_I = \epsilon^{aJ}_{aI}\circ \lambda_J(a)$.
    \end{enumerate}
\end{lem}
\begin{proof}
Both claims follow by unwinding the definitions. More precisely, recalling that $\varphi^Q_P$ is given by postcomposition with the natural projection $(V_\pi)_V \to (V_\pi)_U$, where $Q=LV$ and $P=MU$, while $X_P(a,\theta_a)$ is given by precomposition with $a^{-1}$ and postcomposition with $\theta_a$, the first claim is immediate.
    
For the second claim we recall that $\lambda_J(a)$ is given by the top exterior power of the isomorphism $\C^{S-J} \to \C^{S-aJ}$ induced by the bijection $a : (S-J) \to (S-aJ)$. The analogous formula holds for $\lambda_I(a)$. On the other hand, $\epsilon^J_I(\omega)=\omega \wedge s$, where $\{s\} = J -I$. But then $\{a(s)\}= aJ-aI$ and $\epsilon^{aJ}_{aI}(\omega')=\omega' \wedge a(s)$, and the claim is obvious.
\end{proof}

\begin{pro} \label{pro:dgequiv}
    \begin{enumerate}
        \item The isomorphism $Y_t(a,\theta_a)$ is $G(F)$-equivariant, i.e. it is an isomorphism of $G(F)$-representations.
        \item The differentials $d_t : Y_t(\pi_i) \to Y_{t-1}(\pi_i)$ intertwine $Y_t(a,\theta_a)$ with $Y_{t-1}(a,\theta_a)$ and $c_t^a(\pi_1)$ with $c_{t-1}^a(\pi_1)$.
        \item The isomorphism $Y_t(a,\theta_a)$ is functorial in triples $(\pi_1,\pi_2,\theta_a)$: given a second triple $(\pi_1',\pi_2',\theta'_a)$ and homomorphisms $f_i : \pi_i \to \pi_i'$ that intertwine $\theta_a$ and $\theta'_a$, the homomorphisms $Y_t(f_i) : Y_t(\pi_i) \to Y_t(\pi_i')$ intertwine $Y_t(a,\theta_a)$ and $Y_t(a,\theta'_a)$.
        \item The isomorphisms $Y_t(a,\theta_a)$ compose: Given another $F$-automorphism $b$ of $G$ that preserves $(P_0,M_0)$ and a homomorphism $\theta_b : b\pi_2 \to \pi_3$, we have the equality $Y_t(ba,\theta_b\circ b(\theta_a))=Y_t(b,\theta_b)\circ b(Y_t(a,\theta_a))$ of homomorphisms $baY_t(\pi_1) \to Y_t(\pi_3)$.
    \end{enumerate}
\end{pro}
\begin{proof}
    (1) The isomorphism $X_J(a,\theta_a)$ is $G(F)$-equivariant by construction, and both source and target of $\lambda_J(a)$ carry the trivial $G(F)$-action.

    (2) We abbreviate $Y_t(a,\theta_a)$ by $\theta_t$ and $Y_J(a,\theta_a)$ by $\theta_J$. Let $y_t \in Y_t(\pi_1)$ and write again $y_t = \sum_J y_t(J)$, the sum running over $J \subset S$ with $|J|=t$. For $I \subset S$ with $|I|=t-1$, the following calculation completes the proof,
    \begin{eqnarray*}
        d(\theta_t(y_t))({aI})
        &=&\sum_{\substack{I \subset J \subset S\\ |J|=t}} \psi^{aJ}_{aI}(\theta_t(y_t)({aJ}))\\
        \eqref{eq:xtaut-prec}&=&\sum_{\substack{I \subset J \subset S\\ |J|=t}} \psi^{aJ}_{aI}(\theta_J(y_t(J)))\\
        (*)&=&\sum_{\substack{I \subset J \subset S\\ |J|=t}} \theta_I(\psi^{J}_{I}(y_t(J)))\\
        &=&\theta_I\Big(\sum_{\substack{I \subset J \subset S\\ |J|=t}} \psi^{J}_{I}(y_t(J))\Big)\\
        &=&\theta_I(d(y_t)(I))
        =\theta_{t-1}(d(y_t))({aI})
    \end{eqnarray*}
    where the equality marked $(*)$ rests on the claim $\theta_I \circ \psi^J_I = \psi^{aJ}_{aI} \circ \theta_J$, which in turn is implied by Lemma \ref{lem:equiv}. The computation for $c_t^a(\pi_1)$ is similar.

    (3) The proof reduces immediately to verifying that $X_P(a,\theta'_a)\circ X_P(f_1) = X_{aP}(f_2) \circ X_P(a,\theta_a)$ for a parabolic subgroup $P \subset G$. This follows from the functoriality of the isomorphisms \eqref{eq:ipg_auto} and \eqref{eq:rpg_auto} and the assumed identity $\theta'_a\circ f_1 = f_2 \circ \theta_a$, to which one applies the functor $X_P$.

    (4) This follows from the fact that each of the isomorphisms \eqref{eq:ipg_auto} and \eqref{eq:rpg_auto}, computed for the product $ba$, is equal to the composition of the corresponding isomorphism for $b$ and that for $a$, from the fact that the functor $X_P$ (like any functor) respects composition of morphisms, and from the fact that $\lambda_J(ba)=\lambda_{aJ}(b)\lambda_J(a)$.
\end{proof}

As a direct consequence, the isomorphisms $Y_t(a,\theta_a)$ induce a $G(F)$\-isomorphism
\[Y(a,\theta_a): aD_A^G(\pi_1)\xrightarrow{c^a(\pi_1)}D_A^G(a\pi_1)\xrightarrow{D_A^G(\theta_a)} D_A^G(\pi_2)\]

We now turn to the functor $D_B^G$. For this purpose, we need to make some preparations. Recall that we have a $G(F)$-bimodule $\mc{H}$. This amounts to saying we have a $G(F)\times G(F)$-action via the following rule: for any $g_1, g_2, x\in G(F)$ and $f\in \mc{H}$,
\[((g_1,g_2)\cdot f)(x):=f(g_1^{-1}xg_2)\]
We now twist the action by the given automorphism $a$ of $G$ as follows. Let $(1\times a)\mc{H}$ (resp.~$(a^{-1}\times 1)\mc{H}$) be the $G(F)\times G(F)$-representation, whose underlying vector space is $\mc{H}$, with action
\[((g_1,g_2)\cdot f)(x):=f(g_1^{-1}x a^{-1}(g_2)).\ (\tx{resp.} ((g_1, g_2)\cdot f)(x):= f(a(g_1^{-1})xg_2))\]

\begin{lem} \label{lem:equivedh}
    We have $G(F)\times G(F)$-isomorphism $\psi: (1\times a)\mc{H} \to (a^{-1}\times 1)\mc{H}$, by $\psi(f)=f\circ a^{-1}$. This induces the following functorial isomorphism for any smooth $G(F)$-representation $\pi$:
    \[a\tx{Hom}_G(\pi, \mc{H}) \to \tx{Hom}_G(a\pi, \mc{H})\]
\end{lem}
\begin{proof}
    The first claim is clear. Note that we have identification of vector spaces 
    \[a\tx{Hom}_G(\pi, \mc{H}) = \tx{Hom}_G(\pi, (1 \times a)\mc{H}),\]
    compatible with the $G(F)$-action by definition. Composing with $\psi$, we have the isomorphism 
    \[\tx{Hom}_G(\pi, (1\times a)\mc{H})\to \tx{Hom}_G(\pi, (a^{-1}\times 1)\mc{H}),\]
    and the latter is equal to $\tx{Hom}_G(a\pi, \mc{H})$ as $G(F)$-representations.
\end{proof}

Let $\theta_a: a\pi_1\to \pi_2$ be a $G(F)$-isomorphism as above. Let $0\to Z_s\to \dots \to Z_0\to \pi_1\to 0$ be a finite projective resolution with finitely generated projectives. By Lemma \ref{lem:equivedh}, we have functorial isomorphism
\[a\tx{Hom}_G(Z_i,\mc{H}) \to \tx{Hom}_G(aZ_i, \mc{H})\]
and hence an isomorphism
\[c_h(\pi_1): aD_B^G(\pi_1)\to D_B^G(a\pi_1)\]
Composing with the inverse of the functorial isomorphism (recall that $D_B^G$ is contravariant)
\[D_B^G(\theta_a): D_B^G(\pi_2)\to D_B^G(a\pi_1)\]
we obtain the desired isomorphism:
\[H(\theta_a): aD_B^G(\pi_1)\to D_B^G(\pi_2)\]

It is plain that $(a\pi_1)^\vee =a \pi_1^\vee$ and that $(\theta_a^{-1})^\vee: a\pi_1^\vee = (a\pi_1)^\vee \to \pi_2^\vee$ is an isomorphism. Therefore, we can also speak of the equivariance map $Y(a,(\theta_a^{-1})^\vee): aD^G_A(\pi_1^\vee)\to D^G_A(\pi_2^\vee)$. Recall that $D_B^G$ is equivalent to $D_A^G \circ (-)^\vee$ by Proposition \ref{pro:dhprp}(6), the following proposition says that the two equivariance properties for $D_A^G((-)^\vee)$ and $D_B^G$ are compatible under this isomorphism.

\begin{pro} \label{pro:dhdgequ}
    Let $\beta_\pi: D_B^G(\pi)\to D_A^G(\pi^\vee)$ be the functorial isomorphism in Proposition \ref{pro:dhprp}(6). Then the following diagram commutes:
    \[ \xymatrix{aD_B^G(\pi_1) \ar[rr]^{c_h(\pi_1)} \ar[d]_{a\beta_\pi}&&D_B^G(a\pi_1)\ar[rr]^{D_B^G(\theta_a^{-1})}\ar[d]_{\beta_{a\pi_1}}&&D_B^G(\pi_2)\ar[d]_{\beta_{\pi_2}} \\
    aD_A^G(\pi_1^\vee)\ar[rr]^{c^a(\pi_1^\vee)}&&D_A^G((a\pi_1)^\vee) \ar[rr]^{D_A^G((\theta_a^{-1})^\vee)}&&D_A^G(\pi_2^\vee)   
    }\]
    where $c^a(\pi_1^\vee)$ is the composition $c^a(\pi_1^\vee): aD_A^G(\pi_1^\vee)\to D_A^G(a\pi_1^\vee)$ with the equality $a\pi_1^\vee = (a\pi_1)^\vee$. Note that the bottom map is $Y(a,(\theta_a^{-1})^\vee)$.
\end{pro}
\begin{proof}
    The right square commutes by the functoriality of $\beta$. It remains to check that the left square also commutes. Let us recall the construction of $\beta_\pi$ in \cite{BBK18}, and we refer to that paper for more details. We may assume that $G(F)$ has compact center, loc. cit. First, there is a so-called specialization complex $\mathfrak{C}$ consisting of $G(F)\times G(F)$-representations:
    \[0\to Z_S \to \bigoplus_{|I|=r-1} Z_I \to\dots\to Z_\emptyset \to 0,\]
    where $Z_I=i^{G\times G}_{P_I\times P_I^-}\mc{C}^\infty(M_I)$. We warn the reader that in \cite{BBK18}, their convention is $P_\emptyset=G$ while our convention is $P_\emptyset = P_0$. Then, for $\pi\in \mc{R}(G)$, we have a functorial isomorphism of complexes:
    \[DL(\pi^\vee)\to \tx{Hom}_{G(F)}(\pi, \mathfrak{C}),\]
    where $DL(\pi)$ is the complex $0\to X_r(\pi)\to\dots X_0(\pi)$ with a chosen total order in $S$ (see remark \ref{rem:sign}), and $\tx{Hom}_{G(F)}(-, Z_I)$ is the space of $G(F)$-equivariant homomorphisms with respect to the action of $G(F)$ for $i^G_P$ and viewed as a $G(F)$-representation via the action for $i^G_{P^-}$. More precisely, this isomorphism is given by the composition:
    \begin{align*}
    \tx{Hom}_{G(F)}(\pi, i^{G\times G}_{P\times P^-} \mc{C}^\infty(M)) &= \tx{Hom}_{M(F)}(r^G_P\pi, (1\times i^G_{P^-})\mc{C}^\infty (M))\\
    & = i^G_{P^-}\tx{Hom}_{M(F)}(r^G_P\pi,\mc{C}^\infty(M)) \\
    &= i^G_{P^-}(r^G_P\pi)^\vee=i^G_{P^-}r^G_{P^-}\pi^\vee 
    \end{align*}
    where the second equality is a direct verification, c.f. claim \ref{clm:ihom}, and the third equality is given by matrix coefficients. In particular, this isomorphism gives a complex with respect to the opposite simple roots $(M_0, P_0^-)$. As $a$ also preserves this pair, this does no harm by remark \ref{rem:dgwldf}. On the other hand, we may form a quotient $\bar{Z}_I=i^{G\times G}_{P\times P^-}(\mc{C}^\infty(M_I)/\mc{C}_0^\infty(M_I))$, where $\mc{C}_0^\infty(M_I)$ is a certain subspace; more precisely, it consists of those functions whose support is closed in the partial compactification of $M_I$, cf. \cite[Proposition 2.4]{BBK18}. In particular, $Z_S=\bar{Z}_S=\mc{C}^\infty(G)$. We then obtain a quotient map of complexes: $\mathfrak{C}\to \bar{\mathfrak{C}}$. This quotient gives an injective resolution $0\to \mc{H}\to \bar{\mathfrak{C}}$ with $\mc{H}\to \bar{Z}_S$ the inclusion map, and inducing an isomorphism of complexes in $\mc{D}^b(G)$:
    \[DL(\pi^\vee)\to \tx{RHom}_G(\pi, \mc{H}).\]

    Clearly, we have a commutative diagram:
    \begin{equation}\label{diag:dhdgeqvr}
    \xymatrix{
        a\tx{Hom}_G(\pi, \mathfrak{C})\ar[r]\ar[d]_{=} & a\tx{Hom}_G(\pi, \bar{\mathfrak{C}})\ar[d]_{=}\\
        \tx{Hom}_G(\pi, (1\times a)\mathfrak{C}) \ar[r] & \tx{Hom}_G(\pi, (1\times a)\bar{\mathfrak{C}}).
    }
    \end{equation}
    Here $(1\times a)i^{G\times G}_{P\times P^-}$ means that we twist on the action of $G(F)$ for $i^G_{P^-}$ by $a^{-1}$ and leave the other action unchanged. It is easy to verify that the following map
    \[T_P: (1\times a)i^{G\times G}_{P\times P^-}\mc{C}^\infty(M)\to (a^{-1}\times 1)i^{G\times G}_{aP\times aP^-}\mc{C}^\infty(aM),\]
    by 
    \[(T_P(F)(x,y))(m)=F(a^{-1}x, a^{-1}y)(a^{-1}m),\] 
    is well-defined and a $G(F)\times G(F)$-equivariant isomorphism. Since the automorphism $a$ preserves $(M_0,P_0)$, we may sum over all $I\subseteq S$ and obtain a $G(F)\times G(F)$-equivariant isomorphism 
    \[T: (1\times a)\mathfrak{C} \to (a^{-1}\times 1)\mathfrak{C},\]
    which induces an isomorphism on the quotient 
    \[\bar{T} : (1\times a)\bar{\mathfrak{C}} \to (a^{-1}\times 1)\bar{\mathfrak{C}}. \]
    Since $(1\times a)\bar{\mathfrak{C}}$ (resp.~$(a^{-1}\times 1)\bar{\mathfrak{C}}$) gives a resolution of $(1\times a)\mc{H}$ (resp.~$(a^{-1}\times 1)\mc{H}$), and the morphism $\bar{T}$ is compatible with the $\psi$ defined in Lemma \ref{lem:equivedh} (they both are given by precomposing with $a^{-1}$), we have following commutative diagram,
    \begin{equation} \label{diag:dgdheqvr'}
    \scalebox{0.95}{
    \xymatrix{
        \tx{Hom}_G(\pi, (1\times a)\mathfrak{C})\ar[d]_{T}\ar[r]& \tx{Hom}_G(\pi, (1\times a)\bar{\mathfrak{C}})\ar[d]_{\bar{T}}\ar[r]& \tx{RHom}_G(\pi, (1\times a)\mc{H})\ar[d]_{\psi} \\
        \tx{Hom}_G(\pi, (a^{-1}\times 1)\mathfrak{C})\ar[r]\ar[d]_{=}& \tx{Hom}_G(\pi, (a^{-1}\times 1)\bar{\mathfrak{C}})\ar[r]\ar[d]_{=}&\tx{RHom}_G(\pi,(a^{-1}\times 1)\mc{H} )\ar[d]_{=}\\
        \tx{Hom}_G(a\pi, \mathfrak{C}) \ar[r]&\tx{Hom}_G(a\pi, \bar{\mathfrak{C}})\ar[r]&\tx{RHom}_G(a\pi, \mc{H}).
    }}
    \end{equation}
    Combining diagrams \ref{diag:dhdgeqvr} and \ref{diag:dgdheqvr'}, we obtain the following commutative diagram
    \[\xymatrix{
        aDL(\pi^\vee)\ar[r]\ar[d]& a\tx{Hom}_G(\pi, \mathfrak{C})\ar[r]\ar[d]&a\tx{RHom}_G(\pi, \mc{H})\ar[d]\\
        DL(a\pi^\vee)\ar[r]&\tx{Hom}_G(a\pi, \mathfrak{C})\ar[r]&\tx{RHom}_G(a\pi, \mc{H}),
    }\] 
    It is plain that the vertical maps for $DL(\pi^\vee)$ and $\tx{RHom}_G(\pi, \mc{H})$ are exactly the equivariance homomorphisms we defined above, and we complete the proof.
\end{proof}

\subsection{The Borel--Serre resolution} \label{sub:bs}

In \cite{BoSe76}, Borel and Serre study the cohomology of the spherical Tits building of $G$. In \cite[\S3]{BoSe76} they employ a complex very similar to the complex of Theorem \ref{thm:aubmain}. In this subsection we will compare the two complexes, because they will eventually lead to two different versions of $D_A^{\tilde G}$.

Even though Borel and Serre use $\Z$-modules instead of $G(F)$\-representations, one can use the same construction in the setting of $G(F)$-representations and arrive at a complex whose terms are the functors $X_t$ of \S\ref{sub:defcomp}. These differ from the terms $Y_t$ of Theorem \ref{thm:aubmain} by the tensor factor $\Lambda_J$ of the direct factor $Y_J$ of $Y_t$, where $J \subset S$ satisfies $|J|=t$. The factor $\Lambda_J$ is responsible for adding signs into the differentials of the complex $Y$. Since it is omitted in the construction of Borel--Serre, the signs have to be added instead into the differentials manually. Borel and Serre do this as follows.

Write $r=|S|$ and fix a bijection $S \to \{1,\dots,r\}$. This has the effect of fixing a total order on $S$ and can also be understood as fixing an orientation on the chamber in the spherical building corresponding to the minimal parabolic subgroup $P_0$. Every facet of that chamber is then endowed with the induced orientation, which corresponds to endowing every subset of $S$ with the induced total order. Consider $I \subset J \subset S$ such that $|J|=|I|+1$. Thus $J=I \cup \{j\}$ for some $j \in S \sm I$. Define
\begin{equation} \label{eq:sign}
n(I,J)=\#\{i \in J|\, i \leq j\},\qquad \epsilon(I,J) = (-1)^{n(I,J)}.	
\end{equation}
Of course the quantities $n(I,J)$ and $\epsilon(I,J)$ depend on the choice of bijection $S \to \{1,\dots,r\}$. Define
\[ '\varphi_I^J : X_J \to X_I,\qquad '\varphi_I^J = \epsilon(I,J) \cdot \varphi_I^J\]
and 
\[ {}'d_t = \sum_{\substack{I \subset J\\ t=|J|=|I|+1}} {}'\varphi_I^J : X_t \to X_{t-1}. \]
A minor generalization of \cite[Corollary 3.3]{BoSe76}, which applies to the case $\pi=\mathbf{1}$, shows that
\[ 0 \to X_r(\pi) \to X_{r-1}(\pi) \to  \dots \to X_t
(\pi) \]
is exact for $\pi \in \mc{R}_t(G)$. 

\begin{fct} \label{fct:bs1}
    For $I=\{i_1<\dots<i_t\} \subset S$ define $e_I \in \Wedge^{|I|}\C^I$ as $e_I = e_{i_1} \wedge \dots \wedge e_{i_t}$. Define $\tilde e_{S-I} \in \Wedge^{|S-I|}\C^{S-I}$ by the condition $\tilde e_{S - I} \wedge e_I = (-1)^{|I|} e_S$. Let $\kappa_I : \C \to \Lambda_I$ be the isomorphism defined by $1 \mapsto \tilde e_{S-I}$. The isomorphisms $\tx{id}_{X_I}\otimes \kappa_I : X_I \to Y_I$ splice to isomorphisms $X_t \to Y_t$ that translate the differential $'d_t$ to the differential $d_t$.
\end{fct}

For applications to the definition of $D_A^{\tilde G}$ we now consider the setting of \S\ref{sub:dgequiv}, i.e. an $F$-automorphism $a$ of $G$ that preserves the chosen minimal parabolic pair $(M_0,P_0)$, and hence induces a permutation of $S$. Let $\pi_1,\pi_2 \in \mc{R}(G)$, and let $\theta_a : a\pi_1 \to \pi_2$ be an isomorphism. We then have the isomorphisms $X_P(a,\theta_a) : aX_P(\pi_1) \to X_{aP}(\pi_2)$ of \eqref{eq:xpat}, which splice to an isomorphism $X_t(a,\theta_a) : aX_t(\pi_1) \to X_t(\pi_2)$. They are $G(F)$-equivariant, functorial, and multiplicative in $a$, i.e. the analogs of Proposition \ref{pro:dgequiv}(1,3,4) hold. However, they do not intertwine the differentials $'d_t$, i.e. the analog of Proposition \ref{pro:dgequiv}(2) fails. This is because the sign $\epsilon(I,J)$ is not equivariant under automorphisms, i.e. $\epsilon(a(I),a(J))$ is generally not equal to $\epsilon(I,J)$. This can be repaired as follows. Denote by $\Sigma_t$ the symmetric group on $t$ letters.

\begin{dfn} \label{dfn:sign}
Let $I \subset S$ be a subset of cardinality $t$ and let $\sigma \in \Sigma_r$.
\begin{enumerate}
	\item Let $\sigma_I \in \Sigma_t$ be the unique element such that, if $I=\{i_1<i_2<\dots<i_t\}$, then $\sigma(I)=\{\sigma(i_{\sigma_I^{-1}(1)})<\dots<\sigma(i_{\sigma_I^{-1}(t)})\}$. 	
	\item Let $\epsilon(\sigma,I)=\tx{sgn}(\sigma_I)$.
\end{enumerate}
\end{dfn}

\begin{fct} \label{fct:bs2} 
    \begin{enumerate}
        \item Given $\sigma,\tau \in \Sigma_r$ we have $(\sigma\tau)_I=\sigma_{\tau(I)} \cdot \tau_I$. In particular, we have $\epsilon(\sigma\tau,I)=\epsilon(\sigma,\tau(I)) \cdot \epsilon(\tau,I)$.
        \item If $I$ is invariant under $\sigma$, then $\epsilon(\sigma,I)$ is the sign of the permutation $\sigma|_I$.
        \item For subsets $I \subset J \subset S$ with $|J|=|I|+1$, and $\sigma \in \Sigma_r$, we have
\[ \epsilon(\sigma(I),\sigma(J))/\epsilon(I,J) = \epsilon(\sigma,I)/\epsilon(\sigma,J). \]
        \item The isomorphisms $'X_I(a,\theta_a)=\epsilon(a^{-1},I)\cdot X_I(a,\theta_a) : aX_I(\pi_1) \to X_{aI}(\pi_2)$ splice to an isomorphism $'X_t(a,\theta_a) : aX_t(\pi_1) \to X_t(\pi_2)$ that, in addition to being $G(F)$-equivariant, functorial, and multiplicative, also commutes with the differential $'d_t$.
        \item The isomorphism $\tx{id}_{X_I}\otimes\kappa_I$ of Fact \ref{fct:bs1} translates $Y_t(a,\theta_a)$ to $\tx{sgn}(a|S) \cdot {}'X_t(a,\theta_a)$.
    \end{enumerate}
\end{fct}

\begin{rem}  If one prefers to think of a chosen total order on $S$, one can define $\sigma_I$ alternatively as the following permutation of the subset $I$ of $S$. For any two subsets $I_1,I_2$ of the same cardinality let $\beta_{I_2,I_1} : I_1 \to I_2$ be the unique bijection preserving the orders on $I_1$ and $I_2$ induced by the order on $S$. Define $\sigma_I : I \to I$ by $\beta_{I,\sigma(I)} \circ (\sigma|_I)$. The formula of Fact \ref{fct:bs2}(1) then becomes $(\sigma\tau)_I=\beta_{I,\tau(I)} \circ \sigma_{\tau(I)} \circ \beta_{\tau(I),I} \circ \tau_I$.
\end{rem}

\section{The Aubert involution for disconnected groups} \label{sec:aubdisc}

In this section we consider a non-archimedean local field $F$ and an affine algebraic group $\tilde G$ whose identity component $G$ is reductive. The topological group $\tilde G(F)$ is locally compact and totally disconnected and we consider again the category $\mc{R}(\tilde G)$ of smooth representations of $\tilde G(F)$ of finite length. The goal is to define a duality functor $D_A^{\tilde G} : \mc{R}(\tilde G) \to \mc{R}(\tilde G)$ in analogy with the case $\tilde G=G$. In fact, we will have $\tx{Res}\circ D_A^{\tilde G}=D_A^G\circ\tx{Res}$, where $\tx{Res} : \mc{R}(\tilde G) \to \mc{R}(G)$ denotes the restriction functor.

\subsection{The definition of $D_A^{\tilde G}$ and first properties} \label{sub:defdgdis}

Fix a minimal parabolic pair $(M_0,P_0)$ of $G$ and denote again by $A_0$ the maximal split central torus in $M_0$ and by $S \subset X^*(A_0)$ the set of simple (relative) roots of $A_0$ in $G$. Let $r=|S|$. For any $0 \leq t \leq r$ we define $\mc{R}_t(\tilde G)$ to be the preimage of $\mc{R}_t(G)$ under the restriction functor $\mc{R}(\tilde G) \to \mc{R}(G)$.

\begin{lem}
    The category $\mc{R}(\tilde G)$ has the orthogonal decomposition $\prod_t \mc{R}_t(\tilde G)$.
\end{lem}
\begin{proof}
    This is a well-known result of Bernstein when $\tilde G=G$. In the general case, the orthogonality of the factors follows from the fact that the restriction functor $\mc{R}_t(\tilde G) \to \mc{R}_t(G)$ is faithful. Consider now $\tilde\pi \in \mc{R}(\tilde G)$ and write $\tx{Res}(\tilde \pi)=\bigoplus_t \pi_t$. This decomposition is canonical. The action of $\tilde G(F)$ on $G(F)$ by conjugation preserves each factor $\mc{R}_t(G)$ of $\mc{R}(G)$, hence each summand $\pi_t$. Therefore $\pi_t$ is a $\tilde G(F)$-subrepresentation of $\tilde\pi$, and we conclude that the decomposition $\tilde \pi = \bigoplus_t \pi_t$ holds and $\pi_t \in \mc{R}_t(\tilde G)$.
\end{proof}

\begin{rem}
    We may define a representation $\tilde\pi \in \mc{R}(\tilde G)$ to be supercuspidal, if $\tx{Res}^{\tilde G}_G(\tilde\pi)$ is supercuspidal. Then $\mc{R}_r(\tilde G)$ contains all supercuspidal representations, and consists of those representations all of whose subquotients are supercuspidal.
\end{rem}

We will now construct the functor $D_A^{\tilde G} : \mc{R}_t(\tilde G) \to \mc{R}_t(\tilde G)$ for a fixed $t$. For every $g \in \tilde G(F)$ the pair $g(M_0,P_0)g^{-1}$ is also a minimal parabolic pair of $G$, hence conjugate to $(M_0,P_0)$ under $G(F)$. Therefore $\tilde G(F)=G(F) \cdot \tilde N(F)$, where $\tilde N \subset \tilde G$ is the normalizer of $(M_0,P_0)$ in $\tilde G$. We have $\tilde N \cap G = M_0$ and hence $\tilde N(F)/M_0(F) = \tilde G(F)/G(F)$. The group $\tilde N(F)$ acts naturally on $A_0$ and this action descends to $\tilde N(F)/M_0(F) = \tilde G(F)/G(F)$. This produces an action of $\tilde G(F)$ on $A_0$, which preserves the set of simple relative roots $S \subset X^*(A_0)$, with $G(F)$ acting trivially.

Given $n \in \tilde N(F)$ consider the automorphism $a=\tx{Ad}(n)$ of $G$. From a representation $\tilde\pi \in \mc{R}_t(\tilde G)$ we obtain the representation $\pi=\tilde\pi|_{G(F)}$ and the isomorphism $\tilde\pi(n) : a\pi \to \pi$. We apply the discussion of \S\ref{sub:dgequiv} to $\pi=\pi_1=\pi_2$ and $\theta_a=\tilde\pi(n)$ and obtain an isomorphism $Y_t(\tilde\pi(n)):=Y_t(a,\theta_a) : aY_t(\pi) \to Y_t(\pi)$ of $G(F)$-representations for every $0 \leq t \leq r$. For any $g \in G(F)$ the action of $G(F)$ on $Y_t(\pi)$ provides an automorphism $Y_t(\pi(g)) : Y_t(\pi) \to Y_t(\pi)$ of complex vector spaces. 

\begin{lem} \label{lem:dg1}
    The composition $Y_t(\pi(g))\circ Y_t(\tilde\pi(n))$ is an automorphism of the $\C$-vector space $Y_t(\pi)$ that depends only on the product $gn \in \tilde G(F)$.
\end{lem}
\begin{proof}
    Let $g,g' \in G(F)$ and $n,n' \in \tilde N(F)$ be such that $gn=g'n'$. Let $y_t \in Y_t(\pi)$ be given by $y_t = \sum_J y_t(J)$ with $y_t(J) = x_t(J)\otimes z_J \in X_{P_J}(\pi)\otimes \Lambda_J$. Then $n=n'm_0$ where $m_0=(n')^{-1}g^{-1}g'n' \in M_0$. In particular, $n$ and $n'$ induce the same bijection $a$ of $S$. For $h \in G(F)$ we have
    \begin{eqnarray*}
        [(Y_t(\pi(g))Y_t(\tilde\pi(n)))(y_t)](aJ)(h) &=& [Y_t(\tilde\pi(n))y_t](aJ)(hg)\\
        &=&\tilde\pi(n)x_t(J)(n^{-1}hgn) \otimes \lambda_J(a)(z_J)\\
        &=&\tilde\pi(n'm_0)x_t(J)((n'm_0)^{-1}hg'n')\otimes\lambda_J(a)(z_J)\\
        &=&\tilde\pi(n')x_t(J)((n')^{-1}hg'n')\otimes\lambda_J(a)(z_J)\\
        &=&[Y_t(\tilde\pi(n'))y_t](aJ)(hg')\\
        &=&[(Y_t(\pi(g'))Y_t(\tilde\pi(n')))(y_t)](aJ)(h)\qedhere
    \end{eqnarray*}
    
\end{proof}

The above lemma allows us to define the automorphism $Y_t(\tilde\pi(\tilde g))$ of the $\C$-vector space $Y_t(\pi)$ for any $\tilde g \in \tilde G(F)$.

\begin{lem}
    For $\tilde g,\tilde g' \in \tilde G(F)$ we have $Y_t(\tilde\pi(\tilde g \tilde g')) = Y_t(\tilde\pi(\tilde g)) \circ Y_t(\tilde\pi(\tilde g'))$.
\end{lem}
\begin{proof}
    In the special case $\tilde g=n \in \tilde N(F)$ and $\tilde g'= n' \in \tilde N(F)$ the claim follows from Proposition \ref{pro:dgequiv}(4) applied to $b=\tx{Ad}(n)$, $a=\tx{Ad}(n')$, $\psi=\tilde\pi(n)$, $\theta = \tilde\pi(n')$.

    For the general case write $\tilde g=gn$ and $\tilde g'=g'n'$. Then $\tilde g\tilde g'=g(ng'n^{-1})nn'$. 
    \begin{eqnarray*}
        Y_t(\tilde\pi(\tilde g\tilde g'))&=&Y_t(\pi(g(ng'n^{-1})))\circ Y_t(\tilde\pi(nn'))\\
        &=&Y_t(\pi(g)) \circ Y_t(\pi(ng'n^{-1}))\circ Y_t(\tilde\pi(n))\circ Y_t(\tilde\pi(n'))\\
        &=&Y_t(\pi(g)) \circ Y_t(\tilde\pi(n))\circ Y_t(\pi(g'))\circ Y_t(\tilde\pi(n'))\\
        &=&Y_t(\tilde\pi(\tilde g))\circ Y_t(\tilde\pi(\tilde g')),
    \end{eqnarray*}
    where the first equation is by definition, the second uses the special case just proved, the third uses the fact that $Y_t(\tilde\pi(n))$ is a $G(F)$-equivariant isomorphism $aY_t(\pi) \to Y_t(\pi)$ according to Proposition \ref{pro:dgequiv}(1), and the fourth equation is again per definition.
\end{proof}

The above lemma endows the $\C$-vector space $Y_t(\pi)$ with the structure of a representation of the group $\tilde G(F)$. We denote this representation by $Y_t(\tilde\pi)$, so that $Y_t(\tilde\pi)(\tilde g) = Y_t(\tilde\pi(\tilde g))$.

\begin{lem}
    For any $\tilde g \in \tilde G(F)$ the automorphism $Y_t(\tilde\pi(\tilde g))$ of $Y_t(\tilde \pi)$ is functorial in $\tilde\pi$.
\end{lem}
\begin{proof}
    This follows from the functoriality in $\tilde\pi$ of $Y_t(\tilde\pi(n))$ and $Y_t(\pi(g))$ for all $n \in \tilde N(F)$ and $g \in G(F)$: for $n$ this is Proposition \ref{pro:dgequiv}(3), while for $g$ it stems from the fact that $Y_t(\pi)$ is built from functors of $G(F)$-representations.
\end{proof}

In other words, we now have a functor
\[ Y_t : \mc{R}(\tilde G) \to \mc{R}(\tilde G). \]
If we want to distinguish in notation this functor from the one for the group $G$, we will write $Y_t^{\tilde G}$ and $Y_t^G$, respectively. The following statement is immediate from the construction.

\begin{fct} \label{fct:xtext}
    $\tx{Res}\circ Y_t^{\tilde G} =  Y_t^G \circ \tx{Res}$.
\end{fct}

We now organize the representations $Y_t(\tilde\pi)$ for $0 \leq t \leq r$ into a complex.

\begin{lem} \label{lem:dtext}
    The differential $d_t$ translates $Y_t(\tilde\pi(\tilde g))$ to $Y_{t-1}(\tilde\pi(\tilde g))$.
\end{lem}
\begin{proof}
    It is enough to treat the special cases $\tilde g = g \in G(F)$ and $\tilde g = n \in \tilde N(F)$. The first case is simply the fact that the differential $d_t$ is equivariant for the $G(F)$-action on the complex $Y_\bullet$, because all constructions are performed using functors of $G(F)$-representations. The second case is Proposition \ref{pro:dgequiv}(2).
\end{proof}

\begin{pro} \label{pro:aubdisc1}
    The functors $Y_0^{\tilde G},\dots,Y_{t-1}^{\tilde G}$ vanish on $\mc{R}_t(\tilde G)$, and the sequence
    \[ 0 \to Y_r^{\tilde G} \to Y_{r-1}^{\tilde G} \to \dots \to Y_t^{\tilde G} \]
    is exact when restricted to $\mc{R}_t(\tilde G)$.
\end{pro}
\begin{proof}
This follows from Theorem \ref{thm:aubmain}, Fact \ref{fct:xtext}, and Lemma \ref{lem:dtext}.
\end{proof}

\begin{dfn} \label{dfn:aubdisc1}
    Define the functor $D_A^{\tilde G} : \mc{R}_t(\tilde G) \to \mc{R}_t(\tilde G)$ as 
    \[ D_A^{\tilde G} := \tx{cok}(d_{t+1} : Y_{t+1}^{\tilde G} \to Y_t^{\tilde G}). \]
\end{dfn}

\begin{lem} \label{lem:dgres}
    $\tx{Res}\circ D_A^{\tilde G} = D_A^G \circ \tx{Res}$.
\end{lem}
\begin{proof}
    This follows at once from Fact \ref{fct:xtext}.
\end{proof}

\begin{lem}
    If $\tilde \pi$ is irreducible, so is $D_A^{\tilde G}(\tilde \pi)$.
\end{lem}
\begin{proof}
    Write $\tilde\tau=D_A^{\tilde G}(\tilde\pi)$ and let $\tau=\tx{Res}(\tilde\tau)$, which by Lemma \ref{lem:dgres} equals $D_A^G(\pi)$, with $\pi=\tx{Res}(\tilde\pi)$. Since $G(F)$ is of finite index in $\tilde G(F)$, $\pi$ is semi-simple. Since $D_A^G$ respects direct sums and maps irreducibles to irreducibles, $\tau$ is also semi-simple. Again, the finite index of $G(F)$ in $\tilde G(F)$ implies that $\tilde \tau$ is semi-simple. It is therefore enough to show that $\tx{End}_{\tilde G(F)}(\tilde\tau)=\C$.
    
    Since $D_A^G$ is an equivalence of categories, it induces an isomorphism of complex vector spaces
    \[ D_A^G : \tx{End}_{G(F)}(\pi) \to \tx{End}_{G(F)}(\tau). \]
    Let $n \in \tilde N(F)$. Applying Proposition \ref{pro:dgequiv}(3) to $f_1=f$, $f_2=\tilde\pi(n)\circ f \circ \tilde\pi(n)^{-1}$, $\pi_1=\pi_2=\pi$, and $\theta=\theta'=\tilde\pi(n)$, we see that 
this isomorphism respects the action of $\tilde N(F)$ on both sides, hence the action of $\tilde G(F)/G(F)$. The irreducibility of $\tilde\pi$ implies that the set of vectors in $\tx{End}_{G(F)}(\pi)$ fixed by $\tilde G(F)/G(F)$ are precisely the scalar endomorphisms. This is transferred by the above isomorphism to $\tx{End}_{G(F)}(\tau)$ and implies that $\tx{End}_{G(F)}(\tau)^{\tilde G(F)/G(F)} = \C$. The latter equals $\tx{End}_{\tilde G(F)}(\tilde\tau)$, implying the irreducibility of $\tilde\tau$.
\end{proof}

\begin{pro} \label{pro:dgsc-disc}
    There is an isomorphism of functors $\tx{id} \to D_A^{\tilde G}$ on the factor $\mc{R}_r(\tilde G)$. In particular, $D_A^{\tilde G}(\tilde \pi) = \tilde\pi$ functorially for any supercuspidal representation $\tilde \pi \in \mc{R}(\tilde G)$.
\end{pro}
\begin{proof}
    On $\mc{R}_r(\tilde G)$ all $Y_t$ vanish except for $t=r$, and we have the isomorphism $\tx{id} \to Y_r$.
\end{proof}

\subsection{A character formula} \label{sub:cf}

Let $0 \leq t \leq r$ and $\tilde\pi \in \mc{R}_t(\tilde G)$. Let $\tilde g \in \tilde G(F)$ be a regular semi-simple element. Write $\tilde g=gn$ with $g \in G(F)$ and $n \in \tilde N(F)$ and let $a$ be the automorphism on the set $S$ of relative simple roots, hence also on the set of standard parabolic subgroups of $G$, induced by $\tx{Ad}(n)$. It depends only on $\tilde g$.

For an $a$-stable standard parabolic subgroup $P$ of $G$ the discussion of \S\ref{sub:dgequiv} applied to $\pi=\tilde\pi|_{G(F)}$ and the isomorphism $\theta_a=\tilde \pi(n) : a\pi \to \pi$ produces an automorphism $X_P(\tilde\pi(n)) := X_P(a,\theta_a) : aX_P(\pi) \to X_P(\pi)$. Let $X_P(\pi(g))$ be the action of $g$ on the representation $X_P(\pi)$. The automorphism $X_P(\pi(g))\circ X_P(\tilde\pi(n))$ of $X_P(\pi)$ depends only on $\tilde g$ (this follows from the proof of Lemma \ref{lem:dg1}, but will also be derived formally from its statement in the proof of Lemma \ref{lem:char1} below). This extends $X_P(\pi)$ to a representation of the subgroup of $\tilde G(F)$ generated by $G(F)$ and $\tilde g$.

\begin{pro} \label{pro:char}
    The character of $D_A^{\tilde G}(\tilde \pi)$ at $\tilde g$ is given by 
    \[ (-1)^{r-t}\sum_P (-1)^{\dim(A_P^a/A_G^a)} \cdot \tx{tr}(\tilde g|X_P(\pi)), \]
    where the sum runs over the set of those standard parabolic subgroups $P$ of $G$ that are stable under $a$, $A_P$ is the maximal split central torus of the standard Levi subgroup of $P$ and $\tx{tr}(\tilde g|X_P(\pi))$ is the value at $\tilde g$ of the Harish-Chandra character of the representation $X_P(\pi)$.
\end{pro}

\begin{lem} \label{lem:char1}
    The character of $D_A^{\tilde G}(\tilde \pi)$ at $\tilde g$ is given by 
    \[ (-1)^t\sum_J (-1)^{|J|}\cdot \tx{sgn}(a|S-J) \cdot \tx{tr}(\tilde g|X_{P_J}(\pi)), \]
    where the sum runs over the set of those subsets of $S$ that are stable under $a$, $\tx{sgn}(a|S-J) \in \{\pm 1\}$ is the sign of the permutation of $S-J$ induced by $a$, and $\tx{tr}(\tilde g|X_{P_J}(\pi))$ is the value at $\tilde g$ of the character of the representation $X_{P_J}(\pi)$.
\end{lem}
\begin{proof}
    By construction we have the exact sequence
    \[ 0 \to Y^{\tilde G}_r(\tilde\pi) \to \dots \to Y^{\tilde G}_t(\tilde\pi) \to D_A^{\tilde G}(\tilde\pi) \to 0, \]
    from which it follows that the value we are looking for is $(-1)^t\sum_{i=t}^r (-1)^i \chi_i(\tilde g)$, where $\chi_i$ is the character of $Y_i^{\tilde G}(\tilde\pi)$. Since the representations $Y^{\tilde G}_i(\tilde\pi)$ for $0 \leq i < t$ are zero, we can extend the sum over all $0 \leq i \leq r$.

    To compute $\chi_i(\tilde g)$ we write out 
    \[ Y^{\tilde G}_i(\tilde\pi) = \bigoplus_J X_{P_J}(\pi) \otimes \Lambda_J, \]
    where $J$ runs over the set of subsets of $S$ of cardinality $i$. Write $\tilde g = gn$ with $g \in G(F)$ and $n \in \tilde N(F)$ and let $a$ be the automorphism of $S$ induced by conjugation by $n$. We recall that the action of $\tilde g$ on $Y_i^{\tilde G}(\tilde\pi)$ sends the summand $X_{P_J}\otimes \Lambda_J$ to the summand $X_{P_{aJ}} \otimes \Lambda_{aJ}$. Therefore, if  $aJ \neq J$, then the summand indexed by $J$ does not contribute to the trace. We conclude that $\chi_i(\tilde g)$ is the value at $\tilde g$ of the character of 
    \[ \bigoplus_J X_{P_J}(\pi) \otimes \Lambda_J, \]
    where now $J$ runs over those subsets of $S$ that have cardinality $i$ and are stabilized by $a$. Each summand is now preserved by the action of $\tilde g$, so we can evaluate its character at $\tilde g$. Consider $J$ contributing to the sum and let $P=P_J$. By construction, the action of $\tilde g$ on $Y_i(\tilde\pi)$, restricted to $X_P(\pi)\otimes\Lambda_J$, is given by the tensor product of the automorphism $X_P(\pi(g))\circ X_P(\tilde\pi(n))$ of $X_P(\pi)$ and the automorphism $\lambda_J(a)$ of $\Lambda_J$. The latter is an automorphism of a 1-dimensional $\C$-vector space, hence a complex scalar, which by definition equals $\tx{sgn}(a|S-J)$. From lemma \ref{lem:dg1} we conclude that the former depends only on $\tilde g$, as was asserted above.
\end{proof}

\begin{lem} \label{lem:char2}
    Let $a$ be an automorphism of $S$ that preserves the subset $J \subset S$ and let $P$ be the standard parabolic associated to $J$. Then
    \[ \tx{sgn}(a|S-J) = (-1)^{\dim(A_P)-\dim(A_G)} \cdot (-1)^{\dim(A_P^a)-\dim(A_G^a)}.\]
\end{lem}
\begin{proof}
    We note that since $\dim(A_P)-\dim(A_G)=\dim(A_P/A_G)$ and $\dim(A_P^a)-\dim(A_G^a)=\dim((A_P/A_G)^a)$, both sides of this identity do not change if we replace $G$ by its adjoint quotient. Having done so, the set $S$ is a basis of $X^*(A_0)$, where $A_0$ is the maximal split central torus in the minimal Levi subgroup $M_0$, and the image of $S-J$ under the projection $X^*(A_0) \to X^*(A_P)$ is a basis of $X^*(A_P)$, whose elements are permuted by $a$. The set of $a$-orbits in $S-J$ forms a basis for the lattice of co-invariants $X^*(A_P)_a$, which also equals $X^*(A_P^a)$. An $a$-orbit of size $n$ contributes a factor of $(-1)^n$ to $(-1)^{\dim(A_P)}$, the factor $(-1)^1$ to $(-1)^{\dim(A_P^a)}$, and the factor $(-1)^{n-1}$ to $\tx{sgn}(a|S-J)$.
\end{proof}

\begin{proof}[Proof of Proposition \ref{pro:char}]
    We use Lemma \ref{lem:char1} and are left with showing that for an $a$-stable subset $J \subset S$ with associated standard parabolic subgroup $P$ we have $(-1)^{r-|J|} \cdot \tx{sgn}(a|S-J)=(-1)^{\dim(A_P^a/A_G^a)}$. But $r-|J|=|S-J|=\dim(A_P/A_G)$ and we can apply Lemma \ref{lem:char2}.
\end{proof}

We can now compare the effect of the functor $D_A^{\tilde G}$ on the Grothendieck group of $\mc{R}(\tilde G)$ with an automorphism of this Grothendieck group considered in \cite[Appendix A]{Xu17Canad}, where $\tilde G$ is denoted by $G^+$ and is assumed to equal $G \rtimes \<\theta\>$ with $G$ quasi-split over $F$ and $\theta$ an $F$-automorphism of $G$ preserving an $F$-pinning. In that case, Xu defines the automorphism
\[ \tx{inv}^{\theta}(\pi^+) = \sum_{P \in \mc{P}^\theta}(-1)^{\dim(A_P)_\theta}\tx{Ind}_P^G(\tx{Jac}_P\pi^+),\]
for $\pi^+$ any element of the Grothendieck group for $G^+(F)$. Here $\mc{P}^\theta$ is the set of $\theta$-stable standard parabolic subgroups of $G$ and $\tx{Ind}_P^G$ and $\tx{Jac}_P$ are the normalized parabolic induction and restriction functors defined by applying the usual definition to $\pi^+|_{G(F)}$ and equipping the result with the natural action of $G^+(F)$, as we have also done with our functor $X_P(\pi^+)$.

\begin{cor} \label{cor:xu}
    In the special case where $\tilde G=G \rtimes \<\theta\>$ with $G$ quasi-split over $F$ and $\theta$ an $F$-automorphism of $G$ preserving an $F$-pinning, the characters of $D_A^{\tilde G}(\pi^+)$ and $\tx{inv}^\theta(\pi^+)$ agree on the coset $G(F) \rtimes \theta$ up to multiplication by the sign
    \[ (-1)^{(r-t)+\dim(A_G^\theta)}, \]
    for any irreducible $\pi^+ \in \mc{R}_t(G^+)$.
\end{cor}
\begin{proof}
    This follows from Proposition \ref{pro:char} and the definition of $\tx{inv}^\theta(\pi^+)$.
\end{proof}

\subsection{A sign character and a variation of $D_A^{\tilde G}$} \label{sub:sign}

Using the group homomorphism $\tilde G(F)/G(F) \to \tx{Aut}(S)$ we obtain the character 
\begin{equation} \label{eq:eps}
    \epsilon = \epsilon_{\tilde G} : \tilde G(F)/G(F) \to \{\pm 1\},\qquad \epsilon(\tilde g) = \tx{sgn}(a|S)
\end{equation}
where again $a \in \tx{Aut}(S)$ is the automorphism induced by $\tilde g$. From Lemma \ref{lem:char2} we obtain
\begin{cor} \label{cor:eps} 
    $\epsilon(\tilde g) = (-1)^{\dim(A_0/A_G)}(-1)^{\dim(A_0^a/A_G^a)}$. 
\end{cor}

\begin{cor} \label{cor:dga'}
    If we use the Borel--Serre complex of \S\ref{sub:bs} to define a functor $'D_A^{\tilde G}$, then we have 
    \[ 'D_A^{\tilde G} = D_A^{\tilde G} \otimes \epsilon_{\tilde G}.\]
\end{cor}
\begin{proof}
    According to Fact \ref{fct:bs1} the complex $('X_t,{}'d_t)$ is functorially in $\pi$ isomorphic to the complex $(Y_t,d_t)$. Therefore, the two complexes produce the same functor $\mc{R}(G) \to \mc{R}(G)$. However, they are endowed with different actions of $\tilde G(F)$. More precisely, an element $n \in \tilde N(F)$ acts on the complex $Y$ by the automorphism $Y_t(\tilde\pi(n))=Y_t(\tx{Ad}(n),\tilde\pi(n))$, while it acts on the complex $'X$ by $'X_t(\tx{Ad}(n),\tilde\pi(n))$. But according to Fact \ref{fct:bs2}(5), the two actions differ by multiplication by $\tx{sgn}(a|S)=\epsilon(n)$.
\end{proof}

From Proposition \ref{pro:char} we obtain
\begin{cor} \label{cor:chareps}
    Let $0 \leq t \leq r$ and $\tilde\pi \in \mc{R}_t(\tilde G)$.
    The character of 
    \[ 'D_A^{\tilde G}(\tilde\pi) = D_A^{\tilde G}(\tilde\pi) \otimes \epsilon \] 
    at the regular semi-simple element $\tilde g \in \tilde G(F)$ is given by
    \[ (-1)^{t}\sum_P (-1)^{\dim(A_0^a/A_P^a)} \cdot \tx{tr}(\tilde g|X_P(\pi)), \]
    where the sum runs over the set of those standard parabolic subgroups $P$ of $G$ that are stable under $a$ and $\tx{tr}(\tilde g|X_P(\pi))$ is the value at $\tilde g$ of the character of the representation $X_P(\pi)$.
\end{cor}

\subsection{Parabolic subgroups of disconnected groups} \label{sub:pardisc}

Our next goal is to give a reinterpretation of $D_A^{\tilde G}$ in terms of generalizations to $\tilde G$ of the functors of parabolic induction and restriction. For this we need a notion of a ``parabolic subgroup'' of $\tilde G$.

Recall that a subgroup $P \subset G$ is called parabolic if it satisfies the following equivalent conditions:
\begin{enumerate}
    \item $P$ contains a Borel subgroup of $G$.
    \item $G/P$ is a complete variety.
\end{enumerate}
In this case $P$ is automatically connected and is equal to its normalizer in $G$.

Therefore, for a subgroup $\tilde P \subset \tilde G$ the following statements are equivalent.
\begin{enumerate}
    \item $\tilde P$ contains a Borel subgroup of $G$.
    \item $\tilde P \cap G$ is a parabolic subgroup of $G$
    \item $\tilde G/\tilde P$ is a complete variety.
\end{enumerate}

The following additional condition, which is not implied by the above conditions, can also be useful:

\begin{enumerate}[resume]
    \item $\tilde P = N_{\tilde G}(P)$, where $P=\tilde P \cap G$.
\end{enumerate}

\begin{dfn}
    A \emph{parabolic subgroup} of $\tilde G$ is a subgroup $\tilde P \subset \tilde G$ that satisfies the equivalent conditions 1.--3. above. This group is called \emph{wide} if it also satisfies condition 4. above, and \emph{narrow}, if $\tilde P \subset G$.
\end{dfn}

The following two statements follow from the fact that a parabolic subgroup of $G$ is equal to its own normalizer.

\begin{fct}
    A wide parabolic subgroup $\tilde P \subset \tilde G$ is equal to its normalizer in $\tilde G$.
\end{fct}

\begin{fct}
    The maps $P \mapsto N_{\tilde G}(P)$ and $\tilde P \mapsto G \cap \tilde P$ are mutually inverse $\tilde G$-equivariant bijections between the set of parabolic subgroups of $G$ and the set of wide parabolic subgroups of $\tilde G$.
\end{fct}

To classify more general parabolic subgroups of $\tilde G$ we fix a parabolic subgroup $P \subset G$ and consider those parabolic subgroups $\tilde P \subset \tilde G$ such that $\tilde P \cap G=P$. Note that $\pi_0(\tilde G)$ acts on the set of $G$-conjugacy classes of parabolic subgroups of $G$. We will write $\pi_0(\tilde G)_{[P]}$ for the stabilizer of the conjugacy class of $P$ under this action. We have the exact sequence
\begin{equation}
1 \to P \to N_{\tilde G}(P) \to \pi_0(\tilde G)_{[P]} \to 1,
\end{equation}
which in particular identifies the component group of the wide parabolic subgroup of $\tilde G$ corresponding to $P$ with $\pi_0(\tilde G)_{[P]}$. 

\begin{fct}
Every parabolic subgroup of $\tilde G$ that contains $P$ is contained in $N_{\tilde G}(P)$ and in this way the set of such subgroups is in bijection with the set of subgroups of $\pi_0(\tilde G)_{[P]}$. 
\end{fct}

Let us now consider a minimal parabolic pair $(M_0,P_0)$ of $G$, and let $A_0$ be the maximal split torus in the center of $M_0$. A parabolic subgroup of $G$ that contains $P_0$ is called standard. The set of standard parabolic subgroups of $G$ is in bijection with the set of subsets of the set $S \subset X^*(A_0)$ of simple relative roots. We write $P_I$ for the standard parabolic subgroup associated to $I \subset S$ under this bijection.

\begin{cor}
    The map $I \mapsto N_{\tilde G}(P_I)$ sets up a bijection between the set of $\tilde G(F)$-orbits of subsets of $S$ and the set of $\tilde G(F)$-conjugacy classes of wide parabolic subgroups of $\tilde G$.
\end{cor}

\begin{dfn}
    A parabolic subgroup $\tilde P \subset \tilde G$ is called \emph{standard} if it contains $P_0$.
\end{dfn}

\begin{rem}
    Let $\tilde P_0=N_{\tilde G}(P_0)$. The conjugacy under $G(F)$ of the minimal parabolic subgroups of $G$ implies that $\tilde P_0$ meets every component of $\tilde G$ that contains an $F$-point. Therefore, a parabolic subgroup of $\tilde G$ that contains $\tilde P_0$ must also necessarily meet all such components of $\tilde G$. However, there may well exist wide parabolic subgroups of $\tilde G$ that do not meet all such components.
\end{rem}

Recall the notation $\tilde N=N_{\tilde G}(M_0,P_0)$ for the normalizer in $\tilde G$ of the pair $(M_0,P_0)$.

\begin{lem}\label{lem:npg}
    Let $\tilde P \subset \tilde G$ be a standard parabolic subgroup of $\tilde G$ and let $P = \tilde P \cap G$ and $\tilde N_P = \tilde N \cap \tilde P$. The inclusion $\tilde N \to \tilde G$ induces an isomorphism $\tilde N(F)/\tilde N_P(F) \to \tilde G(F)/G(F)\tilde P(F)$.
\end{lem}
\begin{proof}
Using $\tilde G(F)=G(F)\tilde N(F)$ we are reduced to computing the intersection $\tilde N(F) \cap G(F)\tilde P(F)$, i.e. the normalizer of $(M_0,P_0)$ in $G(F)\tilde P(F)$. If $g\tilde p$ normalizes $(M_0,P_0)$ then $(M_0,P_0)^g$ is a minimal parabolic pair in $P=\tilde P \cap G$, hence equal to $^p(M_0,P_0)$ for some $p \in P(F)$ by \cite[Theorem C.2.5]{CGP15} (note that this theorem is stated for pseudo-parabolic subgroups, but in the setting of a parabolic subgroup of a reductive group, the concepts of parabolic and pseudo-parabolic subgroups agree). Thus $gp \in \tilde N(F) \cap G(F)=M_0(F)$. Now $g\tilde p=gp(p^{-1})\tilde p \in \tilde P(F)$ and we conclude that the normalizer of $(M_0,P_0)$ in $G(F)\tilde P(F)$ lies in $\tilde P(F)$, hence equals $\tilde N_P(F)$.
\end{proof}

\begin{rem}
    There is an alternative approach to classifying the parabolic subgroups of $\tilde G$, or, more precisely, the subgroups of $\tilde G(F)$ that are of the form $\tilde P(F)$ for parabolic subgroups $\tilde P \subset \tilde G$. The tuple $(\tilde G(F),P_0(F),\tilde N(F),R)$, where $R$ is the set of reflections along the simple relative roots $S$, is a generalized Tits system in the sense of Iwahori, cf. \cite{Iwahori65}. Basic properties of generalized Tits systems provide a classification of the subgroups of $\tilde G(F)$ of the form $\tilde P(F)$. The reader may consult the end of Section 1.4 in the Erratum to \cite{BTBOOK}, in particular Lemma 1.4.20 there.
\end{rem}

\begin{dfn} Let $\tilde P \subset \tilde G$ be a parabolic subgroup. A \emph{Levi factor} of $\tilde P$ is a subgroup $\tilde M \subset \tilde P$ such that 
    \begin{enumerate}
        \item The group $M = P \cap \tilde M$ is a Levi factor of $P$.
        \item $\tilde M = N_{\tilde P}(M)$.
    \end{enumerate}
\end{dfn}

\begin{lem}
    \begin{enumerate}
        \item Given a parabolic subgroup $\tilde P$ of $\tilde G$ with $P=G \cap \tilde P$, the maps $M \mapsto N_{\tilde P}(M)$ and $\tilde M \mapsto P \cap \tilde M$ are mutually inverse, $\tilde P$-equivariant bijections between the set of Levi factors of $P$ and the set of Levi factors of $\tilde P$.
        \item The unipotent radical $U$ of $P$ is a closed normal subgroup of $\tilde P$ and is a complement to any Levi factor $\tilde M$ of $\tilde P$. In particular, the map $\tilde M \to \tilde P \to \tilde P/U$ is an isomorphism.
    \end{enumerate}
\end{lem}
\begin{proof}
    (1) follows from $N_P(M)=M$ and (2) follows from the conjugacy of Levi factors of $P$ under $U$.
\end{proof}

\begin{dfn} \label{dfn:paropdisc}
    Let $\tilde P$ be a parabolic subgroup of $\tilde G$ and let $\tilde M$ be a Levi factor of $\tilde P$. Define the $\tilde M$-\emph{opposite} parabolic subgroup $\tilde P^-$ as $\tilde P^-=\tilde M \cdot U^-$, where  $P^-=MU^-$ is the parabolic subgroup of $G$ that is $M$-opposite to $P$, $P=\tilde P \cap G$, $M=\tilde M \cap G$, $P=MU$.
\end{dfn}

Thus $\tilde P^- = \tilde M U^-$ is a parabolic subgroup of $\tilde G$ with Levi factor $\tilde M$ and we have $\tilde P \cap \tilde P^- = \tilde M$.

\subsection{Parabolic induction and restriction in the disconnected setting} \label{sub:parindresdisc}

The functors $i_P^G$ and $r_P^G$ can be generalized to the disconnected case to functors
\[ r_{\tilde P}^{\tilde G} : \mc{R}(\tilde G) \leftrightarrow \mc{R}(\tilde M) : i_{\tilde P}^{\tilde G},\]
where $\tilde P \subset \tilde G$ is a parabolic subgroup and $\tilde M$ stands for the quotient $\tilde P/U$. For this, we let $\delta_{\tilde P}$ denote the modulus character of the locally compact group $\tilde P(F)$; it descends to $\tilde M(F)$. Given a representation $\tilde\pi \in \mc{R}(\tilde G)$ we form as before the vector space of coinvariants $(V_{\tilde\pi})_{U(F)}$ and endow it with the action $\delta_{\tilde P}^{-1/2}\otimes \tilde\pi$ of $\tilde M(F)$. Given a representation $\tilde\sigma \in \mc{R}(\tilde M)$ we inflate it to $\tilde P(F)$ and form the smooth (equivalently compact) induction $\tx{Ind}_{\tilde P(F)}^{\tilde G(F)}(\delta_{\tilde P}^{1/2} \otimes \tilde\sigma)$.

The following statement is proved in the same way as in the connected case.
\begin{fct} \label{fct:adjoint}
    The functors $r_{\tilde P}^{\tilde G}$ and $i_{\tilde P}^{\tilde G}$ form an adjoint pair. More precisely, the maps
    \[ \Psi_{\tilde\pi,\tilde\sigma} : \tx{Hom}_{\tilde G(F)}(\tilde \pi,i_{\tilde P}^{\tilde G}(\tilde \sigma)) \leftrightarrow \tx{Hom}_{\tilde M(F)}(r_{\tilde P}^{\tilde G}(\tilde \pi),\tilde \sigma) : \Phi_{\tilde\pi,\tilde\sigma} \]
    given by $\Psi_{\tilde\pi,\tilde\sigma}(\alpha) = \tx{ev}_1 \circ \alpha$ and $\Phi_{\tilde\pi,\tilde\sigma}(\beta)(v)=(\tilde g \mapsto \beta(\tilde gv))$ are mutually inverse isomorphisms of abelian groups, functorial in $\tilde\sigma$ and $\tilde\pi$.
\end{fct}

\begin{lem} \label{lem:iprpres}
    \begin{enumerate}
        \item $\tx{Res}^{\tilde M}_M \circ r_{\tilde P}^{\tilde G} = r_P^G \circ \tx{Res}^{\tilde G}_G$.
        \item $\tx{Res}^{\tilde G}_G \circ i_{\tilde P}^{\tilde G} = \bigoplus\limits_{c \in \tilde G(F)/G(F)\tilde P(F)} c \cdot (i_P^G \circ \tx{Res}^{\tilde M}_M)$.
    \end{enumerate}
\end{lem}
\begin{proof}
    (1) follows from the definition. (2) reduces by the Mackey formula to the case when $\tilde G=G \cdot \tilde P$, in which case it again follows from the definition.
\end{proof}

\begin{cor} \label{cor:iprpdcmp}
    For any $0\leq t\leq r_M\leq r$, where $r_M$ is the rank of the derived subgroup of $M$, the functors $i^{\tilde G}_{\tilde P}$ and $r^{\tilde G}_{\tilde P}$ restrict to the following two functors:
    \[i^{\tilde G}_{\tilde P}: \mc{R}_t(\tilde M)\to \mc{R}_t(\tilde G), \qquad r^{\tilde G}_{\tilde P}: \mc{R}_t(\tilde G)\to \mc{R}_t(\tilde M), \]
    and they still form an adjoint pair.
\end{cor}
\begin{proof}
    The second assertion follows from the first one, fact \ref{fct:adjoint} and the Bernstein decomposition for $\tilde G$ and $\tilde M$. The first assertion follows from Lemma \ref{lem:iprpres}, the definition of $\mc{R}_t(\tilde G)$ and $\mc{R}_t(\tilde M)$ and \cite[VI 7.3 Proposition]{Ren10}.
\end{proof}

The induction and restriction in stages isomorphisms hold in this context as well. Let $\tilde P \subset \tilde Q \subset \tilde G$ be parabolic subgroups, with unipotent radicals $U \subset \tilde P$ and $V \subset \tilde Q$ and Levi quotients $\tilde M=\tilde P/U$ and $\tilde L=\tilde Q/V$. Then $V \subset U$ and $\tilde P/V \subset \tilde L$ is a parabolic subgroup of $\tilde L$ with unipotent radical $U/V$ and Levi quotient $(\tilde P/V)/(U/V)=\tilde M$. We have the mutually inverse, functorial isomorphisms
\begin{equation} \label{eq:tipg_stages}
i_{\tilde P}^{\tilde G}\tilde\sigma \leftrightarrow i_{\tilde Q}^{\tilde G}i_{\tilde P/V}^{\tilde L} \tilde\sigma
\end{equation}
relating an element $\tilde f : \tilde G \to V_{\tilde\sigma}$ of $i_{\tilde P}^{\tilde G}\tilde\sigma$ and an element $\tilde f' : \tilde G \times \tilde L \to V_{\tilde \sigma}$ of $i_{\tilde Q}^{\tilde G}i_{\tilde P/V}^{\tilde L}\tilde \sigma$ by $\tilde f(g)=\tilde f'(g,1)$ and $\tilde f'(g,l)=\delta_{\tilde Q}^{-1/2}(l)\tilde f(qg)$, where $g \in \tilde G(F)$ and $q \in \tilde Q(F)$ with image $l \in \tilde L(F)$. Moreover, we have the mutually inverse isomorphisms
\begin{equation} \label{eq:trpg_stages}
r_{\tilde P}^{\tilde G}\tilde \pi \leftrightarrow r_{\tilde P/V}^{\tilde L}r_{\tilde Q}^{\tilde G}\tilde \pi
\end{equation}
given by the natural identification $(V_{\tilde \pi})_{U(F)}=((V_{\tilde \pi})_{V(F)})_{(U/V)(F)}$.

There are also other types of stages isomorphisms. Let $\tilde G_1$ be an algebraic subgroup of $\tilde G$ containing $G$. Let $\tilde P=\tilde M U$ be a parabolic subgroup of $\tilde G_1$ and hence of $\tilde G$. We have a functorial isomorphism for $\tilde \pi\in \mc{R}(\tilde M)$
\begin{equation} \label{eq:Indip}
    \tx{Ind}^{\tilde G}_{\tilde G_1}i^{\tilde G_1}_{\tilde P} \tilde \pi \to i^{\tilde G}_{\tilde P}\tilde \pi
\end{equation}
sending an element $f: \tilde G(F)\to V_{i^{\tilde G_1}_{\tilde P}\tilde \pi}$ of $\tx{Ind}^{\tilde G}_{\tilde G_1}i^{\tilde G_1}_{\tilde P} \tilde \pi$ to $f': \tilde G(F)\to V_{\tilde \pi}, g\mapsto f(g)(1)$.
Also, suppose $P\subset \tilde P_1 \subset \tilde P_2$ are parabolic subgroups of $\tilde G$, with Levi factors $M\subset \tilde M_1\subset \tilde M_2$ respectively. Suppose $\tilde P_1\cap G=\tilde P_2\cap G=P$, and hence $\tilde M_1$ has finite index in $\tilde M_2$. We have functorial isomorphism for $\tilde \sigma\in \mc{R}(\tilde M_1)$
\begin{equation} \label{eq:ipInd}
    i^{\tilde G}_{\tilde P_2}\tx{Ind}^{\tilde M_2}_{\tilde M_1}\tilde \sigma \to i^{\tilde G}_{\tilde P_1}\tilde \sigma
\end{equation}
sending an element $f: \tilde G(F) \to V_{\tx{Ind}\tilde \sigma}$ of $i^{\tilde G}_{\tilde P_2}\tx{Ind}^{\tilde M_2}_{\tilde M_1}\tilde \sigma $ to $f': \tilde G(F)\to V_{\tilde \sigma}, g\mapsto f(g)(1)$. Note that $\delta_{\tilde P_2}(m) = \delta_{\tilde P_1}(m)$ for any $m\in \tilde M_1(F)$.

Similarly, we may define the contragredient functor $(-)^\vee: \mathcal{R}(\tilde G) \to \mathcal{R}(\tilde G)$, which sends a representation $\tilde\pi$ to its smooth dual $\tilde\pi^\vee$. 
\begin{fct} \label{fct:conprp}
    \begin{enumerate}
        \item $\tx{Res}^{\tilde G}_G \circ (-)^\vee = (-)^\vee \circ \tx{Res}^{\tilde G}_G.$
        \item $\tx{Ind}^{\tilde G}_{\tilde G_1}\circ (-)^\vee = (-)^\vee \circ \tx{Ind}^{\tilde G}_{\tilde G_1}$ for any subgroup $\tilde G_1$ of $\tilde G$ containing $G$.
        \item $((-)^\vee)^\vee$ is equivalent to the identity functor on $\mc{R}(\tilde G)$.
        \item $(-)^\vee$ respects the Bernstein decomposition in the sense that $\tilde \pi^\vee\in \mc{R}_t(\tilde G)$ if $\tilde \pi\in\mc{R}_t(\tilde G)$.
    \end{enumerate}
\end{fct}

Moreover, we may recover the second adjointness via the contragredient functor, which is based on the following lemma.

\begin{lem}\label{lem:iprpcon}
    \begin{enumerate}
        \item $r^{\tilde G}_{\tilde P}\circ (-)^{\vee}$ is equivalent to $(-)^{\vee}\circ r^{\tilde G}_{\tilde P^-}$.
        \item $i^{\tilde G}_{\tilde P}\circ (-)^{\vee}$ is equivalent to $(-)^{\vee}\circ i^{\tilde G}_{\tilde P}$.
    \end{enumerate}
\end{lem}
\begin{proof}
    For (1), we know that $\tx{Res}^{\tilde G}_G\circ r^{\tilde G}_{\tilde P} = r^G_P\circ \tx{Res}^{\tilde G}_G$ and similar for the contragredient functor. Therefore, by the classical result that $r^G_P\circ (-)^{\vee}$ is naturally isomorphic to $(-)^{\vee}\circ r^G_{P^-}$, we obtain a functorial isomorphism of $\mathbb{C}$-vector spaces
    \[F_{\tilde \pi}: r^{\tilde G}_{\tilde P}(\tilde \pi^\vee) \to (r^{\tilde G}_{\tilde P^-}\tilde \pi)^\vee,\]
    which respects the action of $M(F)$, where $M=\tilde M \cap G$. It remains to check that it respects the action of $\tilde M(F)$. Let $\tilde N_M$ be the normalizer of $(M_0, P_0\cap M)$ in $\tilde M$. (We may assume $M_0\subset M$) Then, as in \S\ref{sub:defdgdis}, we have $\tilde M(F)=M(F)\tilde N_M(F)$ and we are reduced to checking the equivariance under $\tilde N_M(F)$. Let $a=\tx{Ad}(n)$ for a fixed $n\in \tilde N_M(F)$, then we have an isomorphism of $G(F)$-representation $\tilde{\pi}(n): a\pi\to \pi$, where $\pi=\tilde \pi|_G$. Note that $aP=nPn^{-1}=n(\tilde P\cap G)n^{-1}=\tilde P\cap G=P$, and similarly $aP^-=P^-$. Hence we have an isomorphism of $M(F)$-representations
    \[ar^G_P(\pi)=r^G_{aP}(a\pi)= r^G_P(a\pi)\xrightarrow{r^G_P(\tilde{\pi}(n))} r^G_P(\pi),\]
    which sends $v+V_\pi(U)$ to $\tilde{\pi}(n)(v)+V_\pi(U)$. Since it is clear that $(a\pi)^\vee= a(\pi^\vee)$, the $G(F)$-equivariant isomorphism $\tilde \pi^\vee(n): a\pi^\vee \to \pi^\vee$ equals the following composition
    \[ a\pi^\vee =(a\pi)^\vee \xrightarrow{(\tilde{\pi}(n)^{-1})^\vee} \pi^\vee.\]
    We then have the following diagram of $M(F)$-equivariant isomorphisms
    \[ \xymatrix{
        ar^G_P(\pi^\vee)=r^G_P(a\pi^\vee) \ar[d]_{F_{\tilde \pi}} \ar[rr]^-{r^G_P(\tilde{\pi}^\vee(n))} && r^G_P(\pi^\vee) \ar[d]^{F_{\tilde \pi}} \\
        a((r^G_{P^-}\pi)^\vee)= (r^G_{P^-}a\pi)^\vee \ar[rr]^-{(r^G_{P^-}(\tilde{\pi}(n)^{-1}))^\vee} && (r^G_{P^-}\pi)^\vee
    } \]
    The top horizontal arrow $r^G_P(\tilde \pi^\vee(n))$ equals $r^G_P((\tilde \pi(n)^{-1})^\vee)$. Hence the diagram commutes by the functoriality of $F_{\tilde \pi}$. As both horizontal arrows are the action of $n$, we conclude that $F_{\tilde \pi}$ is $\tilde M(F)$-equivariant.
     
    For (2), we use the proof as in the classical case. Let $K\subset G(F)$ be a special maximal compact subgroup in good position with respect to $P$. Then we have $G(F)=P(F)K$ by Iwasawa decomposition, see \cite{BTBOOK}. Consider the subgroup $\tilde G_1=\tilde P\cdot G$ of $\tilde G$. It is an algebraic subgroup with finite index. Suppose we could show $i^{\tilde G_1}_{\tilde P}\circ (-)^{\vee}$ is naturally isomorphic to $(-)^{\vee}\circ i^{\tilde G_1}_{\tilde P}$. Then we can use \eqref{eq:Indip} and fact \ref{fct:conprp} to get the desired result. Hence we may assume $\tilde G=\tilde P\cdot G$; in particular, $\tilde G(F)=\tilde P(F)G(F)$. Since $G(F)$ is unimodular, so is $\tilde G(F)$. Choose Haar measures $dg$ on $\tilde G(F)$, $dk$ on $K$ and a left Haar measure $d_lp$ on $\tilde P(F)$ such that the following condition holds
    \[\int_{\tilde G(F)}f(g) dg = \int_{\tilde P(F)\times K}f(pk)d_lpdk\]
    for any $f\in \mc{C}^\infty_c(\tilde G(F))$. Define the following pairing:
    \[i^{\tilde G}_{\tilde P}(\tilde \pi) \times i^{\tilde G}_{\tilde P}( \tilde \pi^\vee) \xrightarrow{(-,-)} \mathbb{C}\]
    with 
    \[(\varphi, \varphi^\vee):=\int_K\langle \varphi(k), \varphi^\vee(k) \rangle dk \]
    where $\langle, \rangle$ is the pairing between $\tilde \pi$ and $\tilde \pi^\vee$. This gives a perfect $\tilde G(F)$-pairing, see e.g. \cite[Proposition 8.2.3]{GetzHahn2024}.
\end{proof}

\begin{cor}[Second adjointness] \label{cor:secadjoint}
    We have the following isomorphism functorial in $\tilde\pi \in \mc{R}(\tilde G)$ and $\tilde\sigma \in \mc{R}(\tilde M)$
    \[\tx{Hom}_{\tilde M(F)}(\tilde \sigma, r^{\tilde G}_{\tilde P}(\tilde\pi))\leftrightarrow \tx{Hom}_{\tilde G(F)}(i^{\tilde G}_{\tilde P^-}(\tilde \sigma), \tilde\pi).\]
\end{cor}
\begin{proof}
    This is a direct consequence of Lemma \ref{lem:iprpcon}, fact \ref{fct:adjoint} and \ref{fct:conprp}(3).
\end{proof}

\subsection{Reinterpretation of $D_A^{\tilde G}$ in terms of induction and restriction} \label{sub:rp}

We will now provide an alternative description of $D_A^{\tilde G}$ using a different complex $'Y_\bullet^{\tilde G}$, whose definition is analogous to that of $Y_\bullet^G$ for the connected group $G$, and is based on the functors introduced in the preceding subsection. We will then show that the two complexes $'Y_\bullet^{\tilde G}$ and $Y_\bullet^{\tilde G}$ are functorially isomorphic. To save notation, we will write $'\tilde Y_\bullet$ and $\tilde Y_\bullet$ instead, where the $'\tilde Y_\bullet$ will be defined below, while $\tilde Y_\bullet$ is the complex defined in \S\ref{sub:defdgdis}.

Let $\tilde P$ be a parabolic subgroup of $\tilde G$ with reductive quotient $\tilde M$. We define the functor 
\[ X_{\tilde P} : \mc{R}(\tilde G) \to \mc{R}(\tilde G) \]
as the composition $i_{\tilde P}^{\tilde G}\circ r_{\tilde P}^{\tilde G}$, equivalently $I_{\tilde P}^{\tilde G} \circ R_{\tilde P}^{\tilde G}$. In analogy with $X_P$, we have
\[ X_{\tilde P}(\tilde\pi) = \{\tilde f : \tilde G(F) \to (V_{\tilde\pi})_{U(F)}\,|\, \tilde f(\tilde p\tilde g) = \tilde\pi(\tilde p)\tilde f(\tilde g)\}. \]
We now want to define the analog $Y_{\tilde P_J}(\tilde\pi)$ of $Y_{P_J}(\pi)$. While in the connected setting we defined $Y_{P_J}(\pi)=X_{P_J}(\pi)\otimes_\C \Lambda_J$, where $\Lambda_J=\Wedge^{|S-J|}\C^{S-J}$, here it is more convenient to define
\[ '\tilde Y_J(\tilde\pi) := {}'Y_{\tilde P_J}(\tilde \pi) :=  i_{\tilde P_J}^{\tilde G}( r_{\tilde P_J}^{\tilde G}(\tilde\pi) \otimes \Lambda_J), \]
where $\Lambda_J$ is the 1-dimensional representation of $\tilde M_J(F)$ given via the isomorphism $\tilde M_J(F)/M_J(F) \cong \tilde N(F)_J/M_0(F)$ and the action of $\tilde N(F)_J/M_0(F)$ on $\Lambda_J$ coming from the natural permutation action of $\tilde N(F)_J/M_0(F)$ on $S$ which preserves $J$; we have used the notation $\tilde N(F)_J$ for the stabilizer of $J$ in $\tilde N(F)$.

\begin{lem} \label{lem:translate}
    Let $J \subset S$. Let $P_J$ be the corresponding parabolic subgroup of $G$ and let $\tilde P_J = N_{\tilde G}(P_J)$ be the corresponding wide parabolic subgroup of $\tilde G$. 
    \begin{enumerate}
        \item The map 
        \[ '\tilde Y_J(\tilde \pi) \to \bigoplus_{c \in  \tilde N(F)/\tilde N(F)_J} Y_{cJ}(\pi),\qquad \tilde f \mapsto (f_c) \]
        is a $\tilde G(F)$-equivariant isomorphism of $\C$-vector spaces. Here we define 
        \[ f_c(g) = (\tilde\pi(c) \otimes \lambda_J(c))\tilde f(c^{-1}g) \]
        and the action of $\tilde G(F)$ on the right hand side is defined by 
        \[ (gf)_c(h) = f_c(hg),\quad (nf)_c(h) = (\tilde\pi(n) \otimes \lambda_{n^{-1}cJ}(n))f_{n^{-1}c}(n^{-1}hn) \]
        for $g,h \in G(F)$, $n \in \tilde N(F)$.
        \item For $n \in \tilde N(F)$ the map
        \[ '\tilde Y_{J}(\tilde\pi) \to '\tilde Y_{{nJ}}(\tilde\pi),\qquad \tilde f \mapsto {^n}\tilde f, \qquad {^n}\tilde f(\tilde g)=(\tilde\pi(n)\otimes\lambda_{J}(n))\tilde f(n^{-1}\tilde g)\]
        is a $\tilde G(F)$-equivariant isomorphism that, under the isomorphisms of 1. is translated to the identity $\bigoplus_{c \in \tilde N(F)/\tilde N(F)_J} Y_{{cJ}}(\pi)=\bigoplus_{c' \in \tilde N(F)/\tilde N(F)_{nJ} } Y_{{c'nJ}}(\pi)$.
    \end{enumerate}
\end{lem}
\begin{proof}
Direct calculation.
\end{proof}

We continue with $\tilde\pi \in \mc{R}(\tilde G)$ and $\pi=\tx{Res}^{\tilde G}_G(\tilde\pi)$. Consider $I,J \subset S$ with $|J|=|I|+1$. We are not requiring $I \subset J$. We have the standard parabolic subgroups $P_I$ and $P_J$ and the standard wide parabolic subgroups $\tilde P_I$ and $\tilde P_J$. Note that $\tilde P_I$ need not be contained in $\tilde P_J$ even if $P_I \subset P_J$. Recall the map $\varphi^J_I : X_J(\pi) \to X_I(\pi)$ is given by composing $f \in  X_J(\pi)$ with the natural projection map $[-]_{U_I}: (V_\pi)_{U_J(F)} \to (V_\pi)_{U_I(F)}$, and the map $\psi^J_I : Y_J(\pi) \to Y_I(\pi)$ is defined as $\varphi^J_I \otimes \epsilon^J_I$.

We define $'\tilde\psi^J_I : {}'\tilde Y_{J}(\tilde\pi) \to {}'\tilde Y_{I}(\tilde \pi)$ by
\[ '\tilde\psi^J_I(\tilde f)(\tilde g) = \sum_{\substack{n \in \tilde N(F)/\tilde N_J(F) \\  I \subset nJ}} \big([-]_{U_I(F)}\circ \tilde\pi(n)\big) \otimes \big( \epsilon^{nJ}_I\circ \lambda_J(n)\big)\tilde f(n^{-1}\tilde g). \]
Define
\[ '\tilde Y_t(\tilde\pi) = \bigoplus_{\substack{I \in \mc{P}(S)/\tilde N(F)\\ |I|=t}} {}'\tilde Y_{I}(\tilde\pi) \]
and 
\[ 'd_t : {}'\tilde Y_t(\tilde\pi) \to {}'\tilde Y_{t-1}(\tilde \pi) \]
as the sum of the maps $'\tilde\psi^J_I$ where $I \subset S$ has cardinality $t-1$, $J \subset S$ has cardinality $t$, and both $I$ and $J$ are independently taken as representatives for the $\tilde N(F)$-orbits in the power set of $S$.

\begin{lem} \label{lem:newcomplex}
    The isomorphisms of Lemma \ref{lem:translate} splice together to isomorphisms $\xi_t : {}'\tilde Y_t \to \tilde Y_t$ that are functorial in $\tilde\pi$ and intertwine the differentials $'d_t$ and $d_t$.
\end{lem}
\begin{proof}
    This is an immediate computation.

\end{proof}

\begin{cor} \label{cor:dgdisc-ri}
    The functors $'\tilde Y_0,\dots,{'\tilde Y_{t-1}}$ vanish on $\mc{R}_t(\tilde G)$, and the sequence
    \[ 0 \to {'\tilde Y_r} \to {'\tilde Y_{r-1}} \to \dots \to {'\tilde Y_t} \to D_A^{\tilde G} \to 0 \]
    is exact when restricted to $\mc{R}_t(\tilde G)$. 
\end{cor}
\begin{proof}
This follows from Proposition \ref{pro:aubdisc1}, Lemma \ref{lem:newcomplex}, and Definition \ref{dfn:aubdisc1}.
\end{proof}

\subsection{Cohomological duality for disconnected groups}

\begin{lem}
    The category of smooth representations of $\tilde G(F)$ has finite homological dimension. In addition, if $\pi$ is a finitely generated $\tilde G(F)$ representation, it has a finite finitely generated projective resolution.
\end{lem}
\begin{proof}
    In the connected case, this is proved in \cite[Theorem 29]{BernsteinPadicGroups} and \cite[Lemma 2.2]{YS2026M}. Since $G(F)$ has finite index in $\tilde G(F)$, it is also true for $\tilde G$.
\end{proof}

Let $\mc{D}^b(\tilde G)$ be the bounded derived category of smooth representations of $\tilde G$. We have the $\tilde G$-bimodule $\mc{\tilde H} := \mc{C}^\infty_c(\tilde G(F))$ and for any $Z \in \mc{D}^b(\tilde G)$ we can consider $\tx{RHom}_{\mc{\tilde H}}(Z,\mc{\tilde H})$. This is a right $\tilde G$-module that can be converted to a left $\tilde G$-module in the standard way: $gx:=xg^{-1}$. Write $D^{\tilde G}_\tx{coh} : \mc{D}^b(\tilde G) \to \mc{D}^b(\tilde G)$. This makes sense by the above lemma.

\begin{lem} \label{lem:homhGres}
    We have a functorial isomorphism of $G(F)$-representations (not necessarily of finite length):
    \[\tx{Res}^{\tilde G}_G(\tx{Hom}_{\tilde G (F)}(\tilde \pi, \mc{\tilde H})) \xrightarrow{\simeq} \tx{Hom}_{G(F)}(\tx{Res}^{\tilde G}_G\tilde \pi, \mc{H})\]
    where $\mc{H}$ is the Hecke algebra of $G(F)$.
\end{lem}
\begin{proof}
    Given a smooth representation $\tilde \pi$ of $\tilde G(F)$, write $\pi= \tx{Res}^{\tilde G}_G\tilde \pi$ as usual. Define the following maps
    \[R: \tx{Hom}_{\tilde G(F)}(\tilde \pi, \mc{\tilde H}) \to \tx{Hom}_{G(F)}(\pi, \mc{H}) \]
    sending $\varphi: v\mapsto \varphi_v$ to $R(\varphi): v\mapsto \varphi_v|_{G(F)}$, and 
    \[L: \tx{Hom}_{G(F)}(\pi, \mc{H})\to \tx{Hom}_{\tilde G(F)}(\tilde \pi, \mc{\tilde H}) \] 
    sending $\psi: v\mapsto \psi_v$ to $L(\psi): v\mapsto \bar{\psi}_v$, where $\bar{\psi}_v\in \mc{\tilde H}$ is defined as follows. Choose a coset decomposition $\tilde G(F)= \coprod_{i=1}^ng_iG(F)$. Then, for $x\in \tilde G(F)$,
    \[\bar{\psi}_v(x) = \sum_{i=1}^n \chi_{g_iG(F)}(x) \psi_{g_i^{-1}v}(g_i^{-1}x) \]
    where $\chi_{g_iG(F)}$ is the characteristic function on $g_iG(F)$ and hence the sum makes sense. It is easy to check that $\chi_{g_iG(F)}(-)\psi_{g_i^{-1}v}(g_i^{-1}-)$, as a function from $g_iG(F)$ to $\mathbb{C}$, is independent of the choice of the representative $g_i$, and that $\bar{\psi}_v$ is locally constant and compactly supported. It is easy to check that both $L$ and $R$ are $G(F)$-equivariant and inverse to each other.
\end{proof}

As a direct consequence of Lemma \ref{lem:homhGres} and the fact that $\tx{Res}^{\tilde G}_G$ preserves projective (because it has exact right adjoint), we have the following lemma.

\begin{lem} \label{lem:cohres}
    $\tx{Res}^{\tilde G}_G \circ D^{\tilde G}_\tx{coh}  = D^G_\tx{coh} \circ \tx{Res}^{\tilde G}_G$. 
\end{lem}

\begin{cor} \label{cor:dcoh}
    Let $0\leq t\leq r$ and $\tilde \pi \in \mc{R}_t(\tilde G)$. Then the complex $\tx{RHom}(\tilde \pi, \mc{\tilde H})$ has non-zero cohomology in a single degree $d(t)$ depending only on $t$. In addition, this non-zero cohomology lies in $\mc{R}_t(\tilde G)$.
\end{cor}
\begin{proof}
    Both assertions follow from Lemma \ref{lem:cohres}, Theorem \ref{thm:bermain} and the definition of $\mc{R}_t(\tilde G)$.
\end{proof}

\begin{dfn}
    Define the functor $D_B^t: \mc{R}_t(\tilde G) \to \mc{R}_t(\tilde G)$ as $\tx{H}^{d(t)}(D^{\tilde G}_{\tx{coh}}(-))$. Define the cohomological duality
    \[D_B^{\tilde G}:=\prod_tD_B^t: \mc{R}(\tilde G) \to \mc{R}(\tilde G)\]
\end{dfn}

Note that $D_B^{\tilde G}$ is a contravariant functor. In this subsection, we prove that the properties in Proposition \ref{pro:dhprp} for the connected case are also true for the disconnected case.

\begin{lem} \label{lem:dhres}
    $\tx{Res}^{\tilde G}_G\circ D_B^{\tilde G} = D_B^G\circ \tx{Res}^{\tilde G}_G$. In particular, $D_B^{\tilde G}$ is exact.
\end{lem}
\begin{proof}
    The first equality follows from Lemma \ref{lem:cohres} and the definitions of cohomological dualities. The second statement follows from the exactness of $D_B^G$ and $\tx{Res}^{\tilde G}_G$.
\end{proof}

\begin{pro} \label{pro:dhinvlo}
    $D_B^{\tilde G}$ is an involution. In particular, it is an equivalence of categories.
\end{pro}
\begin{proof}
    The proof for the connected case also works for the disconnected case. Given $\tilde \pi \in \mc{R}_t(\tilde G)$ for some $0\leq t \leq r$, we first choose a finite projective resolution $0\to \tilde Z_s\to \dots\to \tilde Z_0\to \tilde \pi \to 0$ with finitely generated projectives. Let us write $V^*$ for $\tx{Hom}_{\tilde G}(V, \mc{\tilde H})$ for any representation $V$ of $\tilde G(F)$. Then we have a complex 
    \[0\to \tilde Z_0^*\xrightarrow{\partial_0}\dots \xrightarrow{\partial_{s-1}} \tilde Z_s^*\xrightarrow{\partial_s} 0\]
    which has non-zero cohomology in a single degree $d(t)$, by Corollary \ref{cor:dcoh}. Suppose for the moment that each $\tilde Z^*_i$ is a projective representation. We may use this property to truncate this complex and obtain a resolution of $D_B^{\tilde G}(\tilde \pi)$ as follows. If $d(t)=s$, then composing with the canonical projection 
    \[\tilde Z_s^*\to D_B^{\tilde G}(\tilde \pi)=\tilde  Z_s^*/\tx{Im}(\partial_{s-1})\]
    we obtain a projective resolution. If $d(t)\leq s-1$, we have the following short exact sequence:
    \[0\to K_{s-1} \to \tilde Z^*_{s-1}\to \tilde Z_s^*\to 0,\]
    where $K_i:=\tx{Ker}(\partial_i)$ for $ 0\leq i\leq s$. This sequence splits, as both $\tilde Z^*_{s-1}$ and $\tilde Z^*_s$ are projective. In particular, $K_{s-1}$, being a direct summand of $\tilde Z^*_{s-1}$, is also projective. If $d(t)=s-1$, then the complex
    \[ 0\to \tilde Z^*_0\to \dots \to \tilde Z^*_{s-2} \to K_{s-1} \to D_B^{\tilde G}(\tilde \pi)\to 0, \]
    is a (finitely generated) projective resolution of $D_B^{\tilde G}(\tilde \pi)$. If $d(t)\leq s-2$, then $\tx{Im}(\partial_{s-1})=K_s$, and we may do the above argument again. Eventually, we will obtain a projective resolution of the form:
    \[ 0\to \tilde Z^*_0\to \dots \to \tilde Z^*_{d(t)-1} \to K_{d(t)} \to D_B^{\tilde G}(\tilde \pi)\to 0. \]
    Applying the functor $(-)^*$ again and assuming that $\tilde Z^{**}$ is functorially isomorphic to $\tilde Z$ for finitely generated projective representations, we obtain the desired isomorphism $(D_B^{\tilde G})^2(\tilde \pi)\cong \tilde \pi$.

    It remains to show that for a finitely generated projective representation $\tilde Z$ of $\tilde G(F)$, the representation $\tilde Z^*$ is also projective and $\tilde Z^{**}$ is functorially isomorphic to $\tilde Z$. These facts have been proved in the connected case, see \cite{BernsteinPadicGroups} or \cite[Proposition 2.9]{YS2026M}, and the argument also works in the disconnected case. For the reader's convenience, let us sketch the proof here. It is clear that we may identify the category of smooth $\tilde G(F)$ representations with the category of non-degenerate $\mc{\tilde H}$-left modules. Then the functor $\tx{Hom}_{\tilde G(F)}(-, \mc{\tilde H})$ is identified with $\tx{Hom}_{\mc{\tilde H}}(-, \mc{\tilde H})$, and finitely generated projective $\tilde G(F)$ representations correspond to finitely generated projective non-degenerate left $\mc{\tilde H}$-module. Every such module is a direct summand of a finite direct sum $\oplus \mc{\tilde H}e_i$ with $e_i$ idempotents. Now two statements are clear.
\end{proof}

\begin{pro} \label{pro:dhiprp}
    Given a parabolic subgroup $\tilde P\subset \tilde G$ with Levi decomposition $\tilde P= \tilde M U$,
    \begin{enumerate}
        \item $i^{\tilde G}_{\tilde P}\circ D_B^{\tilde M}$ is equivalent to $D_B^{\tilde G}\circ i^{\tilde G}_{\tilde P^-}$.
        \item $r^{\tilde G}_{\tilde P}\circ D_B^{\tilde G}$ is equivalent to $D_B^{\tilde M}\circ r^{\tilde G}_{\tilde P}$.
    \end{enumerate}
\end{pro}
\begin{proof}
    The proof for the connected case also works for the disconnected case, which is based on the second adjointness (Corollary \ref{cor:secadjoint}) and the following claims.
    \begin{clm}
        We have an isomorphism of $\tilde M\times \tilde G$-representations (resp.~$\tilde G\times \tilde M$-representations):
        \[(1\times i^{\tilde G}_{\tilde P})\mc{H}_{\tilde M} = (r^{\tilde G}_{\tilde P}\times 1)\mc{H}_{\tilde G} \qquad\tx{resp.} \qquad (i^{\tilde G}_{\tilde P}\times 1)\mc{H}_{\tilde M} = (1\times r^{\tilde G}_{\tilde P})\mc{H}_{\tilde G}.\]
    \end{clm}
    We again give a sketch of proof here, which is the same as the arguments in \cite[Proposition 2.7]{YS2026M}. Let $\mc{C}_c(U\backslash \tilde G)$ be the set of compactly supported complex-valued functions on $U(F) \backslash \tilde G(F)$. It is a $\tilde M(F)\times \tilde G(F)$-representation by the following rule: for $(m,g)\in\tilde M(F)\times \tilde G(F)$ and $f\in \mc{C}_c(U\backslash \tilde G)$:
    \[((m,g)\cdot f)(Ux): = \delta_{\tilde P}(m)^{\frac{1}{2}}f(Um^{-1}xg)\] 
    Let $\mc{C}_c^{\infty}(U\backslash \tilde G)$ be the subset of smooth vectors with respect to this action. Then we have isomorphisms:
    \[(1\times i^{\tilde G}_{\tilde P})\mc{H}_{\tilde M} \to \mc{C}_c^\infty(U\backslash \tilde G)\]
    by sending $f: \tilde G(F) \to \mc{H}_{\tilde M}$ to the function $Ug\mapsto f(g)(1)$, and 
    \[(r^{\tilde G}_{\tilde P}\times 1)\mc{H}_{\tilde G} \to \mc{C}_c^\infty(U\backslash \tilde G) \]
    by sending a representative function $f\in \mc{\tilde H}$ to the following function
    \[Ug \mapsto \int_Uf(ug) du.\] 
    Similarly, we have isomorphisms:
    \[(i^{\tilde G}_{\tilde P}\times 1)\mc{H}_{\tilde M} \xrightarrow{} \mc{C}_c^\infty(\tilde G/U) \to (1\times r^{\tilde G}_{\tilde P})\mc{H}_{\tilde G}.\]
    \begin{clm} \label{clm:ihom}
        For any finitely generated $\tilde M(F)$ representation $\tilde Z$, we have functorial isomorphism of $\tilde G(F)$ representations:
        \[i^{\tilde G}_{\tilde P}(\tx{Hom}_{\tilde M(F)}(\tilde Z, \mc{H}_{\tilde M})) = \tx{Hom}_{\tilde M(F)}(\tilde Z, (1\times i^{\tilde G}_{\tilde P})\mc{H}_{\tilde M}) \].
    \end{clm}
    By definition, both sides are given by functions \(F: \tilde G(F)\times \tilde Z \rightarrow \mc{H}_{\tilde M}\) which are linear maps on $\tilde Z$ and satisfy the following conditions for any fixed $g\in \tilde G(F)$:
        \begin{enumerate}
            \item \(F(pg,v)(x) = \delta_{\tilde P}(m)^{\frac{1}{2}}F(g,v)(xm)\) for all $p\in \tilde P(F)=\tilde M(F)\cdot U(F)$, $v\in \tilde Z$ and $x\in \tilde M(F)$, where $p=mu$ with $m\in \tilde M$ and $u\in U$.
            \item \(F(g,mv)(x)=F(g,v)(m^{-1}x)\) for all \(m\in \tilde M(F)\) and $v\in \tilde Z$.
        \end{enumerate}
    Furthermore, the function $F$ on both sides has to satisfy some local constant property: $F$ is in the left hand side if and only if the corresponding map 
    \[\tilde G(F)\rightarrow \tx{Hom}_{\tilde M(F)}(\tilde Z, \mc{H}_{\tilde M})\] 
    is locally constant, and the right hand side if and only if 
    \[F(-,v): \tilde G(F) \rightarrow \mc{H}_{\tilde M}\]
    is locally constant for each $v\in \tilde Z$. As $\tilde Z$ is finitely generated, the same argument in \cite{YS2026M} shows that these two properties are equivalent.

    With these two claims, we may prove the first statement. Given $\tilde \tau\in \mc{R}_t(\tilde M)$, we choose a finite projective resolution $0\to \tilde Z_s\to\dots \to\tilde Z_0\to \tilde \tau \to 0$ with $\tilde Z_i$ fintiely generated. Then, we have
    \[\tx{Hom}_{\tilde G(F)}(i^{\tilde G}_{\tilde P^-}\tilde Z_i, \mc{H}_{\tilde G}) =\tx{Hom}_{\tilde M(F)}(\tilde Z_i, (r^{\tilde G}_{\tilde P}\times 1)\mc{H}_{\tilde G}) = i^{\tilde G}_{\tilde P}\tx{Hom}_{\tilde M(F)}(\tilde Z_i, \mc{H}_{\tilde M})\]
    where the first isomorphism is given by second adjointness \ref{cor:secadjoint} and the second one given by the two claims above. Noting that both parabolic induction and restriction preserve finitely generated projective representations, we finish the proof of the first statement. The proof for the second statement is similar, with another claim.
    \begin{clm}
        For any finitely generated projective representation $\tilde Z$ of $\tilde G(F)$, we have functorial isomorphism:
        \[r^{\tilde G}_{\tilde P}\tx{Hom}_{\tilde G(F)}(\tilde Z, \mc{H}_{\tilde G}) \rightarrow \tx{Hom}_{\tilde G(F)}(\tilde Z, (1\times r^{\tilde G}_{\tilde P})\mc{H}_{\tilde G}).\]
    \end{clm}
    Note that the natural map 
    \[\tx{Hom}_{\tilde G(F)}(\tilde Z, \mc{H}_{\tilde G}) \to \tx{Hom}_{\tilde G(F)}(\tilde Z, (1\times r^{\tilde G}_{\tilde P})\mc{H}_{\tilde G})\]
    is surjective (as vector spaces), as $\tilde Z$ is projective. The same argument in \cite{YS2026M} yields that it descends to the desired isomorphism.

    Given $\tilde \pi\in \mc{R}_t(\tilde G)$, we choose a finite projective resolution $0\to \tilde Z_l\to\dots\to \tilde Z_0\to \tilde \pi\to0$ with $\tilde Z_i$ finitely generated. Then we have
    \[\tx{Hom}_{\tilde M(F)}(r^{\tilde G}_{\tilde P}\tilde Z_j, \mc{H}_{\tilde M}) = \tx{Hom}_{\tilde G(F)}(\tilde Z_j, (i^{\tilde G}_{\tilde P}\times 1)\mc{H}_{\tilde M})=r^{\tilde G}_{\tilde P}\tx{Hom}_{\tilde G(F)}(\tilde Z_j, \mc{H}_{\tilde G}),\]
    and we complete the proof.
\end{proof}

In \S\ref{sub:dgequiv}, we study the equivariance properties of $D_A^G$ and $D_B^G$. Using that property for $D_A^G$, we extend the action of $G(F)$ on $D_A^G(\pi)$ to $\tilde G(F)$ and define it to be the Aubert duality for $\tilde G$, denoted by $D_A^{\tilde G}$, as discussed in \S\ref{sub:defdgdis}. For $D_B^G$, instead of extending the action, we define the functor $D_B^{\tilde G}$ directly, whose underlying vector space is isomorphic to $D_B^G$ (see Lemma \ref{lem:dhres}). Since we can also use the equivariance property for $D_B^G$ to extend the action, the following lemma suggests that these two approaches coincide.

\begin{lem} \label{lem:dhdiscon}
    Let $\tilde \pi\in\mc{R}(\tilde G)$, $n\in \tilde N(F)$ and $a=\tx{Ad}(n)$. Write $\pi=\tx{Res}^{\tilde G}_G\tilde \pi$, then we have $G(F)$-isomorphism $\theta_a=\tilde \pi(n): a\pi\to \pi$, as discussed in \S\ref{sub:defdgdis}. We have the following commutative diagram in $\mc{R}(G)$:
    \[\xymatrix{
        aD_B^G(\pi)\ar[r]^{H(\theta_a)}\ar[d]& D_B^G(\pi)\ar[d]\\
        a\tx{Res}^{\tilde G}_GD_B^{\tilde G}(\tilde \pi) \ar[r]^{\tilde \theta(n)} &\tx{Res}^{\tilde G}_GD_B^{\tilde G}(\tilde \pi)
    }\]
    where $H(\theta_a)$ is defined in \S\ref{sub:dgequiv}, $\tilde \theta_a=D_B^{\tilde G}(\tilde \pi)(n)$ and the two vertical maps are the isomorphisms in Lemma \ref{lem:dhres}.
\end{lem}
\begin{proof}
    We may assume $\tilde \pi\in \mc{R}_t(\tilde G)$ for some $t$. Let $0\to\tilde Z_s\to\dots\to\tilde Z_0\to \tilde \pi\to 0$ be a finite finitely generated projective resolution, then $0\to Z_s\to \dots \to Z_0\to \pi$ is also a projective resolution for $\pi$, where $Z_i=\tx{Res}^{\tilde G}_G\tilde Z_i$. Then the vertical isomorphisms are induced by the following map
    \[\tx{Hom}_{G(F)}(Z_i,\mc{H})\to \tx{Res}^{\tilde G}_G\tx{Hom}_{\tilde G(F)}(\tilde Z_i, \mc{\tilde H})\]
    by sending $\psi: v\mapsto \psi_v$ to $L(\psi): v\mapsto \bar{\psi}_v$, see the proof of Lemma \ref{lem:homhGres}. Recall that $H(\theta_a)$ is the composition
    \[aD_B^G(\pi)\xrightarrow{c_h(\pi)} D_B^G(a\pi) \xrightarrow{D_B^G(\theta^{-1}_a)}D_B^G(\pi).\]
    Thus the composition of $H(\theta_a)$ with the right vertical isomorphism is induced by the following map
    \[\tx{Hom}_{G(F)}(Z_i, \mc{H})\to \tx{Res}^{\tilde G}_G\tx{Hom}_{\tilde G(F)}(\tilde Z_i, \mc{\tilde H})\]
    by sending $\psi: v\mapsto \psi_v$ to $LH(\psi): v\mapsto \overline{\psi_{n^{-1}v}\circ a^{-1}}$. On the other hand, the composition of the left vertical isomorphism with $\tilde \theta(n)$ is induced by the map
    \[\tx{Hom}_{G(F)}(Z_i,\mc{H})\to \tx{Res}^{\tilde G}_G\tx{Hom}_{\tilde G(F)}(\tilde Z_i,\mc{\tilde H})\]
    by sending $\psi:v\mapsto \psi_v$ to $\tilde \theta L(\psi): v\mapsto R_n\bar{\psi}_v$, where $(R_nf)(g)=f(gn)$, the right regular action. Choose a coset decomposition $\tilde G(F)=\coprod_{i=1}^ng_iG(F)$ such that $g_i=n$ for a unique $i$. This does no harm because the definition of $\bar{\psi}$ is independent of the choice of representatives. For arbitrary $g\in G(F)$, we have $gn\in nG(F)$ as $G(F)$ is normal in $\tilde G(F)$. Therefore,
    \begin{align*}
    (R_n\bar{\psi}_v)(g)=\bar{\psi}_v(gn)&=\sum_{i=1}^n\chi_{g_iG(F)}(gn)\psi_{g_i^{-1}v}(g_i^{-1}gn)  \\
    &= \psi_{n^{-1}v}(n^{-1}gn) \\
    &=(\psi_{n^{-1}v}\circ a^{-1})(g)
    \end{align*}
    Therefore, the function $R_n\bar{\psi}_v$ restricting on $G(F)$ equals $\psi_{n^{-1}v}\circ a^{-1}$. As restricting on $G(F)$ is the inverse of $L$, see Lemma \ref{lem:homhGres}, we deduce that 
    \[ \overline{\psi_{n^{-1}v}\circ a^{-1}} = R_n\bar{\psi}_v, \] 
    as required.
\end{proof}

This lemma will be used in the next subsection to establish the connection between $D_B^{\tilde G}$ and $D_A^{\tilde G}((-)^\vee)$.

\subsection{More properties}

In this subsection we aim to establish some additional properties of the functor $D_A^{\tilde G}$, which are analogous to those of $D_A^G$. 

We first generalize the relation $D_A^G((-)^\vee)\cong D_B^G$ to the disconnected case. 

\begin{pro} \label{pro:dhequdgcon}
    The functor $D_A^{\tilde G}((-)^\vee)$ is equivalent to $D_B^{\tilde G}$.
\end{pro}
\begin{proof}
    For $\tilde \pi\in \mc{R}(\tilde G)$, let us write $\pi = \tx{Res}^{\tilde G}_G\tilde \pi$ as usual. By Lemma \ref{lem:dgres}, \ref{lem:dhres} and Proposition \ref{pro:dhprp}(6), we have functorial isomorphisms for $\tilde \pi\in\mc{R}(\tilde G)$:
    \[\tx{Res}^{\tilde G}_GD_B^{\tilde G}(\tilde \pi)\cong D_B^G(\pi)\xrightarrow{\beta_\pi} D_A^G(\pi^\vee)=\tx{Res}^{\tilde G}_GD_A^{\tilde G}(\tilde \pi^\vee),\]
    that is, a functorial isomorphism $\alpha_{\tilde \pi}: D_B^{\tilde G}(\tilde \pi)\to D_A^{\tilde G}(\tilde \pi^\vee)$ of vector spaces respecting the action of $G(F)$. It then remains to show it also respects the action of $\tilde N(F)$. Fix $n\in \tilde N(F)$, and write $a=\tx{Ad}(n)$ and $\theta_a: a\pi\to \pi$ as before. By Lemma \ref{lem:dhdiscon} and construction of $D_A^{\tilde G}(\tilde \pi)(n)$, we reduce to checking the following diagram commutes:
    \[\xymatrix{
       aD_B^G(\pi)\ar[rr]^{H(\theta_a)}\ar[d]_{a\beta_\pi} && D_B^G(\pi) \ar[d]_{\beta_\pi} \\
       aD_A^G(\pi^\vee)\ar[rr]^{D_A^{\tilde G}(\tilde \pi^\vee)(n)}&& D_A^G(\pi^\vee),
    }\]
    but this is Proposition \ref{pro:dhdgequ}. Note that $Y(a,(\theta_a^{-1})^\vee)=D_A^{\tilde G}(\tilde \pi^\vee)(n)$ by construction.
\end{proof}

\begin{cor} \label{cor:dhsc}
    For $\tilde \pi\in \mc{R}_r(\tilde G)$, we have a functorial isomorphism $D_B^{\tilde G}(\tilde \pi)\cong \tilde \pi^\vee$. In particular, $D_B^{\tilde G}$ is equivalent to the contragredient functor on supercuspidal blocks.
\end{cor}
\begin{proof}
    This is a direct consequence of Propositions \ref{pro:dhequdgcon} and \ref{pro:dgsc-disc}.
\end{proof}

\begin{pro} \label{pro:dhctr}
    $D_B^{\tilde G}$ commutes with the contragredient functor.
\end{pro}
\begin{proof}
    This is proved in the connected case by using the "lifting isomorphism lemma" in \cite[Corollary 3.7]{YS26}, and this argument also works for the disconnected case. First, we fix some $0\leq t\leq r$. For any $I\subset S$ with $|I|=t$, let $\tilde P_I=\tilde M_I U_I$ be the corresponding standard wide parabolic subgroup and Levi decomposition. By Corollary \ref{cor:iprpdcmp}, we have functors:
    \[i_I:=i^{\tilde G}_{\tilde P_I}: \mc{R}_t(\tilde M_I) \to \mc{R}_t(\tilde G), \qquad r_I:=r^{\tilde G}_{\tilde P_I^-}: \mc{R}_t(\tilde G)\to \mc{R}_t(\tilde M_I)\]
    They form an adjoint pair by the second adjointness \ref{cor:secadjoint} and Bernstein decomposition. We then have adjunction maps
    \[\kappa_I(\tilde \pi): i_Ir_I\tilde \pi \to \tilde \pi, \qquad \zeta_I(\tilde \tau):\tilde \tau \to r_Ii_I\tilde \tau,\]
    for $\tilde \pi\in \mc{R}_t(\tilde G)$ and $\tilde \tau\in \mc{R}_t(\tilde M_I)$. Similarly, $i^{\tilde G}_{\tilde P^-_I}$ and $r^{\tilde G}_{\tilde P_I}$ also form such pairs, and are denoted by $\bar{i}_I$ and $\bar{r}_I$. Define
    \[\mc{A}:=\mc{R}_t(\tilde G),\qquad \mc{B}:=\prod_{|I|=t}\mc{R}_t(\tilde M_I)\]
    and 
    \[i:=\prod i_I: \mc{B}\to \mc{A},\qquad r:=\prod r_I: \mc{A}\to \mc{B}\]
    Then it is clear that $i$ and $r$ form an adjoint pair, with adjunction maps $\kappa=\prod \kappa_I$ and $\zeta=\prod \zeta_I$. Similarly, $\bar{i}:=\prod \bar{i}_I$ and $\bar{r}:=\prod \bar{r}_I$ also form an adjoint pair. For any $\tilde \pi\in \mc{R}_t(\tilde G)$, the following composition
    \[r\tilde\pi \xrightarrow{\zeta(r\tilde \pi)} rir\tilde \pi \xrightarrow{r(\kappa(\tilde \pi))} r\tilde \pi\]
    equals identity, by the definition of adjunction maps. In particular, we see that $r(\tx{Coker}(\kappa(\tilde \pi)))=0$ as $r$ is exact. As $\tx{Coker}(\kappa(\tilde \pi))\in \mc{R}_t(\tilde G)$, it has cuspidal support in the (standard) Levi subgroups $M_I$ with $|I|=t$ when restricting as a $G(F)$-representation. We then deduce that it is zero and hence $\kappa(\tilde \pi)$ is surjective. A similar argument shows that $\zeta(\tilde \tau)$ is injective, for $\tilde \tau=(\tilde \tau_I)\in \mc{B}$.
    
    Define $F_{I}(-):= D_B^{\tilde M_I}((-)^\vee)$, $E_{I}(-):= (D_B^{\tilde M_I}(-))^\vee$ as functors from $\mc{R}_t(\tilde M_I)$ to itself, and similarly $F_{\tilde G}$, $E_{\tilde G}$. Also define 
    \[F_b:=\prod F_I, E_b:=\prod E_I: \mc{B}\to \mc{B}.\] 
    These are exact covariant functors and we have equivalences of functors
    \[i\circ F_{b}\cong F_{\tilde G}\circ \bar{i}, \qquad r\circ F_{\tilde G} \cong F_b\circ \bar{r} \]
    and 
    \[i\circ E_b \cong E_{\tilde G}\circ \bar{i}, \qquad r\circ E_{\tilde G} \cong E_b\circ \bar{r},\]
    by Lemma \ref{lem:iprpcon} and Proposition \ref{pro:dhiprp}. On the block $\mc{R}_t(\tilde M_I)$, we know that $D_B^{\tilde M_I}$ is equivalent to $(-)^\vee$ by Corollary \ref{cor:dhsc}. Therefore, we have a functorial isomorphism $\alpha_I(\tilde \tau): F_{\tilde M_I}(\tilde \tau)\to E_{\tilde M_I}(\tilde \tau)$ and hence an isomorphism $\alpha:= \prod \alpha_I: F_b\to E_b$. We can now apply the "lifting isomorphism lemma" in \cite{YS26} to lift the isomorphism $\alpha: F_b \to E_b$ to $\alpha': F_{\tilde G}\to E_{\tilde G}$, provided that both $D_B^{\tilde G}$ and $(-)^\vee$ permute the first adjointness (fact \ref{fct:adjoint}) and the second adjointness. We have introduced this notion in \S\ref{sub:dguniq}. By first adjointness, we have adjunction maps 
    \[\eta_I(\tilde{\pi}): \tilde \pi \to \bar{i}_Ir_I \tilde \pi, \qquad \epsilon_I(\tilde \tau): r_I\bar{i}_I\tilde \tau\to\tilde \tau,\]
    for $\tilde \pi\in \mc{R}_t(\tilde G)$ and $\tilde \tau \in \mc{R}_t(\tilde M_I)$. If the composition of maps
    \[ i_Ir_ID_B^{\tilde G}(\tilde \pi) \xrightarrow{\simeq} D_B^{\tilde G}(\bar{i}_Ir_I\tilde \pi) \xrightarrow{D_B^{\tilde G}(\eta_I)} D_B^{\tilde G}(\tilde \pi)\]
    equals $\kappa_I(D_B^{\tilde G}(\tilde \pi))$, and 
    \[D^{\tilde M_I}_B(\tilde \tau)\xrightarrow{D^{\tilde M_I}_B(\epsilon_I(\tilde \tau))}D^{\tilde M_I}_B(r_I\bar{i}_I\tilde \tau) \xrightarrow{\simeq} r_Ii_ID_B^{\tilde M_I}(\tilde \tau)\]
    equals $\zeta_I(D^{\tilde M}_B(\tilde \tau))$, we say that $D_B^{\tilde G}$ takes first adjointness to second adjointness. If it also takes second adjointness to the first, we say that $D_B^{\tilde G}$ permutes first and second adjointness. Similarly, since $(-)^\vee$ also has similar commutative properties with parabolic induction and restriction (Lemma \ref{lem:iprpcon}), it also makes sense to say $(-)^\vee$ permutes first and second adjointness. Strictly speaking, the lemma in \cite{YS26} requires $D^{\tilde G}_B$ and $(-)^\vee$ to permute the adjunction maps between $(i,r)$ and $(\bar{i}, \bar{r})$ in the above sense, but since they are finite direct products of $\kappa_I, \zeta_I$ and $\eta_I, \epsilon_I$, and since both $D^{\tilde G}_B$ and $(-)^\vee$ are additive, we may reduce to each $I$.

    It remains to show that both $D_B^{\tilde G}$ and $(-)^\vee$ permute first and second adjointness. For $(-)^\vee$, this is clear, since we obtain second adjointness by $(-)^\vee$ and first adjointness, see Corollary \ref{cor:secadjoint}. The hard part is to show $D_B^{\tilde G}$ also does. In the connected case, this is verified directly by using the proofs of the fact that $D_B^G$ is an involution and that $D_B^G\circ i^G_P \cong i^G_{P^-}\circ D_B^M$, $D_B^M\circ r^G_P\cong r^G_P\circ D_B^G$. Since all these proofs can be generalized in the disconnected case, see Proposition \ref{pro:dhinvlo} and \ref{pro:dhiprp}, the argument in \cite[Proposition 2.3]{YS26} can also be generalized in the disconnected case, which is simply tracing elements. We refer to that paper for more details.
\end{proof}

\begin{cor}
    The Aubert duality $D_A^{\tilde G}$ satisfies the following properties.
    \begin{enumerate}
        \item $D_A^{\tilde G}$ is an involution.
        \item $D_A^{\tilde G }\circ i^{\tilde G}_{\tilde P}$ is equivalent to $i^{\tilde G}_{\tilde P^-}\circ D_A^{\tilde M}$.
        \item $D_A^{\tilde M}\circ r^{\tilde G}_{\tilde P}$ is equivalent to $r^{\tilde G}_{\tilde P^-}\circ D_A^{\tilde G}$.
        \item $D_A^{\tilde G}$ commutes with the contragredient functor.
    \end{enumerate}
\end{cor}
\begin{proof}
    (1) follows from Proposition \ref{pro:dhequdgcon} and \ref{pro:dhinvlo}. (2) and (3) follow from Proposition \ref{pro:dhequdgcon}, \ref{pro:dhiprp} and Lemma \ref{lem:iprpcon}. (4) follows from \ref{pro:dhequdgcon} and \ref{pro:dhctr}.
\end{proof}

To end this section, we generalize the uniqueness property in \S\ref{sub:dguniq} to the disconnected case.

\begin{pro} \label{pro:dhuniqdisc}
    The family of functors $D^{\tilde M}_B: \mc{R}(\tilde M)\to \mc{R}(\tilde M)$, with $\tilde M$ varying over all (standard) Levi subgroups of $\tilde G$, is uniquely determined by the following properties:
    \begin{enumerate}
        \item $D^{\tilde M}_B$ is an exact contravariant functor for all $\tilde M$.
        \item $D_B^{\tilde G }\circ i^{\tilde G}_{\tilde P}$ is equivalent to $i^{\tilde G}_{\tilde P^-}\circ D_B^{\tilde M}$.
        \item $D_B^{\tilde M}\circ r^{\tilde G}_{\tilde P}$ is equivalent to $r^{\tilde G}_{\tilde P}\circ D_B^{\tilde G}$.
        \item $D_B^{\tilde G}$ permutes first and second adjointness as in the proof of Proposition \ref{pro:dhctr}.
        \item Each $D^{\tilde M}_B$ respects Bernstein decomposition, and it is isomorphic to the contragredient functor on supercuspidal blocks.
    \end{enumerate}
\end{pro}
\begin{proof}
    This is a direct corollary of the "lifting isomorphism lemma" argument in the proof of Proposition \ref{pro:dhctr}. Let us use the notations in that proof. Suppose $(F_{\tilde M})_{\tilde M}$ is another family of functors satisfying the above properties. Define
    \[F_b:=\prod_{|I|=t}F_{\tilde M_I}, E_b:=\prod_{|I|=t}D^{\tilde M_I}_B: \mc{B}\to \mc{B}.\]
    Also define $E_{\tilde G}:=D^{\tilde G}_B$. These are exact contravariant functors and we have equivalences of functors
    \[i\circ F_b \cong F_{\tilde G}\circ \bar{i}, \qquad r\circ F_{\tilde G}\cong F_b\circ r\]
    and 
    \[i\circ E_b\cong E_{\tilde G}\circ \bar{i}, \qquad r\circ E_{\tilde G}\cong E_b\circ r.\]
    On $\mc{R}_t(\tilde M_I)$, we have an isomorphism
    \[F_b \cong \prod (-)^\vee \cong E_b.\]
    Since both $E_{\tilde M}$ and $F_{\tilde M}$ permute first and second adjointness, the lifting isomorphism lemma (for contravariant functors) applies and we obtain an isomorphism $D^{\tilde G}_B\cong F_{\tilde G}$ on the block $\mc{R}_t(\tilde G)$. A similar argument works for all $\mc{R}_l(\tilde M)$, and we complete the proof.
\end{proof}

Using the lifting isomorphism lemma (for covariant functors), we can also prove the following.

\begin{pro} \label{pro:dguniqdisc}
    The family of functors $D_A^{\tilde M}: \mc{R}(\tilde M)\to \mc{R}(\tilde M)$, with $\tilde M$ varying over all (standard) Levi subgroups of $\tilde G$, is uniquely determined by the following properties:
    \begin{enumerate}
        \item $D_A^{\tilde M}$ is an exact covariant functor for all $\tilde M$.
        \item $D_A^{\tilde G }\circ i^{\tilde G}_{\tilde P}$ is equivalent to $i^{\tilde G}_{\tilde P^-}\circ D_A^{\tilde M}$.
        \item $D_A^{\tilde M}\circ r^{\tilde G}_{\tilde P}$ is equivalent to $r^{\tilde G}_{\tilde P^-}\circ D_A^{\tilde G}$.
        \item $D_A^{\tilde G}$ preserves first adjointness, cf. Proposition \ref{pro:dguniq}.
        \item Each $D_A^{\tilde M}$ respects Bernstein decomposition, and it is isomorphic to the identity functor on supercuspidal blocks.
    \end{enumerate}
\end{pro}

\section{The Steinberg representation for disconnected groups} \label{sec:stein}

We continue with the set-up of \S\ref{sec:aubdisc}. Thus $F$ is a non-archimedean local field, $\tilde G$ is an affine algebraic group whose identity component $G$ is reductive. The group $G(F)$ has a natural representation, called the Steinberg representation. We will produce in this section a natural extension of this representation to $\tilde G(F)$, compute its character, and discuss twisted endoscopic transfer.

\subsection{The Steinberg representation} \label{sub:stein}

A quick way to define the Steinberg representation of $G(F)$ is as the Aubert-dual of the trivial representation:
\[ \tx{St} = D_A^G(\textbf{1}).\]
We can however describe it in more concrete terms, following \cite{Cas73}, as follows. Let $P_0 \subset G$ be a minimal parabolic subgroup. Consider the induced representation 
\[ i_{P_0}^G(\delta_{P_0}^{-1/2})=\tx{Ind}_{P_0}^G\textbf{1}_{P_0}, \]
which is the right regular representation of $G(F)$ on the space of functions $\mc{C}^\infty(P_0(F) \lmod G(F))$. We identify this space with the subspace of $\mc{C}^\infty(G(F))$ consisting of functions that are left-invariant under $P_0(F)$. Every parabolic subgroup $P_0 \subset P$ gives the submodule $\mc{C}^\infty(P(F) \lmod G(F))$ of $\mc{C}^\infty(P_0(F) \lmod G(F))$. The Steinberg representation is the quotient of $\mc{C}^\infty(P_0(F) \lmod G(F))$ by the sum (not direct) of all $\mc{C}^\infty(P(F) \lmod G(F))$ for $P_0 \subsetneq P$. Call this sum $\Sigma_0$ for future reference.

The fact that this procedure coincides with $D_A^G(\textbf{1})$ is seen as follows. For any $P \subset G$ we have $I_P^GR_P^G(\mathbf{1})=\mc{C}^\infty(P(F) \lmod G(F))$. Since $\mathbf{1} \in \mc{R}_0(G)$, the complex in Theorem \ref{thm:aubmain} has the form
    \[ 0 \to Y_r \to Y_{r-1} \to \dots \to Y_0, \]
so $\tx{St} = \tx{cok}(Y_1 \to Y_0)$, but 
$Y_0 = \mc{C}^\infty(P_0(F) \lmod G(F))$ and the image of the above map is $\Sigma_0$. 

It is well-known, see e.g. \cite{Cas73}, that for a regular semi-simple element $g \in G(F)$ the character of $\tx{St}$ is given by 
\begin{equation} \label{eq:charstein0}
\sum_{(M,P)} (-1)^{\dim(A_0/A_M)}\delta_P(g)^{-\frac{1}{2}}|D_{G/M}(g)|^{-\frac{1}{2}},
\end{equation}
where the sum runs over the set of those parabolic pairs for which $g \in M$.

\subsection{The natural extension} \label{sub:discstein}

We have the following ways to define an extension of the representation $\tx{St}$ of $G(F)$ to a representation $\tilde{\tx{St}}$ of $\tilde G(F)$. First we could take $D_A^{\tilde G}(\tilde{\textbf{1}})$,
where $\tilde{\textbf{1}}$ is the trivial representation of $\tilde G(F)$ and $D_A^{\tilde G}$ is the functor defined in Definition \ref{dfn:aubdisc1}. But we could also use the alternative functor $'D_A^{\tilde G}$ of \S\ref{sub:sign} and consider $'D_A^{\tilde G}(\tilde{\textbf{1}})$.

Finally, we could imitate the explicit description of \S\ref{sub:stein}, as follows. Let $(M_0,P_0)$ be a minimal parabolic pair for $G$. In every coset $\tilde G(F)/G(F)$ there exists a representative $\tilde g \in \tilde G(F)$ that normalizes $(M_0,P_0)$. Thus $\tilde G(F) = \tilde P_0(F) \cdot G(F)$, where $\tilde P_0(F)$ is the normalizer in $\tilde G(F)$ of $P_0$. The inclusion $G \to \tilde G$ induces a homeomorphism $P_0(F) \lmod G(F) \to \tilde P_0(F) \lmod \tilde G(F)$, and hence an isomorphism of vector spaces $\mc{C}^\infty(P_0(F) \lmod G(F)) \to \mc{C}^\infty(\tilde P_0(F) \lmod \tilde G(F))$. In this way we obtain an action of $\tilde G(F)$ on $i_{P_0}^G\delta_{P_0}^{-1/2}=\mc{C}^\infty(P_0(F) \lmod G(F))$. In explicit terms, $(\tilde g \cdot f)(x)=f(\tilde p^{-1}x\tilde p g)$ for $f \in \mc{C}^\infty(P_0(F) \lmod G(F))$ and $\tilde g = \tilde p \cdot g \in \tilde G(F) = \tilde P_0(F) \cdot G(F)$.

We claim that the sum $\Sigma_0$ of the various submodules $\mc{C}^\infty(P(F) \lmod G(F))$ of $\mc{C}^\infty(P_0(F) \lmod G(F))$ is stable under the action of $\tilde G(F)$. Each individual summand is stable under the action of $G(F)$. On the other hand, the action of $\tilde p \in \tilde P_0(F)$ sends $\mc{C}^\infty(P(F) \lmod G(F))$ to $\mc{C}^\infty(P'(F) \lmod G(F))$, where $P'=\tilde p \cdot P \cdot \tilde p^{-1}$ again satisfies $P_0 \subsetneq P'$. In this way we obtain an extension of the Steinberg representation $\mc{C}^\infty(P_0(F) \lmod G(F))/\Sigma_0$ of $G(F)$ to $\tilde G(F)$. 

\begin{lem} \label{lem:stein}
    The explicit description above produces the representation 
    \[ 'D_A^{\tilde G}(\tilde{\textbf{1}}) = D_A^{\tilde G}(\tilde{\textbf{1}})\otimes \epsilon_{\tilde G}. \]
\end{lem}
\begin{proof}
The argument from \S\ref{sub:stein} reduces to checking that the action of $\tilde G(F)$ on 
\[ 'X_{P_0}(\tilde{\textbf{1}})=\mc{C}^\infty(P_0(F) \lmod G(F)) \]
defined via the automorphism $'X_0(a,\theta_a)$ of Fact \ref{fct:bs2} for $a=\tx{Ad}(n)$, $n \in \tilde N(F)$, and $\theta_a=\tx{id}$, agrees with the action induced by the isomorphism
\[ \mc{C}^\infty(P_0(F) \lmod G(F)) = \mc{C}^\infty(\tilde P_0(F) \lmod \tilde G(F)). \] 
But $'X_0(a,\theta_a)=\epsilon(a^{-1},\varnothing)\cdot X_{\varnothing}(a,\theta_a)$ and we observe $\epsilon(a^{-1},\varnothing)=1$ while $X_{\varnothing}(a,\theta_a)$ does coincide with the action of $n$ given by the above isomorphism.
\end{proof}

We are now faced with the question of which of the two options $'D_A^{\tilde G}(\tilde{\textbf{1}})$ and $D_A^{\tilde G}(\tilde{\textbf{1}})$ to call \emph{the} Steinberg representation $\tilde{\tx{St}}$. In general this is a matter of preference. In this paper we set 
\begin{equation} \label{eq:stein}
    \tilde{\tx{St}} := {}'D_A^{\tilde G}(\tilde{\textbf{1}}),
\end{equation}
which coincides with the above explicit description, but is in general \emph{not} equal to $D_A^{\tilde G}(\tilde{\textbf{1}})$. The reason we use this convention will become clear in the following sections -- in addition to having an analogous explicit description to the case of connected groups, the representation $\tilde{\tx{St}}$ also has an analogous character to the case of connected groups, and also arises from a different natural construction applied to $\tx{St}$, namely the ``Whittaker extension'', which is of importance in the theory of automorphic forms. We note however that our choice of $\tilde{\tx{St}}$ here differs from the choice made in \cite{MS95} in the setting of finite fields.

\subsection{The Whittaker extension}

We consider the special case when $G$ is quasi-split. Let $(T,B)$ be a Borel pair. Let $\psi$ be a generic character of the unipotent radical $U$ of $B$. The group $U$ has an exhaustive separated filtration by compact open subgroups. We define the integral below as a limit over this filtration
\[ \int_U f(w_0u)\psi(u)^{-1}du, \qquad f \in \mc{C}^\infty(B \lmod G), \]
where $w_0$ is a representative in $N_G(T)$ for the longest element in the Weyl group. Note that the representative does not matter, since $f$ is left-invariant under $T$.

\begin{lem} The above integral, defined as a limit over an exhaustive separated filtration of compact open subgroups of $U$, stabilizes in finite time and defines a Whittaker functional on $\tx{Ind}_B^G\textbf{1}_B$.
\end{lem}
\begin{proof}
See \cite{CS80} or \cite[Proposition 3.2]{Sha78}.
\end{proof}

\begin{lem} \label{lem:whit-stein}
The Whittaker functional defined above kills the subspace $\mc{C}^\infty(P \lmod G)$ of $\mc{C}^\infty(B \lmod G)$ for every standard parabolic $B \subsetneq P$ and hence descends to a Whittaker functional on the Steinberg representation.
\end{lem}
\begin{proof}
This is immediate from a result of Rodier, \cite[Corollary 1.7]{CS80}. Any Whittaker functional on $\mc{C}^\infty(P \lmod G)$ will come from a Whittaker functional on the $M$-representation $\delta_P^\frac{1}{2}$, i.e. a $\psi$-$(U \cap M)$-fixed vector in the dual $\delta_P^{-\frac{1}{2}}$ that is non-zero. But $U \cap M$ is contained in the derived subgroup of $M$ and therefore $\delta_P$ is trivial on $U \cap M$, while $\psi$ is not trivial on $U \cap M$, so no such vector can exist.
\end{proof}

Assume now that $\tilde G=G \rtimes A$, where $A$ is a finite group of automorphisms of $G$ that preserves the Borel pair $(T,B)$ and the character $\psi$. Thus, $A$ preserves the pair $(B,\psi)$. In the language of \cite{KalLLCD}, the $G(F)$-conjugacy class of $(B,\psi)$ is an $A$-admissible Whittaker datum.

Then an irreducible $\psi$-generic representation $\pi$ of $G(F)$ whose isomorphism class is preserved by the action of $\tilde G(F)$ has a canonical extension $\tilde\pi$ to $\tilde G(F)$. Namely, we have $\tilde G(F)=G(F) \rtimes A$. For each $a \in A$ there exists an automorphism $\tilde\pi(a)$ of the complex vector space $V_\pi$ underlying $\pi$ satisfying $\tilde\pi(a) \circ \pi(h) = \pi(a(h))\circ \tilde\pi(a)$. By Schur's lemma, the set of all such isomorphisms forms a torsor under $\C^\times$. If $\lambda : V_\pi \to \C$ is a Whittaker functional, i.e. satisfies $\lambda(\pi(u)v)=\psi(u)\lambda(v)$, then it is also a Whittaker functional for the representation $h \mapsto \pi(a(h))$. Since the space of Whittaker functionals is $1$-dimensional, within the $\C^\times$-torsor of possible choices for $\tilde\pi(a)$, there exists a unique member that preserves $\lambda$. This choice is independent of the choice of $\lambda$, since it too is determined up to a complex scalar. For any $\tilde g = g \rtimes a$ define $\tilde\pi(\tilde g) = \pi(g) \circ \tilde\pi(a)$. One sees easily that $\tilde\pi$ is a representation of $\tilde G$ that extends $\pi$. We call $\tilde\pi$ the \emph{Whittaker extension} of $\pi$.

\begin{lem} \label{lem:whit-stein-extend}
Let $\tilde G=G \rtimes A$, where $A$ is a finite group of automorphisms of $G$ that preserves the Borel pair $(T,B)$ and the character $\psi$. The representation $\tilde{\textrm{St}}$ defined in \eqref{eq:stein} coincides with the Whittaker extension of the Steinberg representation $\tx{St}$.
\end{lem}
\begin{proof}
Choose $a \in A$. The action of $a$ on $\mc{C}^\infty(B \lmod G)$ via the natural extension of \S\ref{sub:discstein} is via $(af)(x)=f(a^{-1}(x))$. Let $\lambda$ be the canonical Whittaker functional of Lemma \ref{lem:whit-stein}. Then $ \lambda(af)$ is equal to 
\[\int_U f(a^{-1}(w_0)a^{-1}(u))\psi(u)^{-1}du = \int_U f(w_0u)\psi(a(u))^{-1}du = \int_U f(w_0u)\psi(u)^{-1}du, \]
which is again simply $\lambda(f)$. We have used here that $a$ preserves the longest element in the Weyl group, since it preserves $B$, that it preserves the Haar measure since it is of finite order, and that the functional $\lambda$ does not care about the representative $w_0$, as remarked above. 
\end{proof}

\subsection{The character of the Steinberg representation}

\begin{thm} \label{thm:charstein}
    Let $\tilde g \in \tilde G(F)$ be a regular semi-simple element. The character of the representation $\tilde{\tx{St}}$ of \eqref{eq:stein} is given by 
    \[ \sum_{(M,P)}(-1)^{\dim(A_0^{\tilde g}/A_M^{\tilde g})}|D_{\tilde G/\tilde M}(\tilde g)|^{-1/2} \delta_{\tilde P}(\tilde g)^{-1/2}, \]
where in each case the sum runs over those parabolic pairs $(M,P)$ of $G$ that are normalized by $\tilde g$, $\tilde P=N_{\tilde G}P$, $\tilde M=N_{\tilde P}M$, $D_{\tilde G/\tilde M}(\tilde g)=\det(1-\tx{Ad}(\tilde g);\tx{Lie}(G)/\tx{Lie}(M))$, and $\delta_{\tilde P}$ is the modulus character of the topological group $\tilde P(F)$.
\end{thm}
\begin{proof}
    Fix a minimal parabolic pair  $(P_0,M_0)$ of $G$. Let $S$ be the set of simple relative roots. Recall that in \S\ref{sub:cf} we defined an action of $\tilde G(F)$ on $S$. Write $A_0$ for the maximal split torus in the center of $M_0$. For each $I \subset S$ write $(P_I,M_I)$ for the associated standard parabolic pair, and $A_I$ for the maximal split torus in the center of $M_I$. We let $\tilde P_I$ be the normalizer of $P_I$ in $\tilde G$ and $\tilde M_I$ the normalizer of $M_I$ in $\tilde P_I$.

    Using that $\mathrm{\tilde St}$ lies in $\mc{R}_0(\tilde G)$, Corollary \ref{cor:chareps} shows that the value at $\tilde g$ of the character of $\mathrm{\tilde St}$ equals
    \[ (-1)^{\dim(A_0^{\tilde g})}\sum_{\substack{I \subset S \\ \tilde gI=I}}(-1)^{\dim(A_I^{\tilde g})}\textrm{ch}(i_{\tilde P_I}^{\tilde G}(\delta_{\tilde P_I}^{-1/2}))(\tilde g), \]
    where $\textrm{ch}(i_{P_I}^G(\delta_{P_I}^{-1/2}))$ is the character function of the induced representation 
    \[ i_{P_I}^G(\delta_{P_I}^{-1/2})=\mc{C}^\infty(P_I(F) \lmod G(F)) = \mc{C}^\infty(\tilde P_I(F) \lmod \tilde G(F)) = i_{\tilde P_I}^{\tilde G}(\delta_{\tilde P_I}^{-1/2}). \]
    To ease notation, we will now drop the notation $(F)$ and write $G$ in place of $G(F)$, etc. We may further replace without loss of generality $\tilde G$ by the subgroup generated by $G$ and $\tilde g$.
    A formula for the character function of $\mc{C}^\infty(\tilde P_I \lmod \tilde G)$ on $\tilde g$ is derived in \cite[Corollary 5.3.9]{Lem11} and is given
by 
\begin{equation} \label{eq:ipgtchar}
\sum_{\substack{h \in M_I \lmod G\\ h\tilde gh^{-1} \in \tilde M_I}} |D_{\tilde G/\tilde M_I}(h \tilde g h^{-1})|^{-1/2} \delta_{\tilde P_I}(h \tilde gh^{-1})^{-1/2}.	
\end{equation}

Let us briefly indicate the translation between the formula stated by Lemaire and the one stated here. In our situation $\kappa=1$, $G^\natural=G \cdot \tilde g$, $P^\natural = N_{G^\natural}(P)$, and $M^\natural=N_{P^\natural}(M)=N_{G^\natural}(P,M)$. As Lemaire states, the character is zero unless $\tilde g$ can be conjugated under $G$ into $M^\natural$. This is precisely the case when our sum is non-empty, since any element of $\tilde M$ that is $G$-conjugate to $\tilde g$ lies in $M^\natural$. Assume this is the case and let $(S,S^\natural,T,T^\natural)$ be the Cartan quadruple determined by $\tilde g$, which we recall is given by $S=G^{\tilde g,\circ}$, $S^\natural = S\cdot \tilde g$, $T=Z_G(S)$, $T^\natural = T \cdot \tilde g$. Lemaire introduces elements $g_1,\dots,g_s$ such that $g_i^{-1}T^\natural g_i$ are representatives for the $M$-conjugacy classes of those Cartan subspaces of $M^\natural$ that are $G$-conjugate to $T^\natural$. He further introduces elements $n_w \in N_G(T_i^\natural)$ representing the quotient $N_G(T_i^\natural)/N_M(T_i^\natural)$ (in loc. cit. the quotient is written in the reverse order, but this is a typo).

Each element $(g_i n_w)^{-1} \in G$ conjugates $T^\natural$ into $M^\natural$, and hence $\tilde g$ into $M^\natural$, and thus belongs to the index set $\{h \in M \lmod G|\, h \tilde g h^{-1} \in \tilde M\}$. Conversely an element $h$ of that index set conjugates $\tilde g$ into $M^\natural$, hence $T^\natural$ into $M^\natural$. Multiplying $h$ by $M$ on the left we may assume that $h T^\natural h^{-1}=T_i^\natural$, thus $g_ih$ normalizes $T^\natural$, thus $g_i^{-1}h^{-1}$ normalizes $T_i^\natural$. Multiplying further by a suitable $n_w^{-1}$ we obtain $n_w^{-1}g_i^{-1}h^{-1} \in N_M(T_i^\natural)$, i.e. $(g_in_w)^{-1} \in N_M(T_i^\natural) \cdot h$. Thus every coset $h \in M \lmod G$ satisfying $h\tilde gh^{-1} \in \tilde M$ contains an element of the form $(g_in_w)^{-1}$. If it contains another such element $(g_jn_v)^{-1}$, then using both to conjugate $T^\natural$ we see that both results are conjugate by $M$, thus $i=j$, whence in turn $n_w^{-1} \in mn_v^{-1}$ for some $m \in M$. Then $m \in N_M(T_i^\natural)$, so $n_w=n_v$.

The final thing to note is the translation between summation indices. So far we have a sum over $I \subset S$ for which $\tilde gI=I$, and a second sum over $h \in M\lmod G$ for which $h\tilde gh^{-1} \in \tilde M$. The first sum is equivalently over the set of those $G$-conjugacy classes $C$ of parabolic pairs $(P,M)$ for which $\tilde g \cdot C \cdot \tilde g^{-1}=C$. If $(P_I,M_I)$ is the unique standard representative of $C$, then $C=\{h^{-1}(P_I,M_I)h|h \in M_I \lmod G\}$. The condition $h\tilde gh^{-1} \in \tilde M_I$ is equivalent to the condition that $\tilde g$ normalizes $(P,M)=h^{-1}(P_I,M_I)h$. So the double sum we have so far becomes the single sum over the set of parabolic pairs $(\tilde P,\tilde M)$ of $\tilde G$ that are normalized by $\tilde g$.
\end{proof}

Recall that a regular semi-simple element $\tilde g \in \tilde G(F)$ is called \emph{elliptic} if the torus $G^{\tilde g,\circ}/Z(G)^{\tilde g,\circ}$ is anisotropic.

\begin{cor} \label{cor:charstein}
If $\tilde g$ is an elliptic regular semi-simple element, the character of $\tilde{\tx{St}}$ evaluated at $\tilde g$ is equal to $(-1)^{\dim(A_0^{\tilde g}/A_G^{\tilde g})}$.
\end{cor}
\begin{proof}
It is enough to show that $\tilde g$ cannot be conjugated by $G$ into $\tilde M$ for a proper standard Levi subgroup $M \subset G$ corresponding to a $\tilde g$-invariant subset $I \subset S$. Since ellipticity is invariant under conjugation, assume that $\tilde g \in \tilde M$ for such a Levi subgroup $M$. Then $Z(M)^{\tilde g,\circ} \subset G^{\tilde g,\circ}$ and we conclude that $Z(M)^{\tilde g,\circ}/Z(G)^{\tilde g,\circ}$ is anisotropic. In other words, $X^*(A_M^{\tilde g}/A_G^{\tilde g}) \otimes_\Z\Q$ is zero. The latter vector space equals $(X^*(A_M/A_G) \otimes_\Z \Q)^{\tilde g}$. But if $M$ is a proper Levi subgroup of $G$ corresponding to a $\tilde g$-invariant subset $I \subset S$, then the vector space $X^*(A_M/A_G)\otimes_\Z\Q$ is non-zero and has a basis permuted by $\tilde g$, so its subspace of $\tilde g$-fixed points is also non-zero.
\end{proof}

\subsection{Endoscopic properties and twisted Kottwitz signs}

In this subsection we consider the setting of those disconnected groups for which the conjectures in \cite{KalLLCD} are stated. Thus, let $G$ be a connected reductive quasi-split $F$-group, $(T,B,\{X_\alpha\})$ an $F$-pinning of $G$, $A$ a finite group of $F$-automorphisms of $G$ that preserve the pinning, and $\tilde G = G \rtimes A$. Let $\bar z \in Z^1(\Gamma,G/Z(G)^A)$ and let $\tilde G_{\bar z}$ be the inner form of $\tilde G$ given by twisting the rational structure of $\tilde G$ by $\bar z$.

We consider the Steinberg representation $\tx{St}$ of $G_{\bar z}(F)$ and its extension $\tilde{\tx{St}}$ to $\tilde G_{\bar z}(F)$. It is well-known that the Langlands parameter of $\tx{St}$ is the homomorphism
\[ \varphi : W_F \times \tx{SL}_2(\C) \to \hat G \rtimes \Gamma \]
whose restriction to $W_F$ is given by $w \mapsto 1 \rtimes w$ and whose restriction to $\tx{SL}_2(\C)$ is the unique $\hat G$-conjugacy class of the principal embedding. 

Recall the centralizer groups $S_\varphi = \tx{Cent}(\varphi,\hat G)$ and $\tilde S_\varphi=\tx{Cent}(\varphi,\hat G \rtimes A)$.

\begin{lem} \label{lem:sphistein}
    Specify $\varphi$ within its $\hat G$-conjugacy class so that the restriction to $\tx{SL}_2(\C)$ is given by a pinning of $\hat G$ that is stable under $\Gamma$ and $A$. Then $\tilde S_\varphi=Z(\hat G)^\Gamma \rtimes A$.
\end{lem}
\begin{proof}
    Let $\tilde s \in \tilde S_\varphi$. Write $\tilde s = s \rtimes a$. By construction $a$ centralizes $\varphi$, so $s$ must also centralize $\varphi$. But the centralizer of $\varphi|_{\tx{SL}_2(\C)}$ equals $Z(\hat G)$, while the centralizer of $\varphi|_{W_F}$ is $\hat G^\Gamma$.
\end{proof}

In particular, we have $S_\varphi=Z(\hat G)^\Gamma$. This implies via the refined local Langlands conjecture that the $L$-packet $\Pi_\varphi(G_{\bar z})$ must be a singleton, consisting just of $\tx{St}$. For a general $L$-parameter $\varphi$ the conjectures of \cite{KalLLCD} imply that $\Pi_\varphi(\tilde G_{\bar z})$ must consist of all those irreducible representations of $\tilde G_{\bar z}(F)$ whose restriction to $G_{\bar z}(F)$ intersect $\Pi_\varphi(G_{\bar z})$. The following Lemma shows that 
\[ \tx{Irr}(\tilde G_{\bar z}(F)/G_{\bar z}(F)) \to \Pi_\varphi(\tilde G_{\bar z}),\qquad \tau \mapsto \tilde{\tx{St}}\otimes\tau \]
is a bijection.

\begin{lem} \label{lem:clif}
    Let $\tilde\pi$ be a smooth irreducible representation of $\tilde G_{\bar z}(F)$ whose restriction $\pi$ to $G_{\bar z}(F)$ remains irreducible. Then $\tau \mapsto \tilde\pi\otimes\tau$ is a bijection between the set of irreducible representations of $\tilde G_{\bar z}(F)/G_{\bar z}(F)$ and the set of all smooth irreducible representations of $\tilde G_{\bar z}(F)$ whose restriction to $G_{\bar z}(F)$ contains $\pi$.
\end{lem}
\begin{proof}
    It is clear that the map $\tau \mapsto \tilde\pi\otimes\tau$ is well-defined and goes between the two sets as claimed. It remains to prove injectivity and surjectivity.

    Let $\tilde\pi'$ be a smooth irreducible representation of $\tilde G_{\bar z}(F)$ whose restriction to $G_{\bar z}(F)$ contains $\pi$. This restriction is semi-simple and of finite length by \cite[Lemma A.3]{KalSLP}. Therefore the argument of 
    \cite[Lemma A.4(1)]{KalSLP} applies and shows that the set of irreducible constituents of $\tilde\pi'|_{G_{\bar z}(F)}$ is a single orbit for the action of $\tilde G_{\bar z}(F)/G_{\bar z}(F)$ on the set of isomorphism classes of smooth irreducible representations of $G_{\bar z}(F)$. The same is true for $\tilde\pi|_{G_{\bar z}(F)}$. By assumption $\pi$ lies in both of these orbits, which means that they are the same orbit. However, the orbit for $\tilde\pi$ is singleton, since $\tilde\pi$ is an extension of $\pi$. The same is therefore true for the orbit for $\tilde\pi'$. 
    
    This means that $\tilde\pi'|_{G_{\bar z}(F)}$ is a finite direct sum of copies of $\pi$. The representation
    \begin{equation} \label{eq:tau}
        \tau := \tx{Hom}_{G_{\bar z}(F)}(\tilde\pi,\tilde\pi')
    \end{equation}
    of $\tilde G_{\bar z}(F)/G_{\bar z}(F)$ is therefore non-zero and finite-dimensional; its dimension is in fact equal to the non-zero multiplicity of $\pi$ in $\tilde\pi'|_{G_{\bar z}(F)}$. Consider the map
    \begin{equation} \label{eq:taumap}
        \tilde\pi\otimes\tau \to \tilde\pi',\qquad v \otimes f \mapsto f(v).
    \end{equation}
    By construction this map is $\tilde G_{\bar z}(F)$-equivariant. It is non-zero, as evidenced by taking $0 \neq f \in \tx{Hom}_{G_{\bar z}(F)}(\tilde\pi,\tilde\pi')$ and $v \in V_\pi$ that is not in the kernel of $f$. Since the target is irreducible, Schur's lemma shows that the map is surjective. Restricting the action of $\tilde G_{\bar z}(F)$ to $G_{\bar z}(F)$ we obtain a $G_{\bar z}(F)$-equivariant map 
    \[ \pi^{\oplus\dim(\tau)} \to \pi^{\oplus\dim(\tau)} \]
    that is still surjective, hence must also be injective. We conclude that \eqref{eq:taumap} is an isomorphism. Using again the irreducibility of $\tilde\pi'$ we now infer the irreducibility of $\tau$.

    We have thus seen that any irreducible $\tilde\pi'$ whose restriction to $G_{\bar z}(F)$ contains $\pi$ is of the form $\tilde\pi\otimes\tau$ for an irreducible $\tau$. This establishes that the map $\tau \mapsto \tilde\pi\otimes\tau$ is surjective. Its injectivity follows from \eqref{eq:tau}.
\end{proof}

For the internal structure of this packet we fix $[z] \in H^1(\mc{E}_F^\tx{rig},Z(G)^A \to G)$ that is a lift of $[\bar z] \in H^1(\Gamma,G/Z(G)^A)$. Let $A^{[z]}$ be the stabilizer of $[z]$ under the action of $A$. The projection $\tilde G \to A$ induces an isomorphism $\tilde G_{\bar z}(F)/G_{\bar z}(F) \to A^{[z]}$. In other words, we obtain the bijection
\[ \tx{Irr}(A^{[z]}) \to \Pi_\varphi(\tilde G_{\bar z}),\qquad \tau \mapsto \tilde{\tx{St}} \otimes \tau. \]
On the other hand, the conjectures of \cite{KalLLCD} stipulate a bijection
\[ \tx{Irr}(\pi_0(\tilde S_\varphi^{+,[z]}),[z]) \to \Pi_\varphi(\tilde G_{\bar z}), \]
where $\tilde S_\varphi^{[z]}$ is the preimage of $A^{[z]}$ in $\tilde S_\varphi$, $\tilde S_\varphi^{+,[z]}$ is the preimage of $\tilde S_\varphi^{[z]}$ under the pro-cover $\hat{\bar G} \to \hat G$ given by the projective limit of the dual groups of $G/Z$ for all finite $Z \subset Z(G)^A$, and we are using $[z]$ to also denote the character of $\pi_0(Z(\hat{\bar G})^+)$ corresponding to the cohomology class $[z]$ under the Tate--Nakayama isomorphism. Now Lemma \ref{lem:sphistein} implies that $\tilde S_\varphi^{+,[z]} = \pi_0(Z(\hat{\bar G})^+) \rtimes A^{[z]}$, and $[z]$ has a natural extension to this group, trivial on $A^{[z]}$. Let us call this extension $[\tilde z]$. The argument of Lemma \ref{lem:clif} applies in this situation and gives the bijection
\[ \tx{Irr}(A^{[z]}) \to \tx{Irr}(\pi_0(\tilde S_\varphi^{+,[z]}),[z]),\qquad \tau \mapsto [\tilde z]\otimes \tau. \]
Combining the two bijections we obtain the bijection
\begin{equation} \label{eq:stbijbad}
\tx{Irr}(\pi_0(\tilde S_\varphi^{+,[z]}),[z]) \to \Pi_\varphi(\tilde G_{\bar z}),\qquad [\tilde z]\otimes \tau \mapsto \tilde{\tx{St}}\otimes\tau.
\end{equation}
However, it turns out that this bijection is \emph{not} the correct bijection for the conjecture of \cite{KalLLCD}, because it does not satisfy the character identities, as we will now check. Take $a \in A^{[z]} \subset \tilde S_\varphi^{+,[z]}$ and let $\hat H=\hat G^{a,\circ}$. Then $\hat H$ is stable under $\Gamma$. Write $\mc{H}=\hat H \rtimes \Gamma$. Let $H$ be the quasi-split connected reductive $F$-group with dual group $\hat H$ and $F$-structure given by $\mc{H}$. The action of $\Gamma$ on $\hat H$ preserves the pinning of $\hat H$ induced by the pinning of $\hat G$, hence $\mc{H}={^LH}$. The parameter $\varphi$ thus factors through $^LH$ and induces the Steinberg parameter for $H$.

We consider the character identity of \cite[Conjecture 7.2.1]{KalLLCD} and restrict to regular semi-simple elliptic elements of $\tilde G_{\bar z}(F)$. Since $\tilde\rho=[\tilde z]\otimes\tau$ corresponds to $\tilde\pi=\tilde{\tx{St}}\otimes\tau$, the ``canonical extension'' discussed in that conjecture is simply $\tilde{\tx{St}}_{\bar z} \boxtimes [\tilde z]^{-1}$ and the right-hand side of the character identity evaluated at $(\tilde g,\tilde s) \in \tilde G_{\bar z}(F) \times_{\<a\>} \pi_0(\tilde S_\varphi^{+,[z]})$ becomes $e(G_{\bar z})\cdot \Theta_{\tilde{\tx{St}}}(\tilde g) \cdot [\tilde z](a)$. Recall that by construction $[\tilde z](a)=1$. On the other hand, the element $\tilde g$ transfers to an element $h \in H(F)$ that is regular elliptic, and the left-hand side of the identity becomes $\Theta_{\tx{St}}(h)$. Using Corollary \ref{cor:charstein} we see that 
\[ \Theta_{\tilde{\tx{St}}}(\tilde g) = (-1)^{\dim(A_{\bar z}^{\tilde g}/A_G^{\tilde g})},\qquad \Theta_{\tx{St}}(h) = (-1)^{\dim(A_{0,H}/A_H)}, \]
where $A_{\bar z}$ is the maximal split torus in the minimal Levi subgroup of $G_{\bar z}$, while $A_{0,H}$ is the same for the group $H$. Since $H$ is an elliptic twisted endoscopic group of $G$ we have $A_G^{\tilde g}=A_G^a \cong A_H$ and $A_0^a=A_{0,H}$, where $A_0$ is the maximal split torus in the minimal Levi subgroup of $G$. This shows that the two sides of the desired character identity differ by the sign
\[ e(G_{\bar z})\cdot \Theta_{\tilde{\tx{St}}}(\tilde g) / \Theta_{\tx{St}}(h) = (-1)^{\dim(A_0/A_0^a)-\dim(A_{\bar z}/A_{\bar z}^{\tilde g})}. \]
Let $\epsilon_{\bar z}$ and $\epsilon_0$ denote the characters of \eqref{eq:eps} for the groups $\tilde G_{\bar z}$ and $\tilde G$, respectively. Corollary \ref{cor:eps} shows that the above discrepancy equals $\epsilon_{\bar z}(a) \cdot \epsilon_0(a)$.

Thus, we conclude that in order for \eqref{eq:stbijbad} to be compatible with the expected character identities we must modify it to
\begin{equation} \label{eq:stbijgood}
\tx{Irr}(\pi_0(\tilde S_\varphi^{+,[z]}),[z]) \to \Pi_\varphi(\tilde G_{\bar z}),\qquad [\tilde z]\otimes \tau \mapsto \tilde{\tx{St}}\otimes\tau \otimes \epsilon_{\bar z} \otimes \epsilon_0.
\end{equation}
One can think of 
\[ e([G \rtimes a]_{\bar z}):=\epsilon_{\bar z}(a) \cdot \epsilon_0(a) \cdot e(G_{\bar z})\]
as an analog of the Kottwitz sign $e(G_{\bar z})$ for the coset $[G \rtimes a]_{\bar z}$. It splits into a product of the elementary sign $\epsilon_{\bar z}(a) \cdot \epsilon_0(a)$ that is multiplicative in $a$ and the cohomological sign $e(G_{\bar z})$ coming from the identity component.

\bibliographystyle{amsalpha}
\bibliography{bibliography.bib}

\end{document}